\documentclass[reqno,draft]{amsart} 
\usepackage{amsthm,amssymb,latexsym,mathrsfs}
\usepackage{amsmath}
\usepackage{indentfirst}
\usepackage{cite}
\usepackage{color}

\usepackage[top=2cm,bottom=2cm,left=2cm,right=2cm]{geometry}
\allowdisplaybreaks[1]
\numberwithin{equation}{section}
\newtheorem{theorem}{Theorem}[section]

\newtheorem{lemma}[theorem]{Lemma}
\newtheorem{remark}[theorem]{Remark}
\newtheorem{proposition}[theorem]{Proposition}

\allowdisplaybreaks

\allowdisplaybreaks

\def\esssup{\mathop{\rm ess\, sup}}
\begin{document}


\def\esssup{\mathop{\rm ess\, sup}}

\title[Global Strong Solutions to the VPB with Soft Potential in Bounded Domain]
{Global Strong Solutions to the Vlasov-Poisson-Boltzmann System with Soft Potential in a Bounded Domain}

\author[Fucai Li and Yichun Wang]{Fucai Li, Yichun Wang\\
\small Department of Mathematics, Nanjing University\\
\small Nanjing, 210093, P.R. China}

\address{Department of Mathematics, Nanjing University, Nanjing
 210093, P. R. China}
\email{fli@nju.edu.cn, ycwang@smail.nju.edu.cn}

\date{}

\begin{abstract}
Boundary effects are crucial for dynamics of dilute charged gases governed by the Vlasov-Poisson-Boltzmann (VPB) system. In this paper, we study the existence and regularity of solutions to the VPB system with soft potential in a bounded convex domain with in-flow boundary condition. We establish the existence of strong solutions in the time interval $[0,T]$ for an arbitrary given $T>0$ when the initial distribution function is near an absolute Maxwellian. Our contribution is based on a new weighted energy estimate in some $W^{1,p}$ space and $L_x^3 L_v^{1+}$ space for soft potential. By using the classical $L^2$--$L^\infty$ method and bootstrap argument, we extend the local solutions from small time scale to large time scale.
\end{abstract}


 \maketitle

\section{Introduction}
Vlasov-Boltzmann   equation is a classical model describing the dynamics and collision processes of dilute charged particles with a given field.
Subjected to a field $E$, the evolution of the distribution function  $F(t, x, v)$ of the particles   is governed by
\begin{align}\label{Vlasov-Boltzmann}
\partial_t F + v \cdot \nabla_x F + E \cdot \nabla_v F = Q(F, F).
\end{align}
For the elastic collisions, the Boltzmann collision operator $Q(\cdot, \cdot)$  in \eqref{Vlasov-Boltzmann} is given as the following non-symmetric form:
\begin{align*}
Q(F_1, F_2)=& \int_{\mathbb{R}^3\times \mathbb{S}^2} |u-v|^\gamma b_0(\theta)[F_1(u')F_2(v')-F_1(u)F_2(v)]\mathrm{d}u\mathrm{d}\omega\\
=& Q_{\mathrm{gain}}(F_1, F_2)-Q_{\mathrm{loss}}(F_1, F_2),
\end{align*}
where $\gamma$ equals to $1-\frac{4}{s}$ with inverse power $1<s<4$, and  $(u, v)$ and $(u', v')$ denote the velocities of the particles before and after the collision, which satisfy $u'=u-[(u-v)\cdot \omega]\omega$, $v'=v-[(v-u)\cdot \omega]\omega$ and $|u|^2+|v|^2 = |u'|^2+|v'|^2$.
Throughout this paper, we assume
\begin{align}\label{soft}
\gamma \in (-3, 0), \,\, 0 < b_0(\theta)\leq C \cos \theta,
\end{align}
where $\cos \theta = \omega \cdot \frac{u-v}{|u-v|}$. We remark that if (\ref{soft}) holds, then the collision kernel
in $Q(\cdot, \cdot)$  is called soft potential with Grad's angular cutoff.

It is well-known that a global Maxwellian $\mu$ satisfies $Q(\mu, \mu)=0$, where $\mu(v):=e^{-\frac{|v|^2}{2}}$.

It is important for physical applications to consider the  Vlasov-Boltzmann equation \eqref{Vlasov-Boltzmann} coupled with a  self-consistent electric field. In this case, we obtain the so-called  Vlasov-Poisson-Boltzmann (VPB) system, namely,
\begin{align}\label{vp-1}
&\partial_t F + v \cdot \nabla_x F - \nabla_x \phi \cdot \nabla_v F = Q(F, F),
\end{align}
where $-\nabla_x \phi=E$ in \eqref{Vlasov-Boltzmann} and  the self-consistent potential $\phi$ is determined by the Poisson equation
\begin{align}\label{vp-2}
-\Delta \phi (t, x)= &\int_{\mathbb{R}^3}F(t, x, v)\mathrm{d}v- \rho_0.
\end{align}
Here the background density $\rho_0$ is a constant.
Without loss of generality, by some rescaling, we assume that $\rho_0=\int_{\mathbb{R}^3} \mu(v) \mathrm{d}v$.

We consider the system \eqref{vp-1}-\eqref{vp-2} in a bounded domain $\Omega \subset \mathbb{R}^3$ with boundary $\partial \Omega$. We supply the Poisson equation \eqref{vp-2}
with  the zero Neumann boundary condition:
\begin{align}\label{vp-3}
\frac{\partial \phi}{\partial n}=  \, 0 \,\,\,\,\mathrm{on}\quad \partial\Omega
\end{align}
and the equation \eqref{vp-1} with  the so-called in-flow boundary condition:
\begin{align}\label{vp-4}
F(t, x, v)= G(t, x, v) \,\, \mathrm{for}\, \, (x, v)\in \gamma_-,
\end{align}
where $G$ is a given function, the incoming set is defined as
\begin{align*}
\gamma_-:= \{(x, v)\in \partial \Omega \times \mathbb{R}^3: n(x)\cdot v <0\},
\end{align*}
and $n(x)$ is the outward unit normal vector at a boundary point $x$.
We also define $\gamma_+$ and $\gamma_0$ as:
\begin{align*}
\mathrm{outgoing}\,\,\gamma_+:=& \{(x, v)\in \partial \Omega \times \mathbb{R}^3: n(x)\cdot v >0\},\\
\mathrm{grazing}\,\,\gamma_0:=& \{(x, v)\in \partial \Omega \times \mathbb{R}^3: n(x)\cdot v =0\}.
\end{align*}

We make a perturbation around $\mu$, i.e. $F(t, x, v)= \mu + \sqrt{\mu} f(t, x, v)$, to the problem \eqref{vp-1}-\eqref{vp-4}.
Then the  corresponding problem reads
\begin{align}
&\partial_t f+ v\cdot \nabla_x f -\nabla_x \phi_f \cdot \nabla_v f + \frac{v}{2}\cdot \nabla_x \phi_f f + Lf= \Gamma (f, f)- v\cdot \nabla_x \phi_f \sqrt{\mu},
 \label{perturb eqn}\\
&-\Delta \phi_f (t, x)= \int_{\mathbb{R}^3} f(t,x,v)\sqrt{\mu(v)}\mathrm{d}v\,\,\,\,\mathrm{in}\,\,\Omega, \,\,\,\frac{\partial \phi_f}{\partial n}=0\,\,\,\,\mathrm{on}\,\,\partial \Omega,\label{poisson}\\
& f(t,x,v)=g(t,x,v)\,\,\mathrm{for}\,\,(x,v)\in\gamma_-,\label{incoming}
\end{align}
where $g$ satisfies $G = \mu + \sqrt{\mu} g$,
and  the initial datum $f|_{t=0}=f(0, x, v):= f_0(x, v)=(F(0,x,v)-\mu)/\sqrt{\mu}$. Here we use the notation $\phi_f$  to emphasize the correspondence between $\phi$ and $f$.

The nonlinear operator $\Gamma(\cdot, \cdot)$ and linear operator $L$ in \eqref{perturb eqn}  are defined as
\begin{align*}
\Gamma(f_1, f_2)=\frac{1}{\sqrt{\mu}}Q (\sqrt{\mu}f_1, \sqrt{\mu}f_2),
\end{align*}
and
\begin{align*}
Lf=-\frac{1}{\sqrt{\mu}}(Q(\mu, \sqrt{\mu}f)+ Q(\sqrt{\mu}f, \mu)),
\end{align*}
respectively.
And
$L$ can be further split into $L= \nu-K$. The collision frequency $\nu(v)\equiv \int_{\mathbb{R}^3 \times \mathbb{S}^2}|u-v|^\gamma b_0(\theta)\mu(u)\mathrm{d}u \mathrm{d}\omega$\, for $-3 < \gamma< 0$. Moreover, there exists a constant $C_\gamma >0$ such that
\begin{align*}
\frac{1}{C_\gamma} (1+|v|^2)^{\gamma/2}\leq \nu(v)\leq C_\gamma (1+|v|^2)^{\gamma/2}.
\end{align*}
Equivalently, it holds that $\nu(v)\sim (1+|v|^2)^{\gamma/2}\sim \langle v\rangle^\gamma$.
$K$ is an integral kernel defined as
\begin{align}\label{K}
Kf=&\int_{\mathbb{R}^3 \times \mathbb{S}^2}|u-v|^\gamma b_0(\theta)\sqrt{\mu(u)}\{f(u')\sqrt{\mu(v')}+f(v')\sqrt{\mu(u')}\}\mathrm{d}u \mathrm{d}\omega \nonumber\\
&-\sqrt{\mu(v)}\int_{\mathbb{R}^3 \times \mathbb{S}^2}|u-v|^\gamma b_0(\theta)\sqrt{\mu(u)}f(u)\mathrm{d}u \mathrm{d}\omega:=K_2 f-K_1 f.
\end{align}
The nonlinear term $\Gamma(f_1,f_2)$ is equal to
\begin{align}\label{Gamma}
&\int_{\mathbb{R}^3 \times \mathbb{S}^2}|u-v|^\gamma b_0(\theta)\sqrt{\mu(u)}\{f_1(u')f_2(v')-f_1(u)f_2(v)\}\mathrm{d}u \mathrm{d}\omega := \Gamma_{\mathrm{gain}}(f_1,f_2)-\Gamma_{\mathrm{loss}}(f_1,f_2).
\end{align}

Due to the physical importance of VPB system, it has been extensively studied.
In the whole space $\mathbb{R}^3$ or the period box $\mathbb{T}^3$, there are several results for VPB system.
Guo \cite{Guo2001} constructed the global classical solutions in $\mathbb{R}^3$ to the VPB system near vacuum.
For hard-sphere scattering kernel, global solutions near the global Maxwellian $\mu$  were constructed by Guo \cite{Guo2002} in a periodic box $\mathbb{T}^3$.
For the whole space $\mathbb{R}^3$ case, global solutions near a Maxwellian were constructed in \cite{zhao2006, Hongjun}.
In \cite{SrainVMB}, a self-consistent magnetic effect was also included.
Through a series of works \cite{huijiang2013, huijiang2014, huijiang2017, zhao2012, zhao2013} by Zhao and his collaborators, they finally provided a satisfactory global well-posedness theory to the one-species VPB system near Maxwellians for the whole range ($-3<\gamma \leq 1$) of cut-off intermolecular interactions in the perturbative framework.

Meanwhile, there are some results on the  VPB system in bounded domains.
Mischler \cite{Mischler} proved global existence of renormalized solutions (introduced in  \cite{Lions} for the Boltzmann equation) to the VPB system in a smooth bounded domain of $\mathbb{R}^3$ with general boundary conditions: in-flow, bounce-back, specular and diffusion boundary conditions.
Cao, Kim and Lee \cite{Kim} obtained the global strong solutions to hard-sphere VPB system in convex domains for diffuse boundary condition, by using an $L^2$--$L^\infty$ argument developed by Guo \cite{Guo2010} for constructing classical solutions to Boltzmann equation.
Recently, in a convex domain with the Cercignani-Lampis boundary condition, \cite{jsp} established the local well-posedness of VPB system.

The $L^2$--$L^\infty$ framework has been extensively used to deal with different kinetic equations in bounded domains, such as \cite{Lee} for Boltzmann equation in a general smooth ($C^3$) convex domain, \cite{shuangqian} for Boltzmann equation with soft potential ,   { \cite{Guo2020} for Landau equation,}  and \cite{Ouyang} for the Vlasov-Poisson-Landau system.
For more information about $L^2$--$L^\infty$ method, one can refer to \cite{Hwang}, and see also  \cite{Juhi} on the applications
of the $L^2$--$L^\infty$ framework on the global Hilbert expansion for VPB system in the torus $\mathbb{T}^3$.

For the Boltzmann collision operator $Q(\cdot, \cdot)$, due to the singularities near the zero relative velocity and the lose of positive lower bound of collision frequency, the soft potential case usually behaves worse than the hard potential case. Strain and Guo \cite{Strain} provided some basic tools to deal with soft potentials in $\mathbb{T}^3$. Xiao, et al. \cite{huijiang2017} studied  the VPB system in whole space $\mathbb{R}^3$ with  soft potential for $-3<\gamma <0$.
 Recently,  Liu and  Yang  \cite{shuangqian} considered the   Boltzmann equation with soft potentials in a bounded   domain  with diffusive reflection and specular reflection boundary conditions and obtained  the global-in-time existence and time decay   of positive solutions near the equilibrium $\mu$ in the  $L^2$--$L^\infty$ framework.
  {We also mention that Duan, et al. \cite{duan2017} obtained the global existence and uniqueness of mild solutions to the Boltzmann equation with the cut-off hard and soft potentials  in the whole space or torus when the  initial datum  allows large amplitude oscillations in some sense.}

The purpose of this paper is to investigate the  VPB system in bounded domains for soft potentials, i.e. the problem \eqref{perturb eqn}-\eqref{incoming} under the assumption \eqref{soft}.  To state our result, we need some preliminaries.

First, we give some explanations on the bounded domain $\Omega$. As \cite{Kim} said, a $C^3$ domain means that for any $p \in \partial \Omega$, there exist sufficiently small $\delta_1, \delta_2 >0$, and an one-to-one and onto $C^3-$map
\begin{align*}
\eta_p:\{x_\parallel \in \mathbb{R}^2: |x_\parallel|< \delta_1\}\rightarrow & \,\partial \Omega \cap B(p, \delta_2),\\
x_\parallel = (x_{\parallel, 1}, x_{\parallel, 2})\mapsto & \,\eta_p(x_{\parallel, 1}, x_{\parallel, 2}).
\end{align*}
And a convex domain means that there exists a $C_\Omega >0$ such that for all $p\in \partial\Omega$ and $\eta_p$, and for all $x_\parallel$,
\begin{align}\label{convex}
\sum^2_{i,j=1}\zeta_i \zeta_j \partial_i \partial_j \eta_p (x_\parallel)\cdot n(x_\parallel)\leq -C_\Omega |\zeta|^2 \,\,\,\mathrm{for}\,\,\,\mathrm{all}\,\,\,\zeta \in \mathbb{R}^2.
\end{align}

Next, for given $(t, x, v)$, let $[x(s), v(s)]= [x(s; t,x,v), v(s;t, x, v)]$ be the characteristics for the Vlasov-Boltzmann equation (\ref{Vlasov-Boltzmann}):
\begin{align*}
\frac{\mathrm{d}x(s)}{\mathrm{d}s}=v(s),\,\,\,\frac{\mathrm{d}v(s)}{\mathrm{d}s}=E(s, x(s))
\end{align*}
with the initial condition: $[x(t;t,x,v), v(t;t,x,v)]=[x, v]$.

For $(t,x,v)$, we define the \emph{backward exit time} as
\begin{align*}
t_b(t, x, v):=\sup\{s \geq 0: x(\tau)\in \Omega \,\,\mathrm{for}\,\,\mathrm{all}\, \tau \in (t-s, t)\}.
\end{align*}
We also define $x_b(t,x,v):=x(t-t_b)$ and $v_b(t,x,v):= v(t-t_b)$.
Clearly, for any $x \in \Omega$, $t_b(t, x, v)$ is well-defined for all $v \in \mathbb{R}^3$. If $x \in \partial\Omega$, $t_b$ is well-defined for all $v \cdot n(x)>0$. Thanks to Lemma 2 in \cite{Kim}, if $\Omega$ is convex in the sense of \eqref{convex}, and $E \in C_x^1$ satisfies $n(x)\cdot E(t, x)=0$ for $x \in \partial\Omega$, and $t+ 1\geq t_b$, then it holds $n(x_b) \cdot v_b <0$.

As in \cite{Kim}, for a given function $f(t, x, v)$ satisfies \eqref{perturb eqn}, we define a kinetic weight as
\begin{align}\label{kinetic weight}
\alpha_{f, \varepsilon}(t,x,v):= \widetilde{\chi}\left(\frac{t-t_b+\varepsilon}{\varepsilon}\right)|n(x_b)\cdot v_b|+\left[1-\widetilde{\chi}\left(\frac{t-t_b+\varepsilon}{\varepsilon}\right)\right].
\end{align}
Here we have used a smooth function $\widetilde{\chi}: \mathbb{R}\rightarrow [0, 1]$ satisfying
\begin{align*}
&\widetilde{\chi}(\tau)=0, \, \tau \leq 0; \,\, \mathrm{and} \,\, \widetilde{\chi}(\tau)=1,\,\tau\geq 1;\\
& \frac{\mathrm{d}}{\mathrm{d}\tau}\widetilde{\chi}(\tau)\in [0,4] \,\,\mathrm{for}\,\,\mathrm{all}\,\tau \in \mathbb{R}.
\end{align*}
It is important to notice that
\begin{align*}
\alpha_{f, \varepsilon}(t,x,v)=|n(x)\cdot v|\,\,\mathrm{on}\,\gamma_-,
\end{align*}
which is the main motivation to bring in this definition.
Another crucial property of this kinetic weight is that, along the trajectory, we have
\begin{align*}
[\partial_t + v\cdot \nabla_x - \nabla_x \phi \cdot \nabla_v]\alpha_{f, \varepsilon}(t,x,v)\equiv 0.
\end{align*}
This invariance is due to the fact that $\frac{\mathrm{d}(t-t_b)}{\mathrm{d}t}\equiv 0$, $\frac{\mathrm{d}x_b}{\mathrm{d}t}\equiv 0$ and $\frac{\mathrm{d}v_b}{\mathrm{d}t}\equiv 0$. Hence, $\frac{\mathrm{d}\alpha_{f, \varepsilon}(t, x, v)}{\mathrm{d}t}\equiv 0$.

We denote for the constants $\vartheta>0$ and $\theta>0$ that
\begin{align*}
\widetilde{\vartheta}(t):= \vartheta \left(1+ \frac{1}{(1+t)^\theta}\right),
\end{align*}
and we introduce a new time-velocity weight for the soft potential, namely,
\begin{align*}
w_{\widetilde{\vartheta}}(v)= e^{\widetilde{\vartheta}|v|^2}.
\end{align*}

We need to define the boundary integration for $g(x,v)$, $x \in \partial \Omega$, that
\begin{align*}
\int_{\gamma_\pm}g \mathrm{d}\gamma=\int_{\gamma_\pm}g(x,v)|n(x)\cdot v|\mathrm{d}S_x\mathrm{d}v,
\end{align*}
where $\mathrm{d}S_x$ is the standard surface measure on $\partial \Omega$. We also denote for $h \in L^p(\gamma _\pm)$, $1\leq p < \infty$, that
\begin{align*}
\int_{\gamma_\pm}|h|^p \mathrm{d}\gamma :=|h|^p_{p, \pm}.
\end{align*}
If $h \in L^p(\gamma _\pm)$, we say $h \in L^p(\gamma)$, $1\leq p <\infty$. For $p=\infty$, we regard $L^\infty(\gamma)$ to be equal to $L^\infty (\partial\Omega \times \mathbb{R}^3)$.

Throughout this paper, $A\lesssim B$ means that there is a constant $C>0$  such that $A\leq C B$, while the sign $\lesssim_t$ corresponds to a $t-$dependent constant $C(t)>0$. In addition, we denote $E\sim F$ if it holds that $E\lesssim F \lesssim E$.

Now we are in a position to state our main result.
\begin{theorem}\label{Thm1}
Assume that a bounded, open $C^3$ domain $\Omega \subset \mathbb{R}^3$ is convex as defined in \eqref{convex}. Let $0 < \vartheta \ll 1$, $\theta \gamma +2 >0$, $\rho\equiv \frac{\theta \gamma+2}{2-\gamma}\in (0, 1)$ and $\widetilde{\vartheta}= \vartheta \left(1+ \frac{1}{(1+t)^\theta}\right)$. Assume the following compatibility conditions hold:
\begin{align*}
f_0(x, v)=g(0, x, v) \, \,\,\mathrm{on}\,\, \gamma_-,
\end{align*}
and
\begin{align}
\int_{n(x)\cdot v>0}\sqrt{\mu}f |n(x)\cdot v|\mathrm{d}v=\int_{n(x)\cdot v<0}\sqrt{\mu}g |n(x)\cdot v|\mathrm{d}v \,\,\,\mathrm{on}\,\, \partial\Omega.
\end{align}

Assume further that there exists a small constant $0 < \varepsilon_0 \ll 1$ such that for all $0 < \varepsilon \leq \varepsilon_0$, if an initial datum $F_0 = \mu + \sqrt{\mu}f_0 \geq 0$ satisfies
\begin{align*}
\|w_{2\vartheta}f_0\|_{L^\infty(\Omega \times \mathbb{R}^3)} +\|w_{\vartheta}\alpha^\beta_{f_0, \varepsilon}\nabla_{x, v}f_0\|_{L^p(\Omega \times \mathbb{R}^3)} < \varepsilon
\end{align*}
for $3<p<6-2\varpi$, $\frac{p-2}{p}<\beta< \frac{2-\varpi}{3-\varpi}$, $0 < \varpi \ll 1$, and
\begin{align*}
\|w_\vartheta \nabla_v f_0\|_{L^3(\Omega \times \mathbb{R}^3)}<\infty,
\end{align*}
and there exists a constant $\lambda_0>0$, such that the boundary datum $g$ satisfies
\begin{align*}
\sup_{0 \leq s\leq \infty} \|e^{\lambda_0 s^\rho}w_{\widetilde{\vartheta}}g(s)\|_{L^\infty(\partial\Omega \times \mathbb{R}^3)}+ \sup_{0 \leq s \leq \infty} \|w_\vartheta \nabla_{x,v}g(s)\|_{L^\infty(\partial\Omega \times \mathbb{R}^3)}<\varepsilon,
\end{align*}
then,  for each given finite time $T>0$, the problem \eqref{perturb eqn}--\eqref{incoming} has a unique solution $(f, \phi_f)$ in $[0, T]$, such that $F(t)= \mu + \sqrt{\mu}f(t)\geq 0$. Moreover, there exists $0<\lambda_m < \lambda_0$ such that
\begin{align*}
\sup_{0 \leq t \leq T}\left\{e^{\lambda_m t^\rho}\|w_{\widetilde{\vartheta}}f(t)\|_{L^\infty(\Omega \times \mathbb{R}^3)} \right\} \lesssim \|w_{\widetilde{\vartheta}}f(0)\|_{L^\infty(\Omega \times \mathbb{R}^3)} + \sup_{s\geq 0}\left\{e^{\lambda_0 s^\rho}\|w_{\widetilde{\vartheta}}g(s)\|_{L^\infty(\partial\Omega \times \mathbb{R}^3)} \right\},
\end{align*}
and for some $C>0$, $0<\vartheta_1<\vartheta$,
\begin{align*}
\|w_{\widetilde{\vartheta_1}}\alpha^\beta_{f, \varepsilon}\nabla_{x, v}f(t)\|_{L^p(\Omega \times \mathbb{R}^3)}\lesssim e^{Ct^\rho},\,\,\mathrm{for}\,\,t\in[0, T],
\end{align*}
and for $0<\delta = \delta(p, \beta)$, $0< \vartheta_2 < \vartheta_1$,
\begin{align*}
\|w_{\widetilde{\vartheta_2}}\nabla_v f(t)\|_{L^3(\Omega)L^{1+\delta}(\mathbb{R}^3)}\lesssim_t 1,\,\,\mathrm{for}\,\,t\in[0, T].
\end{align*}
Furthermore, if $(f_1, \phi_{f_1})$ and $(f_2, \phi_{f_2})$ are both solutions to the system \eqref{perturb eqn}--\eqref{incoming} with the same initial and boundary conditions, then
\begin{align*}
\|w_{\widetilde{\vartheta_2}}[f_1(t)-f_2(t)]\|_{L^{1+\delta}(\Omega \times \mathbb{R}^3)}\equiv 0, \,\,\mathrm{for}\,\,t\in [0, T].
\end{align*}

\end{theorem}

Let us sketch our ideas and novelties in the proof of Theorem \ref{Thm1}.

Contrasted with the work \cite{Kim}, where the hard sphere potential $(\gamma=1)$ for Boltzmann collision operator $Q(F,F)$ was considered, we consider the soft potential $(-3<\gamma <0)$ case, which causes the singularity as relative velocity tends to zero.

To treat the singularity of $K=K_2-K_1$ [see \eqref{K}], we introduce a smooth cut-off function $0 \leq \chi \leq 1$ such that $\chi(s)= 1$ if $s \geq 2\varepsilon$; $\chi(s)=0$ if $s\leq \varepsilon$. Then we have
\begin{align}\label{K chi}
K_1= K_1^\chi+K^{1-\chi}_1,\,\,\,K_2= K_2^\chi+K^{1-\chi}_2,
\end{align}
where
\begin{align}\label{k2}
K_1^\chi f=& \int_{\mathbb{R}^3 \times \mathbb{S}^2}|u-v|^\gamma \chi(|u-v|) b_0(\theta)\sqrt{\mu(u)}\sqrt{\mu(v)}f(u)\mathrm{d}u \mathrm{d}\omega,\nonumber\\
K_2^\chi f=& \int_{\mathbb{R}^3 \times \mathbb{S}^2}|u-v|^\gamma \chi(|u-v|) b_0(\theta)\sqrt{\mu(u)}\{f(u')\sqrt{\mu(v')}+f(v')\sqrt{\mu(u')}\}\mathrm{d}u \mathrm{d}\omega \nonumber\\
:= &\int_{\mathbb{R}^3}\mathbf{k}_2^\chi(v, u)f(u)\mathrm{d}u.
\end{align}
In this paper, we regard the kernel $\mathbf{k}_2^\chi(v, u)$ corresponding to $\sqrt{\mu(\cdot)}=e^{-\frac{1}{4}|\cdot|^2}$ as ``standard''.

$\bullet$ We emphasize that the structures of the time-velocity weight will be more subtle in the soft potential case than in the hard sphere case while proving a priori estimates and global existence of solutions to the system \eqref{perturb eqn}--\eqref{incoming}.

There are two main reasons to introduce the velocity-time weight function $w_{\widetilde{\vartheta}}(v)=e^{\vartheta(1+1/(1+t)^\theta)|v|^2}$. We point out that the following heuristic calculations are crucial for Sections 3 and 5.

(1). It is clear that $w_{\vartheta}\leq w_{\widetilde{\vartheta}}\leq w_{2\vartheta}$. When $0< \vartheta \ll 1$, it holds $w_{\vartheta}(v)\leq e^{\frac{1}{4}|v|^2}$, i.e. $\sqrt{\mu(v)}\leq w_{\vartheta}^{-1}(v)$.
We have
\begin{align*}
&\Gamma_{\mathrm{gain}}^\chi (f, g)+ \Gamma_{\mathrm{gain}}^\chi (g, f)\\
=& \int_{\mathbb{R}^3 \times \mathbb{S}^2} |u-v|^\gamma \chi(|u-v|)b_0(\theta)\sqrt{\mu(u)}(f(u')g(v')+ f(v')g(u'))\mathrm{d}u \mathrm{d}\omega\\
\lesssim & \|w_{\widetilde{\vartheta}}f\|_\infty \int_{\mathbb{R}^3 \times \mathbb{S}^2} |u-v|^\gamma \chi(|u-v|)b_0(\theta)w_{\vartheta}^{-1}(u)( w_{\vartheta}^{-1}(u')g(v')+ w_{\vartheta}^{-1}(v')g(u'))\mathrm{d}u \mathrm{d}\omega\\
:= & \|w_{\widetilde{\vartheta}}f\|_\infty \int_{\mathbb{R}^3} \mathbf{k}_{2, \vartheta}^\chi (v, u)g(u)\mathrm{d}u.
\end{align*}
For the standard $\mathbf{k}_2^\chi(v,u)$, it is clear that
\begin{align*}
\int_{\mathbb{R}^3}\mathbf{k}_2^\chi(v,u)g(u)\mathrm{d}u=& \int_{\mathbb{R}^3 \times \mathbb{S}^2} |u-v|^\gamma \chi(|u-v|)b_0(\theta)w_{1/4}^{-1}(u)( w_{1/4}^{-1}(u')g(v')+ w_{1/4}^{-1}(v')g(u'))\mathrm{d}u \mathrm{d}\omega\\
:= & \int_{\mathbb{R}^3}\mathbf{k}_{2, 1/4}^\chi(v,u)g(u)\mathrm{d}u.
\end{align*}
Thus $\mathbf{k}_{2, 1/4}^\chi(v,u)=\mathbf{k}_{2}^\chi(v,u)\lesssim \frac{\exp\left( -\frac{s_0}{8}|u-v|^2- \frac{s_0}{8}\frac{(|v|^2-|u|^2)^2}{|v-u|^2}\right)}{|v-u|(1+|v|+|u|)^{1-\gamma}}$ by Lemma \ref{k2x}. Through the same procedure to estimate $\mathbf{k}_{2, \vartheta}^\chi (v, u)$, we have $\mathbf{k}_{2, \vartheta}^\chi (v, u)\lesssim \frac{\exp\left( -\rho|u-v|^2- \rho\frac{(|v|^2-|u|^2)^2}{|v-u|^2}\right)}{|v-u|(1+|v|+|u|)^{1-\gamma}}$, where $\rho:\vartheta=\frac{s_0}{8}:\frac{1}{4}$. Hence
\begin{align*}
\mathbf{k}_{2, \frac{4\vartheta}{s_0}}^\chi (v, u)\lesssim \frac{\exp\left( -2\vartheta|u-v|^2- 2\vartheta\frac{(|v|^2-|u|^2)^2}{|v-u|^2}\right)}{|v-u|(1+|v|+|u|)^{1-\gamma}}
\end{align*}
by $2\vartheta : \frac{4 \vartheta}{s_0}=\frac{s_0}{8}:\frac{1}{4}$. And so as $\widetilde{\vartheta}$ replaces $\vartheta$ in the previous inequality.

(2). Multiplying $w_{\widetilde{\vartheta}}(v)$ to both sides of (\ref{perturb eqn}), we have the weighted equation:
\begin{align*}
\{\partial_t + v \cdot \nabla_x - \nabla_x \phi \cdot \nabla_v + \widetilde{\nu}\}(w_{\widetilde{\vartheta}}f)= w_{\widetilde{\vartheta}}(Kf+ \Gamma(f, f)- v \cdot \nabla_x \phi \sqrt{\mu}).
\end{align*}
Here $\widetilde{\nu}=\nu+\frac{v}{2}\cdot \nabla \phi+ \frac{\nabla \phi \cdot \nabla_v w_{\widetilde{\vartheta}}}{w_{\widetilde{\vartheta}}}-\frac{\partial_t w_{\widetilde{\vartheta}}}{w_{\widetilde{\vartheta}}}$.
By using Young's inequality, we have
\begin{align*}
\widetilde{\nu}\gtrsim &\langle v\rangle^\gamma - \left(\frac{1}{2}+4\vartheta\right)\langle v\rangle|\nabla \phi|+ \vartheta \theta \langle v\rangle^2 \frac{1}{(1+t)^{\theta+1}}\\
\gtrsim & \langle v\rangle^\gamma (1-(1+t)^{(1+\theta)(1-\gamma)}|\nabla \phi|)+(\vartheta\theta-|\nabla \phi|)\langle v\rangle^2 \frac{1}{(1+t)^{\theta+1}}.
\end{align*}
Note that $(1+t)^{(1+\theta)(1-\gamma)}|\nabla \phi|\lesssim e^{\Lambda_1 t^\rho}\|\nabla \phi(t)\|_\infty\leq \delta_1 \ll 1$. If $\vartheta\theta > \delta_1$, then we have
\begin{align*}
\widetilde{\nu}\gtrsim \frac{1}{2}\langle v\rangle^\gamma + C \langle v\rangle^2 \frac{1}{(1+t)^{\theta+1}}\gtrsim \nu >0,
\end{align*}
where $C>0$ is a small constant. The positivity of $\widetilde{\nu}$ is crucial for Propositions \ref{prop 3}, \ref{prop 4} and \ref{prop 5}, which guarantees
\begin{align*}
\int_0^t \int_{\Omega \times \mathbb{R}^3}\widetilde{\nu}|w_{\widetilde{\vartheta}}f|^p \mathrm{d}x\mathrm{d}v\mathrm{d}s=\int_0^t \left\|\widetilde{\nu}^{\frac{1}{p}}w_{\widetilde{\vartheta}}f \right\|_p^p \geq 0,
\end{align*}
and
\begin{align*}
e^{\int_t^s \widetilde{\nu}(\tau)\mathrm{d}\tau}\leq & 1, \,\,s\leq t,\\
0 \leq \int_0^t e^{\frac{1}{2}\int_t^s \widetilde{\nu}(\tau)\mathrm{d}\tau}\widetilde{\nu}(s, x(s), v(s))\mathrm{d}s = & 2 \int_0^t \mathrm{d}\left(e^{\frac{1}{2}\int_t^s \widetilde{\nu}(\tau)\mathrm{d}\tau}\right)\leq 2.
\end{align*}
On the other hand, if we use $w_{\vartheta}$ instead of $w_{\widetilde{\vartheta}}$, the corresponding $\widetilde{\nu}= \nu + (1/2+ 2\vartheta)v \cdot \nabla \phi$ does not have a lower bound as $|v|$ is large. so we cannot obtain these propositions without $w_{\widetilde{\vartheta}}$.

Moreover, by using Young's inequality again, as $\theta\gamma+2>0$, we have
\begin{align*}
\widetilde{\nu}\gtrsim \frac{1}{2}\langle v\rangle^\gamma + C \langle v\rangle^2 \frac{1}{(1+t)^{\theta+1}}\gtrsim (1+t)^{\frac{(\theta+1)\gamma}{2-\gamma}}\equiv (1+t)^{\rho-1}, \,\, \rho= \frac{\theta \gamma +2}{2-\gamma}\in (0,1).
\end{align*}
Thus
\begin{align*}
e^{\int_t^s \widetilde{\nu}(\tau)\mathrm{d}\tau}\leq e^{\frac{C}{\rho}\int_t^s \mathrm{d}(1+\tau)^\rho}:=e^{\lambda (1+s)^\rho-\lambda (1+t)^\rho}\sim e^{\lambda s^\rho-\lambda t^\rho}.
\end{align*}
It overcomes the difficulty that the collision frequency $\nu$ does not have positive lower bound and provides the sub-exponential decay of the solution.

$\bullet$ The singularity appears in $(1-\chi(|u-v|))|u-v|^\gamma$, since $-3<\gamma<0$ in soft potential case. We shall deal with the tricky terms
\begin{align}\label{1-X}
w_{\widetilde{\vartheta}}\alpha^\beta_{f, \varepsilon}(\Gamma^{1-\chi}_{\mathrm{gain}}(\partial f, f)+ \Gamma^{1-\chi}_{\mathrm{gain}}(f, \partial f)+\Gamma^{1-\chi}_{\mathrm{loss}}(\partial f, f)).
\end{align}
Considering the definition of $\alpha^\beta_{f, \varepsilon}(t,x,v)$ [see \eqref{kinetic weight}], we point out that
$\alpha^\beta_{f, \varepsilon}(v)\lesssim \alpha^\beta_{f, \varepsilon}(w)+O(\varepsilon)$, for $w=u, u', v'$ as $\{t-t_b \geq 0\}\cup \{t-t_b \leq -\varepsilon\}$. While it also holds $\int_{x,v}\mathbf{1}_{\{-\varepsilon<t-t_b<0\}}\mathrm{d}x\mathrm{d}v=O(\varepsilon)$. For the previous both $O(\varepsilon)$ terms, with the condition $\|\nabla_{x,v}f\|_\infty < \infty$, we have the term $O(\varepsilon)\|\nabla_{x,v}f\|_\infty$ as a part of the bound of \eqref{1-X}. Taking $\varepsilon$ sufficiently small, then $O(\varepsilon)\|\nabla_{x,v}f\|_\infty$ is absorbed into the initial data. For details see the proofs of Proposition 3.1. In Section 4, we can verify the condition $\|\nabla_{x,v}f\|_\infty<\infty$ by the separability of $L^1$ and the weakly (star) lower semicontinuity of $L^\infty$ (Through the standard proofs of the latter, see for example \cite{Lieb}, the weak star convergence condition is enough).

However, if we want to obtain the bound of \eqref{1-X} with more direct ways, for example, H\"{o}lder's inequality, the singularity of soft potential behaves wild. For instance, to remove the singularity of $\alpha^{-\beta q}$, we have
\begin{align*}
&\int_{\mathbb{R}^3 \times \mathbb{S}^2}|u-v|^\gamma (1-\chi (|u-v|)) w_{\widetilde{\vartheta}}^{-1}(u')|w_{\widetilde{\vartheta}}\partial f(v')|\mathrm{d}u\mathrm{d}\omega\\
\lesssim &\left( \int_{\mathbb{R}^3 \times \mathbb{S}^2}|u-v|^{\gamma- \frac{3p}{q}} |(u-v)\cdot \omega|^{-\frac{2 p}{q}}(1-\chi (|u-v|))e^{Cp |u-v|^2}
w_{\widetilde{\vartheta}}^{-1}(u')|w_{\widetilde{\vartheta}}\alpha^\beta_{f, \varepsilon}\partial f(v')|^p\mathrm{d}u\mathrm{d}\omega \right)^{\frac{1}{p}}\\
& \times \left(\int_{\mathbb{R}^3 \times \mathbb{S}^2}|u-v|^{\gamma+3} (1-\chi)|(u-v)\cdot \omega|^2 \frac{e^{-Cq|v-u|^2}}{\alpha^{\beta q}_{f, \varepsilon}(v')}\mathrm{d}u\mathrm{d}\omega \right)^{\frac{1}{q}}
:= B \times A.
\end{align*}
Denote that $V\equiv u-v$, then $V_\parallel =[(u-v)\cdot \omega]\omega$, $|V_\parallel|=|(u-v)\cdot \omega|$, $v'=v+V_\parallel$ and  $\mathrm{d}v'=\mathrm{d}V_\parallel$. Thus we have
\begin{align*}
\mathrm{d}u\mathrm{d}\omega = \frac{2 \mathrm{d} V_{\parallel}\mathrm{d} V_{\perp}}{|V_\parallel|^2}= \frac{2 \mathrm{d} V_\perp}{|V_\parallel|^2}\mathrm{d} v'.
\end{align*}
Using Lemma \ref{alpha beta}, the second term $A^q$ equals to
\begin{align*}
\int |V|^{\gamma +3}\frac{e^{-Cq|v'-u'|^2}}{\alpha^{\beta q}_{f, \varepsilon}(v')} \mathrm{d}v'\int_{|V|\leq 2 \varepsilon}|V_\parallel|^2 \frac{2 \mathrm{d} V_\perp}{|V_\parallel|^2} \lesssim 1.
\end{align*}
For the first term $B$, it holds (see, for example \cite{Glassey}) that
\begin{align*}
(u-v)\cdot \omega = -(u'-v')\cdot \omega.
\end{align*}
Thus we have
\begin{align*}
\int_{\mathbb{R}^3}|B|^p \mathrm{d}v= & \int_{\mathbb{R}^3 \times \mathbb{R}^3\times \mathbb{S}^2}|u-v|^{\gamma- \frac{3p}{q}} |(u-v)\cdot \omega|^{-\frac{2 p}{q}}(1-\chi )e^{Cp |u-v|^2}
w_{\widetilde{\vartheta}}^{-1}(u')|w_{\widetilde{\vartheta}}\alpha^\beta_{f, \varepsilon}\partial f(v')|^p\mathrm{d}u \mathrm{d} v\mathrm{d}\omega \\
\lesssim & \int_{\mathbb{R}^3 \times \mathbb{R}^3\times \mathbb{S}^2}|u'-v'|^{\gamma- \frac{3p}{q}} |(u'-v')\cdot \omega|^{-\frac{2 p}{q}}
w_{\widetilde{\vartheta}}^{-1}(u')|w_{\widetilde{\vartheta}}\alpha^\beta_{f, \varepsilon}\partial f(v')|^p\mathrm{d}u' \mathrm{d} v'\mathrm{d}\omega\\
= & \int_{\mathbb{R}^3 \times \mathbb{S}^2}|u-v|^{\gamma- \frac{3p}{q}} |(u-v)\cdot \omega|^{-\frac{2 p}{q}}
w_{\widetilde{\vartheta}}^{-1}(u)\mathrm{d}u \mathrm{d}\omega \int_{\mathbb{R}^3}|w_{\widetilde{\vartheta}}\alpha^\beta_{f, \varepsilon}\partial f(v)|^p \mathrm{d}v.
\end{align*}
We also denote that $(u \cdot \omega)\omega = u_\parallel$, $u=u_\parallel+u_\perp$, $(v \cdot \omega)\omega = v_\parallel$, $v=v_\parallel+v_\perp$, then $|u \cdot \omega|= |u_\parallel|$, $|v \cdot \omega|= |v_\parallel|$, $|u|^2=|u_\parallel|^2+|u_\perp|^2$,
and
\begin{align*}
\mathrm{d}u\mathrm{d}\omega= \frac{2 \mathrm{d} u_\parallel \mathrm{d}u_\perp}{|u_\parallel|^2}= 2 \mathrm{d}|u_\parallel|\mathrm{d}u_\perp,
\end{align*}
and
\begin{align*}
|u-v|\geq |(u-v)\cdot \omega|\geq \big||v \cdot \omega|-|u \cdot \omega|\big|= \big||v_\parallel|-|u_\parallel|\big|.
\end{align*}
Hence, it holds
\begin{align*}
&\int_{\mathbb{R}^3 \times \mathbb{S}^2}|u-v|^{\gamma- \frac{3p}{q}} |(u-v)\cdot \omega|^{-\frac{2 p}{q}}
w_{\widetilde{\vartheta}}^{-1}(u)\mathrm{d}u \mathrm{d}\omega\\
\lesssim & \int \left||v_\parallel|-|u_\parallel|\right|^{\gamma-\frac{5 p}{q}}e^{-\widetilde{\vartheta}|u_\parallel|^2}\mathrm{d}|u_\parallel| \int e^{-\widetilde{\vartheta}|u_\perp|^2}\mathrm{d}u_\perp.
\end{align*}
But $\gamma-\frac{5p}{q}<-1$ $(\mathrm{as}\,\,\, 3<p<6-2\varpi, \frac{1}{p}+\frac{1}{q}=1)$, the above integration is not finite.
It means that if we try to eliminate the singularity of term $A$, then the singularity of term $B$ will become too strong to dominate, vice versa.

$\bullet$ We see from Section 4 that if $\|w_{\widetilde{\vartheta}}f_0\|_\infty+ \sup_{s \geq 0}\|e^{\lambda_0 s^\rho}w_{\widetilde{\vartheta}}g(s)\|_\infty\leq \delta^* M$, $\delta\ll 1$, then there exists $\widehat{T} = (1-\delta^*)/C$, where $C$ is independent on $\delta^*$, $M$ and $\rho$, such that the solution $w_{\widetilde{\vartheta}}f$ exists on $[0, \widehat{T}]$. Note that $\widehat{T}$ is independent on the value of $M$. Furthermore, $\widehat{T}$ is the largest time that $\|w_{\widetilde{\vartheta}}f(t)\|_\infty+ \sup_{s\geq 0}\|e^{\lambda_0 s^\rho}w_{\widetilde{\vartheta}}g(s)\|_\infty$ attains $M$. By using the classical $L^2-L^\infty$ method, we prove that $\widehat{T}$ is actually not the largest time, so we can extend the solution to $\widehat{T}+$ through a bootstrap argument. In fact, for any given finite $T>0$, we can choose $M_T$ sufficiently small and finally extend the solution to $[0, T]$.

$\bullet$ There are also many other small and not so essential differences, for example, we use the Gamma function for controlling the integration of $e^{-\tau^\rho}$ over $\tau \in [0,t]$ (see the proof of Proposition \ref{prop 4}). We will not list them all here.

This paper is organized as follows. In Section 2, we list some basic lemmas about the collision operators with soft potential and the kinetic weight $\alpha_{f,\varepsilon}$.
In Section 3, we establish some important \emph{a} \emph{priori} estimates for the solutions.
In Section 4, we build up the local existence and uniqueness of strong solutions to VPB system.
Section 5 is devoted to extending the solutions to any $T>0$ by using $L^2-L^\infty$ method and the bootstrap argument. Finally, we obtain the proof of Theorem \ref{Thm1}.

\section{Preliminary}

\begin{lemma}
For any multi-index $\beta$ and any $0<s_1<s_2<1$,
\begin{equation}\label{k2x}
|\partial_v^\beta \mathbf{k}_2^\chi (v,u)|\leq C \frac{\exp\left( -\frac{s_2}{8}|u-v|^2- \frac{s_1}{8}\frac{(|v|^2-|u|^2)^2}{|v-u|^2}\right)}{|v-u|(1+|v|+|u|)^{1-\gamma}}.
\end{equation}
Here $C>0$ will depend on $s_1$, $s_2$ and $\beta$.
\end{lemma}
One can refer to Lemma 1 in \cite{Strain} for the proofs. In this paper, we shall use the forms
\begin{align*}
| \mathbf{k}_2^\chi (v,u)|\lesssim &\frac{\exp\left( -\frac{s_0}{8}|u-v|^2- \frac{s_0}{8}\frac{(|v|^2-|u|^2)^2}{|v-u|^2}\right)}{|v-u|(1+|v|+|u|)^{1-\gamma}},
\end{align*}
and
\begin{align*}
|\nabla_v \mathbf{k}_2^\chi (v,u)|\lesssim &\frac{\exp\left( -\frac{s_0}{8}|u-v|^2- \frac{s_0}{8}\frac{(|v|^2-|u|^2)^2}{|v-u|^2}\right)}{|v-u|(1+|v|+|u|)^{1-\gamma}},
\end{align*}
for $s_0=\min\{s_1, s_2\}$.

\begin{lemma}\label{kwh}
Assume $-3 < \gamma <0$, $\vartheta>0$, $\widetilde{\vartheta}(t)=\vartheta(1+(1+t)^{-\theta})$ and $\theta\geq 0$. It holds that
\begin{align*}
w_{\widetilde{\vartheta}}K^{1-\chi}\left( \frac{|h|}{w_{\widetilde{\vartheta}}}\right) \leq C \mu^{\frac{1}{4}}(v)\varepsilon^{\gamma+3}\|h\|_\infty,
\end{align*}
and
\begin{align*}
w_{\widetilde{\vartheta}}\int_{\mathbb{R}^3}\mathbf{k}^\chi (v, u)\frac{e^{\overline{\epsilon}|v-u|^2}|h(u)|}{w_{\widetilde{\vartheta}}(u)}\mathrm{d}u \leq C \langle v\rangle^{\gamma-2}\|h\|_\infty,
\end{align*}
where $\overline{\epsilon}>0$ is sufficiently small and $\langle v\rangle=\sqrt{1+|v|^2}$. $\chi(s)=1$ if $s\geq 2\varepsilon$; $\chi(s)=0$ if $s \leq \varepsilon$.
\end{lemma}
The proof is similar to Lemma 2.2 in \cite{shuangqian}, we omit the details.

\begin{lemma}\label{gamma}
It holds that
\begin{align*}
\|\nu^{-1}w_{\widetilde{\vartheta}}\Gamma(f_1, f_2)\|_\infty \leq C \|w_{\widetilde{\vartheta}}f_1\|_\infty \|w_{\widetilde{\vartheta}}f_2\|_\infty.
\end{align*}
\end{lemma}
The proof is similar to that of Lemma 5 in \cite{Guo2010} for hard potential case. For sake of being self-contained, we briefly sketch it here.
\begin{proof}
Consider the expression of $\Gamma(f_1,f_2)$ in \eqref{Gamma}. For the second term $\Gamma_{\mathrm{loss}}(f_1, f_2)$, we have
\begin{align*}
w_{\widetilde{\vartheta}}(v)\Gamma_{\mathrm{loss}}(f_1, f_2)(v)\lesssim \int |u-v|^\gamma \sqrt{\mu(u)}\mathrm{d}u \big\|w_{\widetilde{\vartheta}}f_1\big\|_\infty \big\|w_{\widetilde{\vartheta}}f_2\big\|_\infty\lesssim \langle v\rangle^\gamma \big\|w_{\widetilde{\vartheta}}f_1\big\|_\infty \big\|w_{\widetilde{\vartheta}}f_2\big\|_\infty.
\end{align*}
For the first term $\Gamma_{\mathrm{gain}}$ in \eqref{Gamma}, by $|u|^2+|v|^2=|u'|^2+|v'|^2$ and $w_{\widetilde{\vartheta}}(\cdot)=e^{\widetilde{\vartheta}|\cdot|^2}$, we have $w_{\widetilde{\vartheta}}(v)\leq w_{\widetilde{\vartheta}}(u')w_{\widetilde{\vartheta}}(v')$ and
\begin{align*}
w_{\widetilde{\vartheta}}(v)\Gamma_{\mathrm{gain}}(f_1, f_2)(v)\leq \int |u-v|^\gamma \sqrt{\mu(u)}w_{\widetilde{\vartheta}}f_1(u')w_{\widetilde{\vartheta}}f_2(v')\mathrm{d}u \lesssim \langle v\rangle^\gamma \big\|w_{\widetilde{\vartheta}}f_1\big\|_\infty \big\|w_{\widetilde{\vartheta}}f_2\big\|_\infty.
\end{align*}
This completes the proof since $\nu(v)\sim \langle v\rangle^\gamma$.
\end{proof}
\begin{lemma}[\!\!\cite{Kim}, Lemma 10]\label{lemma10}
Assume that $E(t, x)\in C_x^1$ is given, and
\begin{align*}
\sup_{t\geq 0}e^{\Lambda_2 t}\|\nabla_x E(t)\|_\infty \leq \delta_2 \ll 1,
\end{align*}
then, for $\Lambda_2 +\delta_2+\varepsilon \leq 1$, there exists a $C>0$ such that
\begin{align*}
|\nabla_v x(s; t, x, v)|\leq C e^{C \delta_2 (\Lambda_2)^{-2}}|t-s|,
\end{align*}
for all $s$ satisfying $\max\{t-t_b(t, x, v),-\varepsilon\}\leq s \leq t$.
\end{lemma}

\begin{lemma}[\!\!\cite{Kim}, Proposition 2]\label{alpha beta}
Assuming that $E(t, x)\in C_x^1$ is given, and
\begin{align*}
& n(x)\cdot E(t, x)=0 \,\,\mathrm{for}\,\,x\in \partial \Omega, \,t\geq 0,\\
&\sup_{t\geq 0}e^{\Lambda_1 t}\|\nabla_x E(t)\|_\infty \leq \delta_1 \ll 1, \\
&
\sup_{t\geq 0}e^{\Lambda_2 t}\|\nabla_x E(t)\|_\infty \leq \delta_2 \ll 1.
\end{align*}
Then for all $0<\sigma<1$ and $N>1$ and for all $s\geq 0$, $x \in \overline{\Omega}$, it holds
\begin{align*}
\int_{|u|\leq N}\frac{\mathrm{d}u}{\alpha^\sigma_{f, \varepsilon}(s, x, u)}\lesssim 1,
\end{align*}
and for any $0<\kappa \leq 2$, it holds
\begin{align*}
\int_{|u|\geq N}\frac{e^{-C|v-u|^2}}{|v-u|^{2-\kappa}}\frac{1}{\alpha^\sigma_{f, \varepsilon}(s, x, u)}\mathrm{d}u \lesssim 1.
\end{align*}
\end{lemma}
Due to the above strict conditions, we shall only apply Lemma \ref{alpha beta} on $[0, T]$ with $T$ a finite time.

\begin{lemma}[\!\!\cite{Kim}, Lemma 3]\label{lemma 3}
For fixed t, a map
\begin{align*}
(x, v)\in \Omega \times \mathbb{R}^3 \mapsto (t-t_b(t, x, v), x_b(t, x, v), v_b(t, x, v))\in \mathbb{R}\times \gamma_-
\end{align*}
is one-to-one and
\begin{align*}
&\left|\det\left(\frac{\partial(t-t_b(t, x, v), x_b(t, x, v), v_b(t, x, v))}{\partial(x, v)}\right)\right| = \frac{1}{|n(x_b(t, x, v))\cdot v_b(t, x, v)|}.
\end{align*}

\end{lemma}

\begin{lemma}[\!\!\cite{Kim}, Lemmas 11 and 12]
For any $0<\delta<1$, if $(f, \phi_f)$ solves the problem \eqref{perturb eqn}--\eqref{incoming}, then for all $t \geq 0$,
\begin{align*}
\|\phi_f (t)\|_{C^{1, 1-\delta}}(\overline{\Omega})\leq C_{\Omega}\|w_{\widetilde{\vartheta}}f(t)\|_\infty.
\end{align*}
Assume that \eqref{p beta}(about $p$) and \eqref{hessian phi} holds, then for all $t \geq 0$, there exists a $C_1\geq 0$ such that
\begin{align*}
\|\phi_f(t)\|_{C^{2, 1-\frac{3}{p}}(\overline{\Omega})}\leq (C_1)^{1/p}\{\|f(t)\|_p+\|\alpha^\beta_{f, \varepsilon}\nabla_x f(t)\|_p\}.
\end{align*}
\end{lemma}

\section{Some a priori Estimates}
\begin{proposition}\label{prop 3}
Let us choose $0<\vartheta \ll 1$ and
\begin{equation}\label{p beta}
\frac{p-2}{p}<\beta< \frac{2-\varpi}{3-\varpi},\, \mathrm{for} \,\,\, 3<p<6-2\varpi, \,\, \forall \,0<\varpi\ll 1.
\end{equation}
Assume that $f$ solves the problem \eqref{perturb eqn}--\eqref{incoming}, and
\begin{align}
& w_{\widetilde{\vartheta}}f, w_{\widetilde{\vartheta}} \alpha
_{f,\varepsilon}^\beta \nabla_{x,v}f \in L^p((0,T)\times \Omega\times \mathbb{R}^3)\cap L^p((0,T)\times \gamma_+),\label{3.-1}\\
&\nabla_{x,v}f\in L^p((0, T); L^\infty (\Omega \times \mathbb{R}^3)),\label{3.3}\\
& w_{\widetilde{\vartheta}}g \in L^p((0,T) \times \gamma_-), w_{4\vartheta}\nabla_{t,x,v}g \in L^p((0,T); L^\infty (\gamma_-)), \label{3.4}\\
& \sup_{0\leq t \leq T} \|w_{\frac{4\vartheta}{s_0}}f(t)\|_\infty \ll 1, \label{3.-4}\\
& \sup_{0\leq t \leq T} e^{\Lambda_1 t^\rho}\|\nabla_x \phi (t)\|_\infty < \delta_1, \label{3.-5}\\
&\sup_{0\leq t \leq T}  e^{\Lambda_2 t}\|\nabla_x^2 \phi (t)\|_\infty < \delta_2,\label{hessian phi}
\end{align}
with
\begin{equation*}
0<\frac{\delta_1}{(\Lambda_1)^\frac{1}{\rho}} \ll 1, \,\,0< \frac{\delta_2}{(\Lambda_2)^2}\ll 1,
\end{equation*}
where $0<\rho<1,  \widetilde{\vartheta}=\vartheta \left( 1+\frac{1}{(1+t)^\theta}\right)$, satisfying $\vartheta\theta>\delta_1$.  The above $s_0=\min \{s_1, s_2\}$ is from the estimate of $\mathbf{k}_2^\chi$ in \eqref{k2x}.

Then there exists a $C_p > 0$ such that
\begin{align}
&\|w_{\widetilde{\vartheta}}f(t)\|_p^p + \|w_{\widetilde{\vartheta}}\alpha_{f,\varepsilon}^\beta \partial f(t)\|_p^p \nonumber\\
&\leq  C_p e^{C_p (1+ \sup_{0\leq s \leq t} \|\nabla^2 \phi (s)\|_\infty)t} \{\|w_{\widetilde{\vartheta}}f(0)\|_p^p + \|w_{\widetilde{\vartheta}}\alpha_{f,\varepsilon}^\beta \partial f(0)\|_p^p + \|w_{\widetilde{\vartheta}}g \|^p_{L^p((0,T)\times \gamma_-)}+\int_0^T \|w_{4\vartheta}\partial g\|_\infty^p\},
\end{align}
where $\partial = \nabla_{x, v}$, $\|\cdot \|_p = \|\cdot \|_{L^p(\Omega \times \mathbb{R}^3)}$ and $\|\cdot \|_\infty = \|\cdot \|_{L^\infty(\Omega \times \mathbb{R}^3)}$.
\end{proposition}
\begin{remark}

(1). With the following new method, we successfully overcome the difficulties caused by the soft potential collision operator. As an a priori estimate, Proposition $\ref{prop 3}$ is valid for $T=\infty$. However, \eqref{hessian phi} may not be valid globally in t. Actually, the infinite-time version of \eqref{hessian phi} is hopeless to be verified without Proposition $\ref{prop 3}$, but the validity of the latter also needs \eqref{hessian phi}. Fortunately, we notice that the finite-time version of Proposition $\ref{prop 3}$ is enough to obtain the local existence.

(2). When $T$ is finite, the conditions \eqref{3.3} and \eqref{3.4} can be simplified by
\begin{align*}
\nabla_{x,v}f \in L^\infty((0,T)\times \Omega \times \mathbb{R}^3),\,\,w_{4\vartheta}\nabla_{t,x,v}g \in L^\infty ((0,T)\times \gamma_-).
\end{align*}
As soon as $T$ is finite, $(1+s)^{\rho-1}$ has a positive lower bound $(1+T)^{\rho-1}$. Applying Gronwall's inequality, we can obtain a by-product, namely,
\begin{align*}
&\|w_{\widetilde{\vartheta}}f(t)\|_p^p + \|w_{\widetilde{\vartheta}}\alpha_{f,\varepsilon}^\beta \partial f(t)\|_p^p \\
&\leq  C_p e^{C_{p,T} (1+ \sup_{0\leq s \leq t} \|\nabla^2 \phi (s)\|_\infty)t^\rho} \Big\{\|w_{\widetilde{\vartheta}}f(0)\|_p^p + \|w_{\widetilde{\vartheta}}\alpha_{f,\varepsilon}^\beta \partial f(0)\|_p^p + \|w_{\widetilde{\vartheta}}g \|^p_{L^p((0,T)\times \gamma_-)}+\sup_{0 \leq s\leq T} \|w_{4\vartheta}\partial g\|_\infty^p\Big\}.
\end{align*}
\end{remark}

\begin{proof}[Proof of Proposition  \ref{prop 3}]
From \eqref{perturb eqn}, we can easily obtain the weighted equation:
\begin{align}\label{weighted eqn}
\{\partial_t + v \cdot \nabla_x - \nabla_x \phi \cdot \nabla_v + \widetilde{\nu}\}(w_{\widetilde{\vartheta}}f)= w_{\widetilde{\vartheta}}(Kf+ \Gamma(f, f)- v \cdot \nabla_x \phi \sqrt{\mu}),
\end{align}
where $\widetilde{\nu}=\nu_{\phi,\widetilde{\vartheta}}=\nu+\frac{v}{2}\cdot \nabla \phi+ \frac{\nabla \phi \cdot \nabla_v w_{\widetilde{\vartheta}}}{w_{\widetilde{\vartheta}}}-\frac{\partial_t w_{\widetilde{\vartheta}}}{w_{\widetilde{\vartheta}}}$.

Based on \eqref{weighted eqn}, we shall respectively estimate $\|w_{\widetilde{\vartheta}}f(t)\|_p^p$ and $ \|w_{\widetilde{\vartheta}}\alpha_{f,\varepsilon}^\beta \partial f(t)\|_p^p$ by the following two steps.

\emph{Step 1}.
Considering the fact that $p |f|^{p-2}f \partial_t f=\partial_t |f|^p$, we multiply $p|w_{\widetilde{\vartheta}}f|^{p-2}w_{\widetilde{\vartheta}}f$ to both sides of \eqref{weighted eqn}, take the integration over $(0,t)\times \Omega \times \mathbb{R}^3$ and integrate by parts, to obtain that
\begin{align}\label{weighted p eqn}
&\|w_{\widetilde{\vartheta}}f (t)\|_p^p + \int_0^t |w_{\widetilde{\vartheta}}f|_{p, +}^p +
\iint_{(0,t)\times \Omega \times \mathbb{R}^3} \widetilde{\nu}|w_{\widetilde{\vartheta}}f|^p \nonumber\\
\lesssim & \|w_{\widetilde{\vartheta}}f (0)\|_p^p+ \int_0^t |w_{\widetilde{\vartheta}}f|_{p, -}^p + \int_0^t \int_{\Omega \times \mathbb{R}^3} |w_{\widetilde{\vartheta}}f|^{p-1}\cdot w_{\widetilde{\vartheta}}(|Kf|+ |\Gamma (f,f)|+ |v \cdot \nabla_x \phi \sqrt{\mu}|).
\end{align}
Here we have used the fact: $\iint_{(0,t)\times \Omega \times \mathbb{R}^3} |w_{\widetilde{\vartheta}}f|^{p-2}w_{\widetilde{\vartheta}}f \cdot \widetilde{\nu} w_{\widetilde{\vartheta}}f=\iint_{(0,t)\times \Omega \times \mathbb{R}^3} \widetilde{\nu}|w_{\widetilde{\vartheta}}f|^p$.
\begin{remark}
There is a crucial advantage of taking the time-velocity weight $w_{\widetilde{\vartheta}}$. It is shown in Section 1 that $\widetilde{\nu}=\nu_{\phi,\widetilde{\vartheta}} \geq \frac{1}{2}\langle v\rangle^\gamma \geq 0$ as $\vartheta \theta > \delta_1$, and hence the term $\iint_{(0,t)\times \Omega \times \mathbb{R}^3} \widetilde{\nu}|w_{\widetilde{\vartheta}}f|^p$ equals to $\int_0^t \left\|\widetilde{\nu}^{\frac{1}{p}}w_{\widetilde{\vartheta}}f \right\|_p^p \geq 0$. Therefore, \eqref{weighted p eqn} can be simplified by
\begin{align}\label{weighted p 0 eqn}
&\|w_{\widetilde{\vartheta}}f (t)\|_p^p + \int_0^t |w_{\widetilde{\vartheta}}f|_{p, +}^p \nonumber\\
\lesssim & \|w_{\widetilde{\vartheta}}f (0)\|_p^p+ \int_0^t |w_{\widetilde{\vartheta}}f|_{p, -}^p + \int_0^t \int_{\Omega \times \mathbb{R}^3} |w_{\widetilde{\vartheta}}f|^{p-1}\cdot w_{\widetilde{\vartheta}}(|Kf|+ |\Gamma (f,f)|+ |v \cdot \nabla_x \phi \sqrt{\mu}|),
\end{align}
thereby avoiding controlling the tricky term $\int_0^t \left\|\widetilde{\nu}^{\frac{1}{p}}w_{\widetilde{\vartheta}}f \right\|_p^p$.

However, if we choose the weight $w_\vartheta(v)=e^{\vartheta|v|^2}$ instead of $w_{\widetilde{\vartheta}}$ as in \eqref{weighted eqn}, the corresponding $\nu_{\phi,\vartheta}=\nu + (1/2+2\vartheta )v \cdot \nabla \phi$ becomes negative as velocity $v$ is large. It means that the term $\int_0^t \int_{x,v} \nu_{\phi,\vartheta}|w_\vartheta f|^p $ may not have nonnegative lower bound and hard to control.
\end{remark}
According to the above discussion, to estimate $\|w_{\widetilde{\vartheta}}f (t)\|_p^p$, we only have to control the right-hand side of \eqref{weighted p 0 eqn} term by term.

(1). The contribution of $Kf$.
\begin{align*}
w_{\widetilde{\vartheta}}K_1 f &= \int_{\mathbb{R}^3 \times \mathbb{S}^2}\sqrt{\mu(v)}|u-v|^\gamma b_0(\theta)\sqrt{\mu(u)}\frac{w_{\widetilde{\vartheta}}(v)}{w_{\widetilde{\vartheta}}(u)}
(w_{\widetilde{\vartheta}}f(u))\mathrm{d}u \mathrm{d}\omega\\
: & = \int_{\mathbb{R}^3}h(v, u)(w_{\widetilde{\vartheta}}f(u))\mathrm{d}u.
\end{align*}
Note that $\int_{\mathbb{R}^3}h(v, u)\mathrm{d}u \lesssim \langle v\rangle^\gamma \lesssim 1$ and  $\int_{\mathbb{R}^3}h(v, u)\mathrm{d}v \lesssim \langle u\rangle^\gamma \lesssim 1$, with H\"{o}lder's inequality, we have
\begin{align*}
\|w_{\widetilde{\vartheta}}K_1 f\|_{L^p_v}\leq & \| \|h^{1/q}\|_{L^p_u} \|h^{1/p}w_{\widetilde{\vartheta}}f(u)\|_{L^p_u}\|_{L^p_v}
\lesssim  \|w_{\widetilde{\vartheta}} f(u)\|_{L^p_u}.
\end{align*}
It yields
\begin{align*}
\int_{\mathbb{R}^3}|w_{\widetilde{\vartheta}} f|^{p-1}|w_{\widetilde{\vartheta}}K_1 f|
\leq \|w_{\widetilde{\vartheta}} f\|_p^{p-1}\|w_{\widetilde{\vartheta}} f\|_p = \|w_{\widetilde{\vartheta}} f\|_p^p.
\end{align*}
Hence the contribution of $K_1 f$ is $\int_0^t \|w_{\widetilde{\vartheta}}f\|_p^p$.

Because of (\ref{k2x}), we know that
\begin{align*}
w_{\widetilde{\vartheta}}K_2^{\chi} f= & \int_{\mathbb{R}^3}\mathbf{k}_2^\chi(v, u)\frac{w_{\widetilde{\vartheta}}(v)}{w_{\widetilde{\vartheta}}(u)}
(w_{\widetilde{\vartheta}}f(u))\mathrm{d}u \\
\leq & \int_{\mathbb{R}^3} \frac{\exp\left( -\frac{s_0}{8}|u-v|^2- \frac{s_0}{8}\frac{(|v|^2-|u|^2)^2}{|v-u|^2}\right)}{|v-u|(1+|v|+|u|)^{1-\gamma}}
\frac{w_{\widetilde{\vartheta}}(v)}{w_{\widetilde{\vartheta}}(u)}
(w_{\widetilde{\vartheta}}f(u))\mathrm{d}u \\
:= & \int_{\mathbb{R}^3}h(v, u)(w_{\widetilde{\vartheta}}f(u))\mathrm{d}u,
\end{align*}
where $s_0 = \min\{s_1, s_2\}$. Let $u= v- \eta$, we calculate the exponent:
\begin{align*}
&-\frac{s_0}{8}|u-v|^2- \frac{s_0}{8}\frac{(|v|^2-|u|^2)^2}{|v-u|^2}-(\pm \widetilde{\vartheta})(|u|^2-|v|^2)\\
= & -(\frac{s_0}{4} \pm \widetilde{\vartheta})|\eta|^2 + (\frac{s_0}{2}\pm 2 \widetilde{\vartheta})v\cdot \eta - \frac{s_0}{2}\frac{|v\cdot \eta|^2}{|\eta|^2}.
\end{align*}
Note that $\widetilde{\vartheta} \leq 2 \vartheta \ll 1$, it holds $\widetilde{\vartheta}< \frac{s_0}{4}$, so $\Delta \equiv - \frac{s_0^2}{4}+ 4 \widetilde{\vartheta}^2 <0$. Therefore, through a similar procedure as the one of Lemma 3 in \cite{Guo2010}, we obtain that
\begin{align*}
&\int_{\mathbb{R}^3}h(v, u)\mathrm{d}u \lesssim \langle v\rangle^{\gamma-2} \lesssim 1,\\
&\int_{\mathbb{R}^3}h(v, u)\mathrm{d}v \lesssim \langle u\rangle^{\gamma-2} \lesssim 1.
\end{align*}
Hence the contribution of $K_2^\chi f$ is $\int_0^t \|w_{\widetilde{\vartheta}}f\|_p^p$.

For the term $w_{\widetilde{\vartheta}}K_2^{1-\chi}f$, for any $g\in L^q_v$, $\frac{1}{p}+ \frac{1}{q}=1$, we have
\begin{align*}
\langle w_{\widetilde{\vartheta}}K_2^{1-\chi}f, g\rangle_{L^p, L^q}=&\int_{\mathbb{R}^3 \times \mathbb{R}^3 \times \mathbb{S}^2} |u-v|^\gamma b_0(\theta)(1-\chi(|u-v|))\sqrt{\mu(u)}\sqrt{\mu(u')}w_{\widetilde{\vartheta}}(v)f(v')g(v)
\mathrm{d}\omega\mathrm{d}u\mathrm{d}v\\
& + \int_{\mathbb{R}^3 \times \mathbb{R}^3 \times \mathbb{S}^2} |u-v|^\gamma b_0(\theta)(1-\chi(|u-v|))\sqrt{\mu(u)}\sqrt{\mu(v')}w_{\widetilde{\vartheta}}(v)f(u')g(v)
\mathrm{d}\omega\mathrm{d}u\mathrm{d}v\\
:=& I_1+I_2.
\end{align*}
Note that $v' = v - ((v-u)\cdot \omega)\omega$, so we have the crucial observation:
\begin{align*}
\big||v|-|v'|\big|\leq |v-u| &\leq 2 \varepsilon,\\
\big||u|-|u'|\big|\leq |v-u| &\leq 2 \varepsilon,\\
\big||u'|-|v|\big|\leq \big||u'|-|u|\big|+\big||u|-|v|\big| & \leq 4 \varepsilon.
\end{align*}
Hence for sufficiently small $\varepsilon$, it holds $w_{\widetilde{\vartheta}}(v)\lesssim w_{\widetilde{\vartheta}}(v')$ and $w_{\widetilde{\vartheta}}(v)\lesssim w_{\widetilde{\vartheta}}(u')$. Then we have
\begin{align*}
I_1 \lesssim &\int_{\mathbb{R}^3 \times \mathbb{R}^3 \times \mathbb{S}^2} |u'-v'|^\gamma b_0(\theta)(1-\chi(|u'-v'|))w_{\widetilde{\vartheta}}(v')f(v')g(v)
\mathrm{d}\omega\mathrm{d}u\mathrm{d}v\\
\lesssim & \left(\int |u'-v'|^\gamma (1-\chi)|w_{\widetilde{\vartheta}}f(v')|^p \mathrm{d}u'\mathrm{d}v'\right)^{\frac{1}{p}} \left(\int |u-v|^\gamma (1-\chi)|g(v)|^q \mathrm{d}u \mathrm{d} v\right)^{\frac{1}{q}}\\
\lesssim & \left(\int |u-v|^\gamma (1-\chi)\mathrm{d}u \int|w_{\widetilde{\vartheta}}f(v)|^p \mathrm{d}v\right)^{\frac{1}{p}} \left(\int |u-v|^\gamma (1-\chi)\mathrm{d}u \int |g(v)|^q  \mathrm{d} v\right)^{\frac{1}{q}}\\
\lesssim & \left(\int |u-v|^\gamma (1-\chi)\mathrm{d}u \right)\|w_{\widetilde{\vartheta}}f\|_p \|g\|_q\\
\lesssim & \varepsilon^{\gamma +3}\|w_{\widetilde{\vartheta}}f\|_p \|g\|_q.
\end{align*}
The estimate of $I_2$ can be treated in a same way. As a result, $\|w_{\widetilde{\vartheta}}K_2^{1-\chi}f\|_p\lesssim \|w_{\widetilde{\vartheta}}f\|_p$.

In summary, the contribution of $Kf$ is $\int_0^t \|w_{\widetilde{\vartheta}}f\|_p^p$.

(2). The contribution of $\Gamma (f, f)$.
On one hand, we have
\begin{align*}
&\Gamma_{\mathrm{gain}}^\chi (f, g)+ \Gamma_{\mathrm{gain}}^\chi (g, f)\\
=& \int_{\mathbb{R}^3 \times \mathbb{S}^2} |u-v|^\gamma \chi(|u-v|)b_0(\theta)\sqrt{\mu(u)}(f(u')g(v')+ f(v')g(u'))\mathrm{d}u \mathrm{d}\omega\\
\lesssim & \|w_{4\vartheta / s_0}f\|_\infty \int_{\mathbb{R}^3 \times \mathbb{S}^2} |u-v|^\gamma \chi(|u-v|)b_0(\theta)w_{4\vartheta / s_0}^{-1}(u)( w_{4\vartheta / s_0}^{-1}(u')g(v')+ w_{4\vartheta / s_0}^{-1}(v')g(u'))\mathrm{d}u \mathrm{d}\omega\\
:= & \|w_{4\vartheta / s_0}f\|_\infty \int_{\mathbb{R}^3} \mathbf{k}_{2, 4\vartheta /s_0}^\chi (v, u)g(u)\mathrm{d}u.
\end{align*}
Through a procedure of calculating the standard $\mathbf{k}_2^\chi$ for $\sqrt{\mu(v)}= e^{-\frac{1}{4}|v|^2}$ (recall that we regard the kernel corresponding to $e^{-\frac{1}{4}|v|^2}$ as ``standard''), for the case $w_{4\vartheta / s_0}^{-1}(v)= e^{- \frac{4\vartheta}{s_0}|v|^2}$, we have
\begin{align*}
\mathbf{k}_{2, 4\vartheta /s_0}^\chi (v, u) \lesssim \frac{\exp\left( -2\vartheta |u-v|^2- 2\vartheta \frac{(|v|^2-|u|^2)^2}{|v-u|^2}\right)}{|v-u|(1+|v|+|u|)^{1-\gamma}}.
\end{align*}
Moreover,
\begin{align*}
&w_{\widetilde{\vartheta}}\left(\Gamma_{\mathrm{gain}}^\chi (f, g)+ \Gamma_{\mathrm{gain}}^\chi (g, f)\right)\\
\lesssim & \|w_{4\vartheta / s_0}f\|_\infty \int_{\mathbb{R}^3} \mathbf{k}_{2, 4\vartheta /s_0}^\chi (v, u)\frac{w_{\widetilde{\vartheta}}(v)}{w_{\widetilde{\vartheta}}(u)}
|w_{\widetilde{\vartheta}}g(u)|\mathrm{d}u.
\end{align*}
Calculate the exponent with $u=v-\eta$:
\begin{align*}
&-2\vartheta |u-v|^2- 2\vartheta \frac{(|v|^2-|u|^2)^2}{|v-u|^2}-(\pm \widetilde{\vartheta})(|u|^2- |v|^2)\\
=& -(4\vartheta \pm \widetilde{\vartheta})|\eta|^2 + (8 \vartheta \pm 2 \widetilde{\vartheta})v\cdot \eta - 8\vartheta \frac{|v \cdot \eta|^2}{|\eta|^2}.
\end{align*}
Note that $\widetilde{\vartheta}\leq 2 \vartheta < 4 \vartheta$, it also holds $\Delta \equiv -64 \vartheta^2 + 4 \widetilde{\vartheta}^2 <0$. Thus, we obtain that
\begin{align*}
&\int_{\mathbb{R}^3} \mathbf{k}_{2, 4\vartheta /s_0}^\chi (v, u)\frac{w_{\widetilde{\vartheta}}(v)}{w_{\widetilde{\vartheta}}(u)}
\mathrm{d}u \lesssim \langle v\rangle^{\gamma-2}\lesssim 1,\\
&\int_{\mathbb{R}^3} \mathbf{k}_{2, 4\vartheta /s_0}^\chi (v, u)\frac{w_{\widetilde{\vartheta}}(v)}{w_{\widetilde{\vartheta}}(u)}
\mathrm{d}v \lesssim \langle u\rangle^{\gamma-2}\lesssim 1.
\end{align*}
Therefore, the contribution of $\Gamma_{\mathrm{gain}}^\chi(f, f)$ is $\sup_{0 \leq s \leq t}\|w_{4\vartheta / s_0}f(s)\|_\infty \int_0^t \|w_{\widetilde{\vartheta}}f\|_p^p$.

On the other hand, note that
\begin{align*}
& w_{\widetilde{\vartheta}}\left(\Gamma_{\mathrm{gain}}^{1-\chi} (f, g)+ \Gamma_{\mathrm{gain}}^{1-\chi} (g, f)\right)\\
\lesssim & \|w_{\widetilde{\vartheta}}f\|_\infty \int_{\mathbb{R}^3\times \mathbb{S}^2} |u-v|^\gamma (1-\chi)b_0(\theta)\sqrt{\mu(u)}w_{\widetilde{\vartheta}}(v)
(w_{\widetilde{\vartheta}}^{-1}(u')g(v')+ w_{\widetilde{\vartheta}}^{-1}(v')g(u')) \mathrm{d}u\mathrm{d}\omega.
\end{align*}
Following the same path as $w_{\widetilde{\vartheta}}K_2^{1- \chi}f$, we have
\begin{align*}
\|w_{\widetilde{\vartheta}}\Gamma_{\mathrm{gain}}^{1-\chi}(f, f)\|_p \lesssim \varepsilon^{\gamma +3}\|w_{\widetilde{\vartheta}}f\|_\infty \|w_{\widetilde{\vartheta}}f\|_p.
\end{align*}

As for the loss term $\Gamma_{\mathrm{loss}}(f, f)$, it is clear that
\begin{align*}
w_{\widetilde{\vartheta}}\Gamma_{\mathrm{loss}}(f, f)\lesssim & |w_{\widetilde{\vartheta}}f(v)|\int_{\mathbb{R}^3}|u-v|^\gamma \sqrt{\mu(u)}f(u)\mathrm{d}u\\
\lesssim & \|w_{\widetilde{\vartheta}}f\|_\infty |w_{\widetilde{\vartheta}}f(v)|.
\end{align*}
In summary, the contribution of $\Gamma (f, f)$ is $\sup_{0 \leq s \leq t}\|w_{4\vartheta / s_0}f(s)\|_\infty \int_0^t \|w_{\widetilde{\vartheta}}f\|_p^p$.

(3). The contribution of $v\cdot \nabla \phi \sqrt{\mu}$. It is easy to know that
\begin{align*}
\|v\cdot \nabla \phi \sqrt{\mu} w_{\widetilde{\vartheta}}\|_{L^p_{x, v}}
\lesssim  \|\nabla \phi\|_{L^p_x} \lesssim \|\phi\|_{W^{2, p}(\Omega)}\lesssim \left\|\int_{\mathbb{R}^3} f \sqrt{\mu}\mathrm{d}v\right\|_{L^p_x} \lesssim \|w_{\widetilde{\vartheta}}f\|_p.
\end{align*}
Hence the contribution of $v\cdot \nabla \phi \sqrt{\mu}$ is $\int_0^t \|w_{\widetilde{\vartheta}}f\|_p^p$.

In conclusion, we have
\begin{align*}
\|w_{\widetilde{\vartheta}}f(t)\|_p^p \lesssim \|w_{\widetilde{\vartheta}}f(0)\|_p^p+ \|w_{\widetilde{\vartheta}} g\|_{L^p((0, T)\times \gamma_-)}^p + \left(1+\sup_{0 \leq s \leq t}\|w_{4\vartheta / s_0}f(s)\|_\infty \right) \int_0^t \|w_{\widetilde{\vartheta}}f\|_p^p.
\end{align*}

\emph{Step 2}.
Considering \eqref{K} and \eqref{Gamma}, and changing the variables $u-v\rightarrow u$, we have
\begin{align*}
\Gamma(f_1,f_2)=&\int_{\mathbb{R}^3\times \mathbb{S}^2}|u-v|^\gamma b_0(\theta)\sqrt{\mu(u)}[f_1(u')f_2(v')-f_1(u)f_2(v)]\mathrm{d}\omega\mathrm{d}u\\
=& \int_{\mathbb{R}^3\times \mathbb{S}^2}|u|^\gamma b_0(\theta)\sqrt{\mu(v+u)}[f_1(v+u_\perp)f_2(v+u_\parallel)-f_1(v+u)f_2(v)]\mathrm{d}\omega\mathrm{d}u,
\end{align*}
where $u_\parallel=(u\cdot \omega)\omega$ and $u_\perp=u-u_\parallel$. By direct computations,
\begin{align*}
\nabla_v \Gamma(f_1,f_2)=\Gamma(\nabla_v f_1,f_2)+\Gamma(f_1,\nabla_v f_2)+\Gamma_v(f_1,f_2),
\end{align*}
where we define
\begin{align*}
\Gamma_v(f_1,f_2)(v):=&\int_{\mathbb{R}^3\times \mathbb{S}^2}|u|^\gamma b_0(\theta)\nabla_v \sqrt{\mu(v+u)}[f_1(v+u_\perp)f_2(v+u_\parallel)-f_1(v+u)f_2(v)]\mathrm{d}\omega\mathrm{d}u\\
=&\int_{\mathbb{R}^3\times \mathbb{S}^2}|u-v|^\gamma b_0(\theta)\nabla_u\sqrt{\mu(u)}[f_1(u')f_2(v')-f_1(u)f_2(v)]\mathrm{d}\omega\mathrm{d}u\\
:=& \Gamma_{v, \mathrm{gain}}(f_1,f_2)(v)-\Gamma_{v, \mathrm{loss}}(f_1,f_2)(v).
\end{align*}
We also denote that
\begin{align*}
\Gamma_{v, \mathrm{gain}}^\chi (f_1,f_2)(v) :=& \int_{\mathbb{R}^3\times \mathbb{S}^2}|u-v|^\gamma \chi(|u-v|) b_0(\theta)\nabla_u\sqrt{\mu(u)}f_1(u')f_2(v')\mathrm{d}\omega\mathrm{d}u, \\
\Gamma_{v, \mathrm{loss}}^\chi (f_1,f_2)(v) :=& \int_{\mathbb{R}^3\times \mathbb{S}^2}|u-v|^\gamma \chi(|u-v|) b_0(\theta)\nabla_u\sqrt{\mu(u)}f_1(u)f_2(v)\mathrm{d}\omega\mathrm{d}u,\\
\nabla_x \Gamma(f_1,f_2)=&\Gamma(\nabla_x f_1,f_2)+\Gamma(f_1,\nabla_x f_2),
\end{align*}
and $\Gamma_{v, \mathrm{loss}}^{1-\chi}:=\Gamma_{v, \mathrm{loss}}-\Gamma_{v, \mathrm{loss}}^{\chi}$, $\Gamma_{v, \mathrm{gain}}^{1-\chi}:=\Gamma_{v, \mathrm{gain}}-\Gamma_{v, \mathrm{gain}}^{\chi}$.
Noticing that the definition of $K_i^\chi f$, in \eqref{K chi}, we have
\begin{align*}
K_1 f=&\sqrt{\mu(v)}\int_{\mathbb{R}^3 \times \mathbb{S}^2}|u-v|^\gamma b_0(\theta)\sqrt{\mu(u)}f(u)\mathrm{d}u \mathrm{d}\omega:=K_1^\chi f+K_1^{1-\chi} f,\\
K_2 f=&\int_{\mathbb{R}^3 \times \mathbb{S}^2}|u-v|^\gamma b_0(\theta)\sqrt{\mu(u)}\{f(u')\sqrt{\mu(v')}+f(v')\sqrt{\mu(u')}\}\mathrm{d}u \mathrm{d}\omega\\
:=& \int_{\mathbb{R}^3}\mathbf{k}_2^\chi (v,u)f(u)\mathrm{d}u+K_2^{1-\chi}f.
\end{align*}
By the same reason as above, we define
\begin{align*}
\nabla_v K_1 f(v):=& K_{v,1}^\chi f+K_{v,1}^{1-\chi} f+K_1^\chi \nabla_v f+K_1^{1-\chi} \nabla_v f,\\
\nabla_v K_2 f(v):=& \int_{\mathbb{R}^3}\nabla_v\mathbf{k}_2^\chi (v,u)f(u)\mathrm{d}u+K_{v,2}^{1-\chi} f+\int_{\mathbb{R}^3}\mathbf{k}_2^\chi (v,u)\nabla_v f(u)\mathrm{d}u+K_2^{1-\chi} \nabla_v f,\\
\nabla_x K_1 f(v):=& K_1^\chi \nabla_x f+K_1^{1-\chi} \nabla_x f,\\
\nabla_x K_2 f(v):=& \int_{\mathbb{R}^3}\mathbf{k}_2^\chi (v,u)\nabla_x f(u)\mathrm{d}u+K_2^{1-\chi} \nabla_x f.
\end{align*}

Based on \eqref{weighted eqn}, analogous to the derivation of \eqref{weighted p 0 eqn}, for $|w_{\widetilde{\vartheta}}\alpha^\beta_{f, \varepsilon}\partial f|^p$, $\partial=\nabla_{x,v}$, we have
\begin{align*}
&\|w_{\widetilde{\vartheta}}\alpha^\beta_{f, \varepsilon}\partial f(t)\|_p^p + \int_0^t \|\widetilde{\nu}^{1/p}w_{\widetilde{\vartheta}}\alpha^\beta_{f, \varepsilon}\partial f\|_p^p+ \int_0^t |w_{\widetilde{\vartheta}}\alpha^\beta_{f, \varepsilon}\partial f|_{p, +}^p \\
\lesssim &\|w_{\widetilde{\vartheta}}\alpha^\beta_{f, \varepsilon}\partial f(0)\|_p^p
+ \int_0^t |w_{\widetilde{\vartheta}}\alpha^\beta_{f, \varepsilon}\partial f|_{p, -}^p
+ \int_0^t \int_{\Omega \times \mathbb{R}^3}w_{\widetilde{\vartheta}}^p \alpha^{\beta p}_{f, \varepsilon}|\partial f|^{p-1} |\mathcal{G}|,
\end{align*}
where
\begin{align}\label{G}
|\mathcal{G}| \lesssim & |\nabla_x f|+ |\nabla^2 \phi||\nabla_v f|+ |\Gamma (\partial f, f)|+ |\Gamma (f, \partial f)|+ |K \partial f|+ |\Gamma_v (f, f)|\nonumber\\
& + |K_v f|+ |\nabla_v \nu||f|+ |\partial(v \cdot \nabla_x \phi)f|+ |\partial(v \cdot \nabla_x \phi \sqrt{\mu})|.
\end{align}
Once again, due to $\widetilde{\nu}\geq 0$ and $\int_0^t \|\widetilde{\nu}^{1/p}w_{\widetilde{\vartheta}}\alpha^\beta_{f, \varepsilon}\partial f\|_p^p \geq 0$, it immediately holds
\begin{align}\label{weighted alpha}
&\|w_{\widetilde{\vartheta}}\alpha^\beta_{f, \varepsilon}\partial f(t)\|_p^p +  \int_0^t |w_{\widetilde{\vartheta}}\alpha^\beta_{f, \varepsilon}\partial f|_{p, +}^p \nonumber\\
\lesssim &\|w_{\widetilde{\vartheta}}\alpha^\beta_{f, \varepsilon}\partial f(0)\|_p^p
+ \int_0^t |w_{\widetilde{\vartheta}}\alpha^\beta_{f, \varepsilon}\partial f|_{p, -}^p
+ \int_0^t \int_{\Omega \times \mathbb{R}^3}w_{\widetilde{\vartheta}}^p \alpha^{\beta p}_{f, \varepsilon}|\partial f|^{p-1} |\mathcal{G}|.
\end{align}
To estimate $\|w_{\widetilde{\vartheta}}\alpha^\beta_{f, \varepsilon}\partial f(t)\|_p^p$, we deal with the right-hand side of \eqref{weighted alpha} term by term, with \eqref{G}.

(1). It is clear that the contribution of $|\nabla_x f|+|\nabla^2 \phi||\nabla_v f|$ is
$(1+ \sup_{0 \leq s\leq t}\|\nabla^2 \phi\|_\infty)\int_0^t \|w_{\widetilde{\vartheta}}\alpha^\beta_{f, \varepsilon}\partial f\|_p^p$.

(2). The contribution of $|\nabla_v \nu(v)||f|$.

Note that $\nu (v)\sim \langle v\rangle^\gamma$, it holds that $|\nabla_v \nu(v)|\cdot \alpha^\beta_{f, \varepsilon}\lesssim \langle v\rangle^{\gamma-1}\cdot \langle v\rangle^\beta \lesssim 1.$ By using Young's inequality, we have
\begin{align*}
&\int_{\Omega \times \mathbb{R}^3}w_{\widetilde{\vartheta}}^{p-1} \alpha^{\beta (p-1)}_{f, \varepsilon}|\partial f|^{p-1} |\nabla_v \nu(v)|\cdot \alpha^\beta_{f, \varepsilon}w_{\widetilde{\vartheta}}|f|
\lesssim \int_{\Omega \times \mathbb{R}^3} |w_{\widetilde{\vartheta}}\alpha^\beta_{f, \varepsilon}\partial f|^p + |w_{\widetilde{\vartheta}} f|^p
\end{align*}
So the contribution of $|\nabla_v \nu(v)||f|$ is $\int_0^t \|w_{\widetilde{\vartheta}}\alpha^\beta_{f, \varepsilon}\partial f\|_p^p + \int_0^t \|w_{\widetilde{\vartheta}} f\|_p^p$.

(3). $|\partial(v \cdot \nabla_x \phi)f|+ |\partial(v \cdot \nabla_x \phi \sqrt{\mu})|$.
On one hand, it holds that
\begin{align*}
\alpha^\beta_{f, \varepsilon} w_{\widetilde{\vartheta}} |\partial(v \cdot \nabla_x \phi)f|
\lesssim & \frac{\langle v\rangle^\beta}{w_{\widetilde{\vartheta}/4}(v)} w_{\widetilde{\vartheta}/4}^{-1}(v) (|\nabla \phi|+ |\nabla^2 \phi|)\|w_{2\widetilde{\vartheta}}f\|_\infty\\
\lesssim & \|w_{2\widetilde{\vartheta}}f\|_\infty w_{\widetilde{\vartheta}/4}^{-1}(v) (|\nabla \phi|+ |\nabla^2 \phi|).
\end{align*}
With Young's inequality, it holds
\begin{align*}
&|w_{\widetilde{\vartheta}}\alpha^\beta_{f, \varepsilon}\partial f|^{p-1} \times \alpha^\beta_{f, \varepsilon} w_{\widetilde{\vartheta}}|\partial(v \cdot \nabla_x \phi)f|\\
\lesssim & \|w_{2 \widetilde{\vartheta}}f\|_\infty \times (|w_{\widetilde{\vartheta}}\alpha^\beta_{f, \varepsilon}\partial f|^p + w_{\widetilde{\vartheta}/4}^{-p}(v) (|\nabla \phi|^p+ |\nabla^2 \phi|^p)).
\end{align*}
Note that
\begin{align*}
\int w_{\widetilde{\vartheta}/4}^{-p}(v)\mathrm{d}v (|\nabla \phi|^p+ |\nabla^2 \phi|^p) \mathrm{d}x \lesssim \|\phi\|^p_{W^{2, p}(\Omega)}\lesssim \|w_{\widetilde{\vartheta}}f\|_p^p.
\end{align*}
On the other hand, $\partial (v \sqrt{\mu} \cdot \nabla \phi)\lesssim \mu^{\frac{1}{4}}(v)(|\nabla \phi|+|\nabla^2 \phi|)$. So the detail of estimating $\partial (v \sqrt{\mu} \cdot \nabla \phi)$ is similar to $|\partial(v \cdot \nabla_x \phi)f|$.

The contribution of $|\partial(v \cdot \nabla_x \phi)f|+ |\partial(v \cdot \nabla_x \phi \sqrt{\mu})|$ is
\begin{align*}
\left(1+ \sup_{0 \leq s \leq t}\|w_{2 \widetilde{\vartheta}}f\|_\infty \right)\left(\int_0^t \|w_{\widetilde{\vartheta}}\alpha^\beta_{f, \varepsilon}\partial f\|_p^p + \int_0^t \|w_{\widetilde{\vartheta}} f\|_p^p\right).
\end{align*}

(4). $|\Gamma_v(f, f)|+|K_v f|$.
First, it holds
\begin{align*}
\alpha^\beta_{f, \varepsilon}w_{\widetilde{\vartheta}}\Gamma_{v, \mathrm{loss}}(f, f)
\lesssim & \langle v\rangle^\beta w_{\widetilde{\vartheta}}\int_{\mathbb{R}^3\times \mathbb{S}^2}|u-v|^\gamma b_0(\theta)\nabla_v \sqrt{\mu(u)}f(u)f(v)\mathrm{d}u \mathrm{d}\omega\\
\lesssim & \|w_{2 \widetilde{\vartheta}}f\|_\infty \int_{\mathbb{R}^3\times \mathbb{S}^2}|u-v|^\gamma b_0(\theta)\nabla_v \sqrt{\mu(u)}\langle v\rangle^\beta w_{\widetilde{\vartheta}}^{-1}(v)w_{\widetilde{\vartheta}}^{-1}(u)
(w_{\widetilde{\vartheta}}f(u))
\mathrm{d}u \mathrm{d}\omega\\
:= & \|w_{2 \widetilde{\vartheta}}f\|_\infty \int_{\mathbb{R}^3}h(v, u)
(w_{\widetilde{\vartheta}}f(u))
\mathrm{d}u.
\end{align*}
Note that $\int_{\mathbb{R}^3}h(v, u)\mathrm{d}u \lesssim 1$ and $\int_{\mathbb{R}^3}h(v, u)\mathrm{d}v \lesssim 1$. By using Young's inequality, the contribution of $\Gamma_{v, \mathrm{loss}}(f, f)$ is
\begin{align*}
\sup_{0 \leq s \leq t}\|w_{2 \widetilde{\vartheta}}f\|_\infty \left(\int_0^t \|w_{\widetilde{\vartheta}}\alpha^\beta_{f, \varepsilon}\partial f\|_p^p + \int_0^t \|w_{\widetilde{\vartheta}} f\|_p^p\right).
\end{align*}
For the gain term, we have
\begin{align*}
&\alpha^\beta_{f, \varepsilon}w_{\widetilde{\vartheta}}\Gamma^\chi_{v, \mathrm{gain}}(f, f)\\
\lesssim & \|w_{4 \vartheta /s_0}f\|_\infty \int_{\mathbb{R}^3}\mathbf{k}_{2, 4\vartheta /s_0}^\chi (v, u)\frac{w_{\widetilde{\vartheta}}(v)}{w_{\widetilde{\vartheta}}(u)}\langle v\rangle^\beta
|w_{\widetilde{\vartheta}}f(u)|\mathrm{d}u\\
\lesssim & \|w_{4 \vartheta /s_0}f\|_\infty \int_{\mathbb{R}^3} \frac{\exp\left( -2\vartheta |u-v|^2- 2\vartheta \frac{(|v|^2-|u|^2)^2}{|v-u|^2}\right)}{|v-u|}
\frac{w_{\widetilde{\vartheta}}(v)}{w_{\widetilde{\vartheta}}(u)}
|w_{\widetilde{\vartheta}}f(u)|\mathrm{d}u.
\end{align*}
Due to
\begin{align*}
&\int_{\mathbb{R}^3} \frac{\exp\left( -2\vartheta |u-v|^2- 2\vartheta \frac{(|v|^2-|u|^2)^2}{|v-u|^2}\right)}{|v-u|}
\frac{w_{\widetilde{\vartheta}}(v)}{w_{\widetilde{\vartheta}}(u)}
\mathrm{d}u \lesssim \langle v\rangle^{-1}\lesssim 1,\\
&\int_{\mathbb{R}^3} \frac{\exp\left( -2\vartheta |u-v|^2- 2\vartheta \frac{(|v|^2-|u|^2)^2}{|v-u|^2}\right)}{|v-u|}
\frac{w_{\widetilde{\vartheta}}(v)}{w_{\widetilde{\vartheta}}(u)}
\mathrm{d}v \lesssim \langle u\rangle^{-1}\lesssim 1,
\end{align*}
we know that the contribution of $\Gamma^\chi_{v, \mathrm{gain}}(f, f)$ is
\begin{align*}
\sup_{0 \leq s \leq t}\|w_{4 \vartheta /s_0}f\|_\infty \left(\int_0^t \|w_{\widetilde{\vartheta}}\alpha^\beta_{f, \varepsilon}\partial f\|_p^p + \int_0^t \|w_{\widetilde{\vartheta}} f\|_p^p\right).
\end{align*}
Meanwhile,
\begin{align*}
&\alpha^\beta_{f, \varepsilon}w_{\widetilde{\vartheta}}\Gamma^{1-\chi}_{v, \mathrm{gain}}(f, f)\\
\lesssim & \|w_{\widetilde{\vartheta}}f\|_\infty \int_{\mathbb{R}^3 \times \mathbb{S}^2} |u-v|^\gamma b_0(\theta)(1-\chi)\nabla_v \sqrt{\mu(u)}w_{\widetilde{\vartheta}}(v)\langle v\rangle^\beta (w_{\widetilde{\vartheta}}^{-1}(v')f(u')+w_{\widetilde{\vartheta}}^{-1}(u')f(v') )
\mathrm{d}u \mathrm{d}\omega.
\end{align*}
Through a similar way to estimate $K_2^{1-\chi} f$ in \emph{Step 1}, we obtain
\begin{align*}
\|\alpha^\beta_{f, \varepsilon}w_{\widetilde{\vartheta}}\Gamma^{1-\chi}_{v, \mathrm{gain}}(f, f)\|_p \lesssim \|w_{\widetilde{\vartheta}}f\|_\infty \|w_{\widetilde{\vartheta}}f\|_p.
\end{align*}
Therefore, in summary, the contribution of $\Gamma_v (f, f)$ is
\begin{align*}
\sup_{0 \leq s \leq t}\|w_{4 \vartheta /s_0}f\|_\infty \left(\int_0^t \|w_{\widetilde{\vartheta}}\alpha^\beta_{f, \varepsilon}\partial f\|_p^p + \int_0^t \|w_{\widetilde{\vartheta}} f\|_p^p\right).
\end{align*}

Next, concerned about the contribution of $K_v f=K_{v,1}f-K_{v,2}f$, it holds for $K_{v,1}f$ that
\begin{align*}
\alpha^\beta_{f, \varepsilon}w_{\widetilde{\vartheta}}K_{v, 1}f \lesssim & \int_{\mathbb{R}^3}
\nabla_v \mathbf{k}_1(v, u)\frac{w_{\widetilde{\vartheta}}(v)}{w_{\widetilde{\vartheta}}(u)}\langle v\rangle^\beta
(w_{\widetilde{\vartheta}}f(u))\mathrm{d}u\\
\lesssim & \int_{\mathbb{R}^3}
|u-v|^\gamma \mu^{1/4}(u)\mu^{1/4}(v)\frac{w_{\widetilde{\vartheta}}(v)}{w_{\widetilde{\vartheta}}(u)}\langle v\rangle^\beta
(w_{\widetilde{\vartheta}}f(u))\mathrm{d}u.
\end{align*}
Noting that
\begin{align*}
&\int_{\mathbb{R}^3}
|u-v|^\gamma \mu^{1/4}(u)\mu^{1/4}(v)\frac{w_{\widetilde{\vartheta}}(v)}{w_{\widetilde{\vartheta}}(u)}\langle v\rangle^\beta \mathrm{d}u \lesssim 1,\\
&\int_{\mathbb{R}^3}
|u-v|^\gamma \mu^{1/4}(u)\mu^{1/4}(v)\frac{w_{\widetilde{\vartheta}}(v)}{w_{\widetilde{\vartheta}}(u)}\langle v\rangle^\beta \mathrm{d}v \lesssim 1,
\end{align*}
so the contribution of $K_{v, 1}f$ is $\int_0^t \|w_{\widetilde{\vartheta}}\alpha^\beta_{f, \varepsilon}\partial f\|_p^p + \int_0^t \|w_{\widetilde{\vartheta}} f\|_p^p$.

For the term $K_{v,2}f=K_{v,2}^\chi f+K_{v,2}^{1-\chi}f$, by Lemma 1 in \cite{Strain}, it holds the crucial estimate that
\begin{align*}
|\partial_v \mathbf{k}_2^\chi (v, u)| \lesssim \frac{\exp\left( -\frac{s_2}{8}|u-v|^2- \frac{s_1}{8}\frac{(|v|^2-|u|^2)^2}{|v-u|^2}\right)}{|v-u|(1+|v|+|u|)^{1-\gamma}}.
\end{align*}
Hence by the former discussion, we know the contribution of $K_{v, 2}^\chi f$ is  $\int_0^t \|w_{\widetilde{\vartheta}}\alpha^\beta_{f, \varepsilon}\partial f\|_p^p + \int_0^t \|w_{\widetilde{\vartheta}} f\|_p^p$.

For the term $K_{v, 2}^{1-\chi} f$, we have
\begin{align*}
K_{v, 2}^{1-\chi} f
=& \int_{\mathbb{R}^3 \times \mathbb{S}^2}|u-v|^\gamma (1-\chi)b_0(\theta)[(\nabla_v \sqrt{\mu(u)}\sqrt{\mu(u')}+ \sqrt{\mu(u)}\nabla_v \sqrt{\mu(u')})f(v')\\
&+ (\nabla_v \sqrt{\mu(u)}\sqrt{\mu(v')}+ \sqrt{\mu(u)}\nabla_v \sqrt{\mu(v')})f(u')]\mathrm{d}u \mathrm{d}\omega\\
\lesssim &  \int_{\mathbb{R}^3 \times \mathbb{S}^2}|u-v|^\gamma (1-\chi)b_0(\theta)[\mu^{1/4}(u)\mu^{1/4}(u')f(v')+ \mu^{1/4}(u)\mu^{1/4}(v')f(u')]\mathrm{d}u \mathrm{d}\omega.
\end{align*}
Through the same way to estimate $K_2^{1-\chi} f$ in \emph{Step 1} again, we know the contribution of $K_{v, 2}^{1- \chi}f$ is $\int_0^t \|w_{\widetilde{\vartheta}}\alpha^\beta_{f, \varepsilon}\partial f\|_p^p + \int_0^t \|w_{\widetilde{\vartheta}} f\|_p^p$.

In summary, the contribution of $|\Gamma_v(f, f)|+|K_v f|$ is
\begin{align*}
\left(1+\sup_{0 \leq s \leq t}\|w_{4 \vartheta /s_0}f\|_\infty\right) \left(\int_0^t \|w_{\widetilde{\vartheta}}\alpha^\beta_{f, \varepsilon}\partial f\|_p^p + \int_0^t \|w_{\widetilde{\vartheta}} f\|_p^p\right).
\end{align*}

(5). $\Gamma(\partial f, f)+ \Gamma (f, \partial f)$.
First, it holds for $\Gamma_{\mathrm{gain}}^\chi (\partial f, f)+\Gamma_{\mathrm{gain}}^\chi (f, \partial f)$ that
\begin{align*}
&\alpha^\beta_{f, \varepsilon}w_{\widetilde{\vartheta}}(\Gamma_{\mathrm{gain}}^\chi (\partial f, f)+\Gamma_{\mathrm{gain}}^\chi (f, \partial f) )\\
\lesssim & \|w_{4 \vartheta /s_0}f\|_\infty \int_{\mathbb{R}^3}\mathbf{k}_{2, 4\vartheta /s_0}^\chi (v, u)\frac{w_{\widetilde{\vartheta}}(v)}{w_{\widetilde{\vartheta}}(u)}\langle v\rangle^\beta
|w_{\widetilde{\vartheta}}\partial f(u)|\mathrm{d}u\\
\lesssim & \|w_{4 \vartheta /s_0}f\|_\infty \int_{\mathbb{R}^3} \frac{\exp\left( -2\vartheta |u-v|^2- 2\vartheta \frac{(|v|^2-|u|^2)^2}{|v-u|^2}\right)}{|v-u|}
\frac{w_{\widetilde{\vartheta}}(v)}{w_{\widetilde{\vartheta}}(u)}
|w_{\widetilde{\vartheta}}\partial f(u)|\mathrm{d}u.
\end{align*}
Denote $\mathbf{k}_{2\vartheta}(v, u):=\frac{\exp\left( -2\vartheta |u-v|^2- 2\vartheta \frac{(|v|^2-|u|^2)^2}{|v-u|^2}\right)}{|v-u|} $ and $\mathbf{k}_w (v, u):= \mathbf{k}_{2\vartheta}(v, u)\frac{w_{\widetilde{\vartheta}}(v)}{w_{\widetilde{\vartheta}}(u)} $. We have already known that $\int_{\mathbb{R}^3} \mathbf{k}_w (v, u) \mathrm{d}u \lesssim 1$ and $\int_{\mathbb{R}^3} \mathbf{k}_w (v, u) \mathrm{d}u \lesssim 1$. Moreover, see for example the appendix in \cite{Kim}, for $2(2\vartheta -\vartheta_1) > \widetilde{\vartheta}$, i.e. $\vartheta_1 < 2\vartheta - \frac{\widetilde{\vartheta}}{2}$, we have $\mathbf{k}_w (v, u)\lesssim \mathbf{k}_{\vartheta_1} (v, u)$. Note that $\mathbf{k}_{\vartheta_1} (v, u)$ is symmetric on $v$ and $u$, but $\mathbf{k}_w (v, u)$ is obviously not.

For some $N>1$, we divide $\mathbb{R}^3_u$ into two parts:  $\{|u|\geq N\}$ and $ \{|u|\leq N\}$.

$(\mathrm{i})$. For the part $\{|u|\geq N\}$, it holds for $\frac{1}{p}+ \frac{1}{q}=1$ that
\begin{align*}
&\int_{\{|u|\geq N\}} \mathbf{k}_w^{\frac{1}{q}+ \frac{1}{p}}(v, u)|w_{\widetilde{\vartheta}}\partial f(u)|\mathrm{d}u\\
\lesssim & \left(\int_{\{|u|\geq N\}}\mathbf{k}_{\vartheta_1}(v, u) \frac{1}{\alpha^{\beta q}_{f, \varepsilon}}\mathrm{d}u \right)^{\frac{1}{q}}\left(\int_{\{|u|\geq N\}}\mathbf{k}_w (v, u)|w_{\widetilde{\vartheta}}\alpha^\beta_{f, \varepsilon}\partial f|^p \mathrm{d}u \right)^{\frac{1}{p}}.
\end{align*}
Because $\beta q = \beta \frac{p}{p-1}< \frac{2-\varpi}{3-\varpi} \frac{p}{p-1}< 1$, by Lemma \ref{alpha beta}, it holds
\begin{align*}
\int_{\{|u|\geq N\}}\mathbf{k}_{\vartheta_1}(v, u) \frac{1}{\alpha^{\beta q}_{f, \varepsilon}}\mathrm{d}u \lesssim \int_{\{|u|\geq N\}} \frac{\exp(- \vartheta_1 |v-u|^2)}{|v-u|} \frac{1}{\alpha^{\beta q}_{f, \varepsilon}}\mathrm{d}u \lesssim 1.
\end{align*}
Hence,
\begin{align*}
&\int_{\Omega \times \mathbb{R}^3}\left|\int_{\{|u|\geq N\}} \mathbf{k}_w(v, u)|w_{\widetilde{\vartheta}}\partial f(u)|\mathrm{d}u\right|^p \mathrm{d}x \mathrm{d}v\\
\lesssim & \int_{\Omega \times \mathbb{R}^3}\int_{\{|u|\geq N\}} \mathbf{k}_w (v, u)|w_{\widetilde{\vartheta}}\alpha^\beta_{f, \varepsilon}\partial f(u)|^p \mathrm{d}u \mathrm{d}x \mathrm{d}v \\
\lesssim & \|w_{\widetilde{\vartheta}}\alpha^\beta_{f, \varepsilon}\partial f\|_{L^p_{x, u}}^p.
\end{align*}

$(\mathrm{ii})$. For the part $\{|u|\leq N\}$, we have
\begin{align*}
&\int_{\{|u|\leq N\}} \mathbf{k}_w(v, u)|w_{\widetilde{\vartheta}}\partial f(u)|\mathrm{d}u\\
\lesssim & \|w_{\widetilde{\vartheta}}\alpha^\beta_{f,\varepsilon}\partial f\|_p
\left(\int_{\mathbb{R}^3} \frac{\mathbf{k}_{\vartheta_1}^q (v, u)\mathbf{1}_{\{|u|\leq N\}}}{\alpha^{\beta q}_{f, \varepsilon}(u)}\mathrm{d}u\right)^{\frac{1}{q}}
\end{align*}
In light of \cite{Kim}, we notice that
\begin{align*}
1+ \frac{1}{\frac{p}{q}}= \frac{1}{\frac{3-\varpi}{q}}+ \frac{1}{\frac{\sigma}{q}},\,\,\, \sigma= \frac{3-\varpi}{2-\varpi},
\end{align*}
and
\begin{align*}
\int_{\mathbb{R}^3} \frac{\mathbf{k}_{\vartheta_1}^q (v, u)\mathbf{1}_{\{|u|\leq N\}}}{\alpha^{\beta q}_{f, \varepsilon}(u)}\mathrm{d}u \leq \left|\frac{e^{-\vartheta_1 q |\cdot|^2}}{|\cdot|^q}\ast \frac{\mathbf{1}_{\{|\cdot|\leq N\}}}{\alpha^{\beta q}_{f, \varepsilon}(\cdot)}\right|,
\end{align*}
then we use Young's inequality of the convolution version, that is
\begin{align*}
\left\|\frac{e^{-\vartheta_1 q |\cdot|^2}}{|\cdot|^q}\ast \frac{\mathbf{1}_{\{|\cdot|\leq N\}}}{\alpha^{\beta q}_{f, \varepsilon}(\cdot)}\right\|_{\frac{p}{q}} \lesssim \left\|\frac{e^{-\vartheta_1 q |\cdot|^2}}{|\cdot|^q}\right\|_{\frac{3-\varpi}{q}} \times \left\|\frac{\mathbf{1}_{\{|\cdot|\leq N\}}}{\alpha^{\beta q}_{f, \varepsilon}(\cdot)}\right\|_{\frac{\sigma}{q}}\lesssim \left\|\frac{\mathbf{1}_{\{|\cdot|\leq N\}}}{\alpha^{\beta q}_{f, \varepsilon}(\cdot)}\right\|_{\frac{\sigma}{q}}.
\end{align*}
Note that $\beta \sigma = \beta \frac{3-\varpi}{2-\varpi}<1$ and use Lemma \ref{alpha beta} again, we know that the right-hand side above is finite. Hence,
\begin{align*}
\int_{\mathbb{R}^3_u} \mathbf{k}_w(v, u)|w_{\widetilde{\vartheta}}\partial f(u)|\mathrm{d}u \lesssim  \|w_{\widetilde{\vartheta}}\alpha^\beta_{f, \varepsilon}\partial f\|_p^p.
\end{align*}
Therefore, the contribution of $\Gamma_{\mathrm{gain}}^\chi (\partial f, f)+\Gamma_{\mathrm{gain}}^\chi (f, \partial f)$ is $\sup_{0 \leq s \leq t}\|w_{4 \vartheta /s_0}f\|_\infty \int_0^t \|w_{\widetilde{\vartheta}}\alpha^\beta_{f, \varepsilon}\partial f\|_p^p$.

Next, for the term $\Gamma_{\mathrm{gain}}^{1-\chi}(\partial f, f)+ \Gamma_{\mathrm{gain}}^{1-\chi}(f, \partial f)$,
if $|u-v|\leq 2\varepsilon$, we have already known that $|v'-u|=|u'-v|\leq 4 \varepsilon$. Recalling the definition of $\alpha_{f, \varepsilon}$ (\ref{kinetic weight}), we divide $t-t_b(t, x, v):= \varsigma(x,v)$ into three cases:

(i). $t-t_b(t, x, v) \leq -\varepsilon$. It holds in this case that $\alpha_{f, \varepsilon}(t, x, v) \equiv 1$. By the continuity of $\varsigma(x, \cdot)$, for $\varepsilon$ sufficiently small, we have $t-t_b(t, x, u)\leq -\frac{\varepsilon}{2}$, so as $u'$ and $v'$. Thus, $\alpha_{f, \varepsilon}(t, x, u)\geq 1-\widetilde{\chi}(\frac{1}{2})\gtrsim 1$, so as $u'$ and $v'$. Hence $\alpha_{f, \varepsilon}(v)\lesssim \alpha_{f, \varepsilon}(u)$, so as $u'$ and $v'$.

(ii). $t-t_b(t, x, v)\geq 0$. It holds in this case that $\alpha_{f, \varepsilon}(t,x,v) = |n(x_b)\cdot v_b|$. We consider
\begin{align*}
|n(x_b)\cdot v_b|-|n(x_b')\cdot u_b|\leq|(n(x_b)-n(x_b'))\cdot v_b|+|n(x_b')\cdot (v_b-u_b)|\leq |\nabla n(\xi)||v_b||x_b-x_b'|+ |v_b-u_b|.
\end{align*}
Here $\xi$ lies between $x_b$ and $x_b'$, $x_b'= x'(t-t_b'(t,x,u))$ and $u_b=u(t-t_b'(t,x,u))$, where $x'(s;t,x,u)$ and  $u(s;t,x,u)$ are the characteristics with $x'(t)=x$ and $u(t)=u$, respectively. Note that
\begin{align*}
v(s)=&v-\int_s^t E(\tau, x(\tau))\mathrm{d}\tau,\\
x(s)=&x-(t-s)v + \int_s^t \int^t_{s_1}E(\tau, x(\tau))\mathrm{d}\tau\mathrm{d}s_1.
\end{align*}
It holds that
\begin{align*}
\frac{\partial v(s)}{\partial v}=\mathrm{Id}_{3\times 3}-\int_s^t \nabla E(\tau, x(\tau))\nabla_v x(\tau)\mathrm{d}\tau,
\end{align*}
by Lemma \ref{lemma10}, we have
\begin{align*}
\int_s^t \nabla E(\tau, x(\tau))\nabla_v x(\tau)\mathrm{d}\tau \lesssim \int_s^t\delta_2 e^{-\Lambda_2 \tau}(t-\tau)\mathrm{d}\tau \lesssim 1.
\end{align*}
Hence $\left|\frac{\partial v(s)}{\partial v}\right|\lesssim 1$. By a similar deduction, it holds that $\left|\frac{\partial x(s)}{\partial v}\right| \lesssim |t-s|^3$. Thus,
\begin{align*}
\left|\frac{\partial v_b}{\partial v}\right| \lesssim 1,\quad \left|\frac{\partial x_b}{\partial v}\right| \lesssim |t_b|^3 \lesssim |t|^3+ \varepsilon.
\end{align*}
Hence we have
\begin{align*}
e^{-C|v|^2}(|n(x_b)\cdot v_b|-|n(x_b')\cdot u_b|)\lesssim (1+t^3)|u-v|\lesssim \varepsilon,\,\,\mathrm{for}\,\,t\leq T \,\,\mathrm{finite}.
\end{align*}
Therefore, in this case, it holds that
\begin{align*}
e^{-C|v|^2}\alpha_{f,\varepsilon}(t,x,v)\leq \alpha_{f,\varepsilon}(t,x,u)+O(\varepsilon),
\end{align*}
so as $u'$ and $v'$.

(iii). $-\varepsilon < t-t_b(t,x,v) < 0$. By the uniform continuity of $h(\cdot, \cdot)=t-t_b(t,x,v)$, we have
\begin{align*}
\int_{\Omega \times \mathbb{R}^3}\mathbf{1}_{-\varepsilon < t-t_b < 0}(x,v)\mathrm{d}x\mathrm{d}v= O(\varepsilon).
\end{align*}

We deal with the $O(\varepsilon)$ from the above two cases.
Note that $\alpha_{f,\varepsilon}(t,x,v)\lesssim 1+ |v_b|\lesssim \langle v\rangle$, and $\sqrt{\mu(u)}\leq C \sqrt{\mu(v)}$ for sufficiently small $\varepsilon>0$, then we have
\begin{align*}
\alpha^\beta_{f,\varepsilon}(v)w_{\widetilde{\vartheta}}(v)\sqrt{\mu(u)}\lesssim \mu^{\frac{1}{4}}(u).
\end{align*}
Therefore, the contributions of both of $O(\varepsilon)$ are controlled by
\begin{align*}
O(\varepsilon)\|w_{\widetilde{\vartheta}}f\|_\infty \|\nabla_{x,v}f\|_\infty.
\end{align*}
From the discussion above, we can control the term
\begin{align*}
w_{\widetilde{\vartheta}}\alpha^\beta_{f, \varepsilon}(\Gamma^{1-\chi}_{\mathrm{gain}}(\partial f, f)+ \Gamma^{1-\chi}_{\mathrm{gain}}(f, \partial f)+\Gamma^{1-\chi}_{\mathrm{loss}}(\partial f, f))
\end{align*}
by
\begin{align*}
 &\|w_{\widetilde{\vartheta}}f\|_\infty \int_{\mathbb{R}^3} |u-v|^\gamma (1-\chi)
\left(|w_{\widetilde{\vartheta}}\alpha^\beta_{f, \varepsilon}\partial f|(u')
+ |w_{\widetilde{\vartheta}}\alpha^\beta_{f, \varepsilon}\partial f|(v')
+ |w_{\widetilde{\vartheta}}\alpha^\beta_{f, \varepsilon}\partial f|(u)\right)\mathrm{d}u\\
&+ O(\varepsilon)\|w_{\widetilde{\vartheta}}f\|_\infty \|\nabla_{x,v}f\|_\infty.
\end{align*}
Note that
\begin{align*}
A:= & \int_{\mathbb{R}^3} |u-v|^\gamma (1-\chi)
|w_{\widetilde{\vartheta}}\alpha^\beta_{f, \varepsilon}\partial f|(u')\mathrm{d}u \\
\leq & \left(\int_{\mathbb{R}^3} |u-v|^\gamma (1-\chi)
|w_{\widetilde{\vartheta}}\alpha^\beta_{f, \varepsilon}\partial f|^p (u') \mathrm{d}u\right)^{\frac{1}{p}}\left(\int_{\mathbb{R}^3} |u-v|^\gamma (1-\chi)\mathrm{d}u\right)^{\frac{1}{q}}.
\end{align*}
Thus, by $\mathrm{d}u\mathrm{d}v=\mathrm{d}u'\mathrm{d}v'$, we have
\begin{align*}
\int_{\mathbb{R}^3}|A|^p \mathrm{d}v \lesssim &\varepsilon^{\frac{(\gamma+3)p}{q}}
\int_{\mathbb{R}^3} |u-v|^\gamma (1-\chi)\mathrm{d}v
\int_{\mathbb{R}^3}|w_{\widetilde{\vartheta}}\alpha^\beta_{f, \varepsilon}\partial f|^p (u) \mathrm{d}u\\
\lesssim & \varepsilon^{(\gamma+3)p}\|w_{\widetilde{\vartheta}}\alpha^\beta_{f, \varepsilon}\partial f\|^p_{L^p_u}.
\end{align*}
The other two terms share the same argument, thus the contribution of $\Gamma^{1-\chi}_{\mathrm{gain}}(\partial f, f)+ \Gamma^{1-\chi}_{\mathrm{gain}}(f, \partial f)+\Gamma^{1-\chi}_{\mathrm{loss}}(\partial f, f)$ is $\sup_{0 \leq s\leq t}\|w_{\widetilde{\vartheta}}f\|_\infty \int_0^t (\|w_{\widetilde{\vartheta}}\alpha^\beta_{f, \varepsilon}\partial f\|^p_p+O(\varepsilon^p)\|\nabla_{x,v}f\|^p_\infty)$.

Now we deal with $\Gamma_{\mathrm{gain}}^{1-\chi}(\partial f, f)+ \Gamma_{\mathrm{gain}}^{1-\chi}(f, \partial f)$ with $|u-v|\geq 2 \varepsilon$, which yields $|u-v|^\gamma \lesssim \varepsilon^\gamma$. We also notice that $\alpha_{f, \varepsilon}\lesssim \langle v\rangle$, thus it holds
\begin{align*}
w_{\widetilde{\vartheta}}\alpha^\beta_{f, \varepsilon}\Gamma_{\mathrm{loss}}^\chi(\partial f, f)\leq & \|w_{\frac{4 \vartheta}{s_0}}f\|_\infty e^{-C |v|^2} \int |u-v|^\gamma \chi(|u-v|) \mu^{\frac{1}{4}}(u)w_{\widetilde{\vartheta}}^{-1}(v)|w_{\widetilde{\vartheta}}\partial f(u)|\mathrm{d}u\\
\lesssim & \|w_{\frac{4 \vartheta}{s_0}}f\|_\infty \varepsilon^\gamma e^{-C |v|^2}\int \mu^{\frac{1}{4}}(u)w_{\widetilde{\vartheta}}^{-1}(v) e^{\varrho |v-u|^2}e^{-\varrho |v-u|^2}|w_{\widetilde{\vartheta}}\alpha^\beta_{f, \varepsilon}\partial f(u)|\alpha^{-\beta}_{f, \varepsilon}(u)\mathrm{d}u.
\end{align*}
For $\varrho>0$ sufficiently small, it holds $\mu^{\frac{1}{4}}(u)w_{\widetilde{\vartheta}}^{-1}(v) e^{\varrho |v-u|^2} \lesssim 1$. Note again that $\beta q <1$, we have
\begin{align*}
w_{\widetilde{\vartheta}}\alpha^\beta_{f, \varepsilon}\Gamma_{\mathrm{loss}}^\chi(\partial f, f)\lesssim & e^{-C |v|^2} \|w_{\widetilde{\vartheta}}\alpha^\beta_{f, \varepsilon} \partial f\|_{L^p_u} \left( \int \frac{e^{-\varrho q |v-u|^2}}{\alpha^{\beta q}_{f, \varepsilon}(u)}\mathrm{d}u\right)^{\frac{1}{q}}\\
\lesssim & e^{-C |v|^2} \|w_{\widetilde{\vartheta}}\alpha^\beta_{f, \varepsilon} \partial f\|_{L^p_u}.
\end{align*}
Hence we have $\int |w_{\widetilde{\vartheta}}\alpha^\beta_{f, \varepsilon}\Gamma_{\mathrm{loss}}^\chi(\partial f, f)|^p \mathrm{d}v \lesssim \|w_{\widetilde{\vartheta}}\alpha^\beta_{f, \varepsilon} \partial f\|^p_{L^p_u}$.

For the last term $\Gamma_{\mathrm{loss}}(f, \partial f)$, it is clear that
\begin{align*}
w_{\widetilde{\vartheta}}\alpha^\beta_{f, \varepsilon}\Gamma_{\mathrm{loss}}(f, \partial f)\leq & |w_{\widetilde{\vartheta}}\alpha^\beta_{f, \varepsilon}\partial f(v)|\int |u-v|^\gamma b_0(\theta)\sqrt{\mu(u)}w_{\widetilde{\vartheta}}^{-1}(u)(w_{\widetilde{\vartheta}}f(u))
\mathrm{d}u \mathrm{d}\omega\\
\lesssim & |w_{\widetilde{\vartheta}}\alpha^\beta_{f, \varepsilon}\partial f(v)| \|w_{\widetilde{\vartheta}}f\|_\infty.
\end{align*}

In summary, the contribution of $\Gamma (\partial f, f)+ \Gamma(f, \partial f)$ is
\begin{align*}
\left(1+\sup_{0 \leq s \leq t}\|w_{4 \vartheta /s_0}f\|_\infty\right) \int_0^t \|w_{\widetilde{\vartheta}}\alpha^\beta_{f, \varepsilon}\partial f\|_p^p+ O(\varepsilon)\sup_{0\leq s\leq t}\|w_{\widetilde{\vartheta}}f\|_\infty \int_0^t \|\nabla_{x,v}f\|^p_\infty.
\end{align*}

(6). $K \partial f$.
The estimate of the term $K \partial f$ shares the same way as $\Gamma (\partial f, f)+ \Gamma(f, \partial f)$. Omitting the details for brevity, we obtain that the contribution of $K \partial f$ is
\begin{align*}
\int_0^t \|w_{\widetilde{\vartheta}}\alpha^\beta_{f, \varepsilon}\partial f\|_p^p+O(\varepsilon) \int_0^t \|\nabla_{x,v}f\|^p_\infty.
\end{align*}

(7). $\int_0^t |w_{\widetilde{\vartheta}}\alpha^\beta_{f, \varepsilon}\partial f|_{p, -}^p$.
We consider (\ref{perturb eqn}), namely,
\begin{align*}
[\partial_t + v\cdot \nabla_x - \nabla_x \phi_f \cdot \nabla_v ]f= -Lf + \Gamma(f, f)-\frac{v}{2}\cdot \nabla_x \phi_f f-v \cdot \nabla_x \phi_f \sqrt{\mu}:= G(f),
\end{align*}
with $f(t, x, v)= g(t, x, v)$ on $(x, v)\in \gamma_-$ and $G(f)$ is the source term. By Duhamel's principle, we can write the solution $f$ as
\begin{align*}
f(t, x, v)= g(t-t_b, x_b, v_b)+ \int_{t-t_b}^t G(f)(s, x(s), v(s))\mathrm{d}s.
\end{align*}
Taking $x$ and $v$ derivatives to both sides, we have
\begin{align}\label{Duhamel}
\nabla_{x, v}f(t, x, v)= \nabla_{x, v}[g(t-t_b, x_b, v_b)]+ \int_{t-t_b}^t \nabla_{x, v}[G(f)(s, x(s), v(s))]\mathrm{d}s.
\end{align}
Along the trajectory, as $(x, v)\rightarrow \gamma_-$, we notice that $t_b(t, x ,v)\rightarrow 0$, i.e. $t-t_b \rightarrow t$, $x_b(t, x, v)\rightarrow x$ and $v_b(t, x, v)\rightarrow v$, so the second term on the right-hand side of (\ref{Duhamel}) vanishes. Hence we obtain that
\begin{align*}
|\nabla_{x, v}f(t, x, v)||_{\gamma_-}=& |\nabla_{x, v}[g(t-t_b, x_b, v_b)]|\\
= & |\nabla_{t, x, v}g| \cdot \left|\frac{\partial (t-t_b, x_b, v_b)}{\partial(x, v)}\right|\\
= & |\nabla_{t, x, v}g|\cdot \frac{1}{|n(x_b)\cdot v_b|} \,\,\,(\mathrm{using\,\, Lemma} \,\ref{lemma 3})\\
= & |\nabla_{t, x, v}g|\cdot \frac{1}{|n(x)\cdot v|}
\end{align*}
By the definition of the kinetic weight $\alpha^\beta_{f, \varepsilon}$, we know that $\alpha^\beta_{f, \varepsilon}= |n(x)\cdot v|^\beta$ on $(x, v)\in \gamma_-$. Therefore,
\begin{align*}
\int_0^t |w_{\widetilde{\vartheta}}\alpha^\beta_{f, \varepsilon}\partial f|_{p, -}^p \lesssim \int_0^t \int w_{\widetilde{\vartheta}}^p |\partial g|^p |n(x)\cdot v|^{\beta p-p+1}\mathrm{d}|v_\parallel|\mathrm{d} v_\perp \mathrm{d}S_x \mathrm{d} s,
\end{align*}
where $v_\parallel = \{v\cdot n(x)\}n(x)$, $v_\perp = v- v_\parallel$, $|v_\parallel|=|n(x)\cdot v|$ and $\mathrm{d}v= \mathrm{d}|v_\parallel|\mathrm{d} v_\perp$. We notice that $\beta p-p+1 > \frac{p-2}{p}p-p+1>-1$, so $|n(x)\cdot v|^{\beta p-p+1}= |v_\parallel|^{\beta p-p+1} \in L^1_{\mathrm{loc}}(\mathbb{R}^1_{|v_\parallel|}\times \mathbb{R}^2_{v_\perp})= L^1_{\mathrm{loc}}(\mathbb{R}^3_v)$. Therefore,
\begin{align*}
\int_0^t |w_{\widetilde{\vartheta}}\alpha^\beta_{f, \varepsilon}\partial f|_{p, -}^p \lesssim \int_0^T \|w_{4\vartheta}\partial g\|^p_\infty.
\end{align*}

In conclusion, we have
\begin{align*}
\|w_{\widetilde{\vartheta}}f(t)\|_p^p + \|w_{\widetilde{\vartheta}}\alpha^\beta_{f, \varepsilon}\partial f(t)\|_p^p \lesssim &\|w_{\widetilde{\vartheta}}f(0)\|_p^p + \|w_{\widetilde{\vartheta}}\alpha^\beta_{f, \varepsilon}\partial f(0)\|_p^p + \|w_{\widetilde{\vartheta}}g \|^p_{L^p((0,T)\times \gamma_-)} + \int_0^T \|w_{4\vartheta}\partial g\|^p_\infty \\
& + \left(1+ \sup_{0 \leq s \leq t}\|\nabla^2 \phi\|_\infty + \sup_{0 \leq s \leq t}\|w_{\frac{4\vartheta}{s_0}}f\|_\infty\right) \int_0^t \left(\|w_{\widetilde{\vartheta}}f\|_p^p + \|w_{\widetilde{\vartheta}}\alpha^\beta_{f, \varepsilon}\partial f\|_p^p\right)\\
& + O(\varepsilon) \int_0^t \|\nabla_{x,v}f\|^p_\infty.
\end{align*}
For given $f_0$ and $g$ as the initial datum and the in-flow condition,
by using the condition \eqref{3.3}, we make $O(\varepsilon)$ term be absorbed into the first line on right-hand side above when $\varepsilon$ is sufficiently small. Eventually, we have
\begin{align*}
\|w_{\widetilde{\vartheta}}f(t)\|_p^p + \|w_{\widetilde{\vartheta}}\alpha^\beta_{f, \varepsilon}\partial f(t)\|_p^p \lesssim &\|w_{\widetilde{\vartheta}}f(0)\|_p^p + \|w_{\widetilde{\vartheta}}\alpha^\beta_{f, \varepsilon}\partial f(0)\|_p^p + \|w_{\widetilde{\vartheta}}g \|^p_{L^p((0,T)\times \gamma_-)} + \int_0^T \|w_{4\vartheta}\partial g\|^p_\infty \\
& + \left(1+ \sup_{0 \leq s \leq t}\|\nabla^2 \phi\|_\infty + \sup_{0 \leq s \leq t}\|w_{\frac{4\vartheta}{s_0}}f\|_\infty\right) \int_0^t \left(\|w_{\widetilde{\vartheta}}f\|_p^p + \|w_{\widetilde{\vartheta}}\alpha^\beta_{f, \varepsilon}\partial f\|_p^p\right).
\end{align*}
Using Gronwall's inequality, we finally obtain the desired result and hence complete the proof of Proposition \ref{prop 3}.
\end{proof}

\begin{proposition}\label{prop 4}
Assume that $(f,\phi_f)$ solve the problem \eqref{perturb eqn}--\eqref{incoming}, and satisfy the conditions in Proposition $\ref{prop 3}$. We also assume that
\begin{equation*}
\|w_{4\vartheta}\nabla_v f_0\|_{L^3_{x,v}}<\infty\,\,\,\,\mathrm{and}\,\,\,\,\sup_{0 \leq s \leq t}\|w_{\frac{8 \vartheta}{s_0}}f\|_\infty \ll 1.
\end{equation*}
Then
\begin{equation*}
\|w_{\widetilde{\vartheta}}\nabla_v f(t)\|_{L^3_x(\Omega) L_v^{1+\delta}(\mathbb{R}^3)} \lesssim_t 1, \,\forall \, t\in (0,T).
\end{equation*}

\end{proposition}

\begin{remark}
We shall use the estimates in Proposition $\ref{prop 3}$ to prove this result, so we restrict $t$ in $(0,T)$. Note that as $t \rightarrow +\infty$, it is shown that $\|w_{\widetilde{\vartheta}}\nabla_v f(t)\|_{L^3_x L_v^{1+\delta}}$ do not have a finite upper bound.

\end{remark}

\begin{proof}[Proof of Proposition \ref{prop 4}]
From (\ref{perturb eqn}), we have
\begin{align*}
[\partial_t +v \cdot \nabla_x - \nabla \phi \cdot \nabla_v + \widetilde{\nu}](w_{\widetilde{\vartheta}}\partial_v f)= w_{\widetilde{\vartheta}} \mathcal{H},
\end{align*}
where
\begin{align}\label{H}
\mathcal{H}\equiv & \partial_v (Kf+ \Gamma(f, f))- \nabla \phi \cdot \partial_v (v \sqrt{\mu})-\partial_v v \cdot \nabla_x f- \partial_v \nu f- \partial_v \left(\frac{v}{2}\cdot \nabla \phi\right)f \nonumber\\
\lesssim & |\partial_x f|+ (1+ |\nabla \phi|)|f|+ |\nabla \phi|\mu^{\frac{1}{4}}+ \partial_v (Kf)+ \partial_v (\Gamma(f, f)).
\end{align}
Along the characteristic, it holds
\begin{align*}
\left( \frac{\mathrm{d}}{\mathrm{d}s}+ \widetilde{\nu}\right)(w_{\widetilde{\vartheta}}\partial_v f) = w_{\widetilde{\vartheta}}\mathcal{H},
\end{align*}
and thus
\begin{align*}
\frac{\mathrm{d}}{\mathrm{d}s}\left( e^{\int_t^s \widetilde{\nu}(\tau)\mathrm{d}\tau}w_{\widetilde{\vartheta}}\partial_v f\right)
= e^{\int_t^s \widetilde{\nu}(\tau)\mathrm{d}\tau} w_{\widetilde{\vartheta}}\mathcal{H}.
\end{align*}
Because $e^{\int_t^s \widetilde{\nu}(\tau)\mathrm{d}\tau} \leq 1$ for $s\leq t$, we have
\begin{align}\label{3,1+ norm}
w_{\widetilde{\vartheta}}\partial_v f(t, x, v)\leq \mathbf{1}_{t<t_b}|w_{\widetilde{\vartheta}}\partial_v f (0)|+ \mathbf{1}_{t\geq t_b}|w_{\widetilde{\vartheta}}\partial_v f (t-t_b)| +\int^t_{\max \{t-t_b, 0\}}w_{\widetilde{\vartheta}}\mathcal{H}(s, x(s), v(s))\mathrm{d}s.
\end{align}

We also notice that if $|v|> \frac{2 \delta_1}{\Lambda_1^{\frac{1}{\rho}}}\Gamma(\frac{1}{\rho}+1):= C_1$, where $\Gamma (a), a>0$, denotes the Gamma function, it holds for $0 \leq s \leq t$ that
\begin{align*}
|v(s)|\geq & |v|- \int_s^t |\nabla \phi|(\tau)\mathrm{d}\tau \\
\geq & |v|- \int_s^t \frac{\delta_1}{e^{\Lambda_1 \tau^\rho}}\mathrm{d}\tau \\
\geq & |v|- \frac{ \delta_1}{\Lambda_1^{\frac{1}{\rho}}}\Gamma\left(\frac{1}{\rho}+1\right)\geq \frac{|v|}{2}.
\end{align*}
Hence $w_{\widetilde{\vartheta}}^{-1}(v(s))\lesssim e^{-\widetilde{\vartheta}\frac{|v|^2}{4}} \in L^1 (\{|v|> C_1\})$. On the other hand, $w_{\widetilde{\vartheta}}^{-1}(v(s)) \leq 1 \in L^1(\{|v|\leq C_1\})$. So $w_{\widetilde{\vartheta}}^{-1}(v(s)) \in L^1 (\mathbb{R}^3_v)$. Moreover, it is a direct corollary that
\begin{align*}
\sup_{s, t, x}\|w_{\widetilde{\vartheta}}^{-1}(v(s))\|_{L^r_v} \lesssim 1, \,\,\mathrm{for} \,\,\mathrm{any}\,\, 1\leq r \leq \infty, \,\, 0 \leq s \leq t.
\end{align*}

We shall control $\|w_{\widetilde{\vartheta}}\partial_v f(t, x, v)\|_{L^3_x L^{1+\delta}_v}$ by dealing with the terms on the right-hand side of \eqref{3,1+ norm}, with the careful estimate \eqref{H}.

$(1)$. The contribution of $w_{\widetilde{\vartheta}}\partial_v f(0, x(0), v(0))$ and $w_{\widetilde{\vartheta}}\partial_v f(t-t_b, x_b, v_b)$.
By using H\"{o}lder's inequality to $x-$variable as $\frac{1}{1+\delta}=\frac{1}{3}+ \frac{1}{\frac{3(1+\delta)}{2-\delta}}$, we have
\begin{align*}
\|w_{\widetilde{\vartheta}}\partial_v f(0, x(0), v(0))\|_{L^3_x L^{1+\delta}_v}\leq & \|w_{4\vartheta}\partial_v f(0, x(0), v(0))\|_{L^3_{x, v}} \cdot \|w^{-1}_{2\vartheta}(v(0))\|_{L_v^{\frac{3(1+\delta)}{2-\delta}}}\\
\lesssim & \|w_{4\vartheta}\partial_v f(0, x(0), v(0))\|_{L^3_{x, v}}.
\end{align*}

On the other hand, because the $\mathbb{R}^3$ Lebesgue measure of $\partial \Omega$ is 0, it immediately holds
\begin{align*}
\|w_{\widetilde{\vartheta}}\partial_v f(t-t_b)\|_{L^3_x L^{1+\delta}_v}\equiv 0.
\end{align*}

$(2)$. $w_{\widetilde{\vartheta}}(1+|\nabla \phi|)|f|$.
Due to $|\nabla \phi|\lesssim \|w_{\widetilde{\vartheta}}f\|_\infty$, the contribution of
$w_{\widetilde{\vartheta}}(1+|\nabla \phi|)|f|$ is $t \sup_{0 \leq s \leq t} \|w_{4 \vartheta}f\|_\infty$.

$(3)$. $w_{\widetilde{\vartheta}}|\nabla \phi(s)|\mu^{\frac{1}{4}}(v(s))$.
For $3 < p < 6-2 \varpi$, there exists an $\alpha >1$, such that $\frac{1}{1+\delta}= \frac{1}{p}+ \frac{1}{\alpha}$. By H\"{o}lder's inequality for $v$ and then $x$, and considering $\mathrm{d}x\mathrm{d}v=\mathrm{d}x(s)\mathrm{d}v(s)$ as well, we have
\begin{align*}
\|w_{\widetilde{\vartheta}}|\nabla \phi(s)|\mu^{\frac{1}{4}}(v(s))\|_{L^3_x L^{1+\delta}_v} \lesssim & \|\nabla \phi (s, x(s))\mu^{\frac{1}{16}}(v(s))\|_{L^p_{x, v}}\|\mu^{\frac{1}{16}}(v(s))\|_{L^\alpha_v}\\
\lesssim & \|\nabla \phi(s, x)\|_{L^p_x}\leq \|\phi(s, x)\|_{W^{2, p}_x} \lesssim \|w_{\widetilde{\vartheta}}f(s)\|_{L^p_{x, v}}.
\end{align*}
So the contribution of $w_{\widetilde{\vartheta}}|\nabla \phi|\mu^{\frac{1}{4}}$ is $\int_0^t \|w_{\widetilde{\vartheta}}f(s)\|_{L^p_{x, v}} \mathrm{d}s$.

$(4)$. $w_{\widetilde{\vartheta}}\partial_x f(s, x(s), v(s))$.
Similar to the method used in \cite{Kim}, we can obtain without difficulty that the contribution is $\int_0^t \|w_{2\widetilde{\vartheta}}\alpha^\beta_{f, \varepsilon}\partial_x f(s)\|_p \mathrm{d}s$.

$(5)$. $w_{\widetilde{\vartheta}}(\partial_v (Kf)+ \partial_v (\Gamma (f, f)))$.
First, we deal with the term $w_{\widetilde{\vartheta}}K_2^\chi \nabla_v f$.
It holds
\begin{align*}
w_{\widetilde{\vartheta}}K_2^\chi \nabla_v f=\int_{\mathbb{R}^3}w_{\widetilde{\vartheta}}(v(s))\mathbf{k}_2^\chi (v(s), u)\nabla_v f(u)\mathrm{d}u,
\end{align*}
and recalling (\ref{k2x}) for $s_0 =\min \{s_1, s_2\}$ gives
\begin{align*}
|\mathbf{k}_2^\chi (v,u)|\lesssim \frac{\exp\left( -\frac{s_0}{8}|u-v|^2- \frac{s_0}{8}\frac{(|v|^2-|u|^2)^2}{|v-u|^2}\right)}{|v-u|(1+|v|+|u|)^{1-\gamma}} \leq \frac{\exp\left( -\frac{s_0}{8}|u-v|^2\right)}{|v-u|}.
\end{align*}

We need to divide $\mathbb{R}_v^3 \times \mathbb{R}_u^3$ into the following three cases to discuss.

\emph{Case 1}. For $C_1 := \frac{2 \delta_1}{\Lambda_1^{\frac{1}{\rho}}}\Gamma(\frac{1}{\rho}+1)$, if $|v|\geq 2 N + C_1$ and $|u|\leq N$, we know that $|v(s)|\geq 2N$ and thus $|v(s)-u|\geq \frac{|v(s)|}{2}$. Then we have
\begin{align*}
w_{\widetilde{\vartheta}}(v(s))e^{-\frac{s_0}{8}|v(s)-u|^2}\leq & \exp\left\{ \widetilde{\vartheta}|v(s)|^2-\frac{s_0}{16}|v(s)-u|^2-\frac{s_0}{64}|v(s)|^2\right\}\\
\lesssim & \exp\left\{ -\frac{s_0}{16}|v(s)-u|^2 - \frac{s_0}{256}|v(s)|^2 - \frac{s_0}{256}|v(s)|^2\right\}.
\end{align*}
Hence
\begin{align*}
\int_{\mathbb{R}^3}w_{\widetilde{\vartheta}}(v(s))\mathbf{k}_2^\chi (v(s), u)\nabla_v f(u)\mathrm{d}u \lesssim e^{-2C_3 |v(s)|^2}\int_{\mathbb{R}^3} \frac{e^{-C_2 |v(s)-u|^2}}{|v(s)-u|}\nabla_v f(s, x(s),u)\mathrm{d}u,
\end{align*}
where $C_2 = \frac{s_0}{16}$ and $C_3 = \frac{s_0}{256}$. Therefore, by H\"{o}lder's inequality for $\frac{1}{1+\delta}= \frac{1}{3}+ \frac{1}{\frac{3(1+\delta)}{2-\delta}}$, it holds
\begin{align*}
\|w_{\widetilde{\vartheta}}K_2^\chi \nabla_v f\|_{L^3_x L^{1+\delta}_v} \lesssim & \left\|e^{-C_3 |v(s)|^2}\right\|_{L_v^{\frac{3(1+\delta)}{2-\delta}}} \left\|\int_{\mathbb{R}^3_u} \frac{e^{-C_3 |v(s)|^2} e^{-C_2 |v(s)-u|^2}}{|v(s)-u|}\nabla_v f(s, x(s),u)\mathrm{d}u\right\|_{L^3_{x, v}}\\
\lesssim & \left\|\int_{\mathbb{R}^3_u} \frac{e^{-C_3 |v|^2} e^{-C_2 |v-u|^2}}{|v-u|}\nabla_v f(s, x,u)\mathrm{d}u\right\|_{L^3_{x, v}},
\end{align*}
where we have used $\mathrm{d}x\mathrm{d}v=\mathrm{d}x(s)\mathrm{d}v(s)$.
Considering $\frac{1}{3}= \frac{1}{\frac{3(1+\delta)}{1-2\delta}}+ \frac{1}{\frac{1+\delta}{\delta}}$, it holds
\begin{align*}
&\left\|\int_{\mathbb{R}^3_u} \frac{e^{-C_3 |v|^2} e^{-C_2 |v-u|^2}}{|v-u|}\nabla_v f(s, x,u)\mathrm{d}u\right\|_{L^3_{v}}\\
\leq & \left\|e^{-C_3 |v|^2}\right\|_{L_v^{\frac{1+\delta}{\delta}}} \left\|\int_{\mathbb{R}^3_u} \frac{ e^{-C_2 |v-u|^2}}{|v-u|}\nabla_v f(s, x,u)\mathrm{d}u\right\|_{L_v^{\frac{3(1+\delta)}{1-2\delta}}}\\
\lesssim & \left\|\frac{e^{-C_2 |v|^2}}{|v|}\ast \nabla_v f(s, x,v)\right\|_{L_v^{\frac{3(1+\delta)}{1-2\delta}}}\\
\lesssim & \left\|\frac{e^{-C_2 |v|^2}}{|v|}\right\|_{L^3(|v|\geq 2N+C_1)} \|\nabla_v f(s, x, v)\|_{L_v^{1+\delta}} \,\,\left(\mathrm{since}\,\, \,1+ \frac{1}{\frac{3(1+\delta)}{1-2\delta}}=\frac{1}{3}+\frac{1}{1+\delta}\right)\\
\lesssim & \|w_{\widetilde{\vartheta}}\nabla_v f(s, x, v)\|_{L_v^{1+\delta}}.
\end{align*}
So in this case, $\|w_{\widetilde{\vartheta}}K_2^\chi \nabla_v f\|_{L^3_x L^{1+\delta}_v} \lesssim \|w_{\widetilde{\vartheta}} \nabla_v f(s)\|_{L^3_x L^{1+\delta}_v}$.

\emph{Case 2}. If $|v|\leq 2N+C_1$, then $|v(s)|\leq |v|+C_1 \leq 2N+2C_1$,  $w_{\widetilde{\vartheta}}(v(s))\lesssim 1$. As $|v|$ is bounded, by H\"{o}lder's inequality, we have $\|\cdot\|_{L^{1+\delta}_v}\lesssim \|\cdot\|_{L^{3}_v}$. So we can use $\mathrm{d}x\mathrm{d}v=\mathrm{d}x(s)\mathrm{d}v(s)$ again. We take some $\alpha$ such that
\begin{align*}
1+\frac{1}{\frac{\alpha (1+\delta)}{1+\delta-\alpha\delta}}=\frac{1}{\alpha}+ \frac{1}{1+\delta}, \,\,\, 0<3-\alpha \ll 1, \,\,\,\frac{\alpha (1+\delta)}{1+\delta-\alpha\delta} >3.
\end{align*}
Note that
\begin{align*}
\left\|\int_{\mathbb{R}_u^3} \mathbf{k}_2^\chi (v, u)\nabla_v f(u)\right\|_{L_v^{\frac{\alpha (1+\delta)}{1+\delta-\alpha\delta}}} \leq &\left\|\frac{e^{-C|v|^2}}{|v|}\ast \nabla_v f(v)\right\|_{L_v^{\frac{\alpha (1+\delta)}{1+\delta-\alpha\delta}}}\\
\leq & \left\|\frac{e^{-C|v|^2}}{|v|}\right\|_{L_v^\alpha} \left\|\nabla_v f\right\|_{L^{1+\delta}_v}\\
\lesssim & \left\|w_{\widetilde{\vartheta}}\nabla_v f(s)\right\|_{L^{1+\delta}_v}.
\end{align*}
So in this case, we also have $\|w_{\widetilde{\vartheta}}K_2^\chi \nabla_v f\|_{L^3_x L^{1+\delta}_v} \lesssim \|w_{\widetilde{\vartheta}} \nabla_v f(s)\|_{L^3_x L^{1+\delta}_v}$.

\emph{Case 3}. $|v|\geq 2N+C_1 \geq 2N$ and $|u|\geq N$.
We go through a similar procedure to the hard sphere case in \cite{Kim}. In soft potential case, we have
\begin{align*}
&w_{\widetilde{\vartheta}}(v(s))\mathbf{k}_2^\chi(v(s), u)\nabla_v f(s, x(s), u)\\
=& \frac{w_{2\widetilde{\vartheta}}(v(s))}{w_{2\widetilde{\vartheta}}(u)} \frac{\mathbf{k}_2^\chi(v(s), u)}{\alpha^\beta_{f, \varepsilon}(u)}\frac{w_{2\widetilde{\vartheta}}(u)\alpha^\beta_{f, \varepsilon}(u)\nabla_v f(s, x(s), u)}{w_{\widetilde{\vartheta}}(v(s))^{1-l}w_{\widetilde{\vartheta}}(v(s))^{l}},
\end{align*}
where $0<l \ll 1$. Then we are able to find some $\widetilde{\rho}< \frac{s_0}{8}-\widetilde{\vartheta}$ such that
\begin{align*}
\mathbf{k}_2^\chi(v(s), u)\frac{w_{2\widetilde{\vartheta}}(v(s))}{w_{2\widetilde{\vartheta}}(u)} \lesssim \mathbf{k}_{\widetilde{\rho}}(v(s), u)\leq \frac{e^{-\widetilde{\rho}|v(s)-u|^2}}{|v(s)-u|}.
\end{align*}
Omitting the details for brevity, we finally obtain that $\|w_{\widetilde{\vartheta}}K_2^\chi \nabla_v f\|_{L^3_x L^{1+\delta}_v} \lesssim \big\|w_{2\widetilde{\vartheta}}\alpha^\beta_{f, \varepsilon}\nabla_v f(s)\big\|_p$.

Hence the contribution of $w_{\widetilde{\vartheta}}K_2^\chi \nabla_v f$ is
\begin{align*}
\int_0^t \|w_{\widetilde{\vartheta}} \nabla_v f(s)\|_{L^3_x L^{1+\delta}_v}+\big\|w_{2\widetilde{\vartheta}}\alpha^\beta_{f, \varepsilon}\nabla_v f(s)\big\|_p.
\end{align*}

Next, we deal with the term $w_{\widetilde{\vartheta}}(\Gamma^\chi_{\mathrm{gain}}(\nabla_v f, f)+\Gamma^\chi_{\mathrm{gain}}(f, \nabla_v f))$.
Note that
\begin{align*}
&\Gamma^\chi_{\mathrm{gain}}(\nabla_v f, f)+\Gamma^\chi_{\mathrm{gain}}(f, \nabla_v f)\\
\lesssim & \Big\|w_{\frac{4\widetilde{\vartheta}}{s_0}}f\Big\|_\infty \int |u-v|^\gamma \chi b_0(\theta)w_{\frac{4\widetilde{\vartheta}}{s_0}}^{-1}(u)
(w_{\frac{4\widetilde{\vartheta}}{s_0}}^{-1}(v')\nabla_v f(u')+ w_{\frac{4\widetilde{\vartheta}}{s_0}}^{-1}(u')\nabla_v f(v'))\mathrm{d}u\mathrm{d}\omega\\
\lesssim &\Big\|w_{\frac{4\widetilde{\vartheta}}{s_0}}f\Big\|_\infty \int \mathbf{k}_{2\widetilde{\vartheta}}(v, u)\nabla_v f(u)\mathrm{d}u,
\end{align*}
where
\begin{align*}
\mathbf{k}_{2\widetilde{\vartheta}}(v, u) \leq \frac{\exp\left\{-2\widetilde{\vartheta}|u-v|^2- 2\widetilde{\vartheta}\frac{(|v|^2-|u|^2)^2}{|v-u|^2}\right\}}{|v-u|}
\end{align*}
and there exists a $\widetilde{\rho}< 2 \widetilde{\vartheta}- \widetilde{\vartheta}$, such that
\begin{align*}
\mathbf{k}_{2\widetilde{\vartheta}}(v(s), u)\frac{w_{2\widetilde{\vartheta}}(v(s))}{w_{2\widetilde{\vartheta}}(u)} \lesssim \mathbf{k}_{\widetilde{\rho}}(v(s), u).
\end{align*}
Then we can take the same procedure for $K_2^\chi \nabla_v f$. So the contribution of  $w_{\widetilde{\vartheta}}(\Gamma^\chi_{\mathrm{gain}}(\nabla_v f, f)+\Gamma^\chi_{\mathrm{gain}}(f, \nabla_v f))$ is
\begin{align*}
\sup_{0 \leq s \leq t}\Big\|w_{\frac{4\widetilde{\vartheta}}{s_0}}f\Big\|_\infty \int_0^t \|w_{\widetilde{\vartheta}} \nabla_v f(s)\|_{L^3_x L^{1+\delta}_v}+\big\|w_{2\widetilde{\vartheta}}\alpha^\beta_{f, \varepsilon}\nabla_v f(s)\big\|_p.
\end{align*}

We point out that there is no essential difficulty or new method needed to deal with the other terms of $w_{\widetilde{\vartheta}}(\partial_v (Kf)+ \partial_v (\Gamma (f, f)))$. Actually, by using the techniques of the proofs in Proposition \ref{prop 3}, we can successfully control the following terms, so we shall only state the results and omit the details.

The contribution of $w_{\widetilde{\vartheta}}(K_2^{1-\chi}\nabla_v f+ \Gamma_{\mathrm{gain}}^{1-\chi}(\nabla_v f, f)+\Gamma_{\mathrm{gain}}^{1-\chi}(f, \nabla_v f))$ is
$(1+\sup_{0 \leq s \leq t}\|w_{\widetilde{\vartheta}}f\|_\infty ) \int_0^t \|w_{\widetilde{\vartheta}}\alpha^\beta_{f, \varepsilon}\nabla_v f(s)\|_p$.
Note that the term $O(\varepsilon)$ is also absorbed into initial data as $\varepsilon$ is sufficiently small.

The contribution of $w_{\widetilde{\vartheta}}(K_1 \nabla_v f +\Gamma_{\mathrm{loss}}(\nabla_v f, f))$ is $(1+\sup_{0 \leq s \leq t}\|w_{2\widetilde{\vartheta}}f\|_\infty ) \int_0^t \|w_{\widetilde{\vartheta}}\alpha^\beta_{f, \varepsilon}\nabla_v f(s)\|_p$.

The contribution of $w_{\widetilde{\vartheta}}(\Gamma_{\mathrm{loss}}(f, \nabla_v f))$ is $\sup_{0 \leq s \leq t}\|w_{\widetilde{\vartheta}}f\|_\infty \int_0^t \|w_{2\widetilde{\vartheta}}\alpha^\beta_{f, \varepsilon}\nabla_v f(s)\|_p$.

The contribution of $w_{\widetilde{\vartheta}}(K_v f +\Gamma_v(f, f))$ is $t(1+\sup_{0 \leq s \leq t}\|w_{\frac{4\widetilde{\vartheta}}{s_0}}f\|_\infty ) \|w_{2\widetilde{\vartheta}}f\|_\infty$.

Therefore, in summary, we have
\begin{align*}
\|w_{\widetilde{\vartheta}}\nabla_v f(t)\|_{L^3_x L^{1+\delta}_v} \lesssim & \|w_{4\vartheta}\nabla_v f_0\|_{L^3_{x, v}}\\
& + t\left(1+\sup_{0 \leq s \leq t}\|w_{\frac{4\widetilde{\vartheta}}{s_0}}f\|_\infty \right) \sup_{0 \leq s\leq t}\left(\|w_{4\vartheta}f\|_\infty+ \|w_{\widetilde{\vartheta}}f\|_p + \|w_{2\widetilde{\vartheta}}\alpha^\beta_{f, \varepsilon}\nabla_{x, v} f\|_p\right)\\
&+ \left(1+\sup_{0 \leq s \leq t}\|w_{\frac{4\widetilde{\vartheta}}{s_0}}f\|_\infty \right) \int_0^t \|w_{\widetilde{\vartheta}}\nabla_v f(s)\|_{L^3_x L^{1+\delta}_v}.
\end{align*}
Then by Proposition \ref{prop 3} and Gronwall's inequality, we finally obtain that
$\|w_{\widetilde{\vartheta}}\nabla_v f(t)\|_{L^3_x L^{1+\delta}_v} \lesssim_t 1$.
\end{proof}

Once we have Proposition \ref{prop 4}, we can obtain the following $L^{1+}$ stability for the solutions to the problem \eqref{perturb eqn}--\eqref{incoming}.

\begin{proposition}\label{prop 5}
Suppose that $(f,\phi_f)$ and $(g,\phi_g)$ solve the problem \eqref{perturb eqn}--\eqref{incoming}, with $f|_{\gamma_-}=h_1, g|_{\gamma_-}=h_2 $, and satisfy all the conditions in Propositions \emph{\ref{prop 3}} and \emph{\ref{prop 4}}. We also assume that $w_{\widetilde{\vartheta}}(h_1-h_2) \in L^{1+\delta} ((0,T)\times \gamma_-)$. Then
\begin{align*}
&\|w_{\widetilde{\vartheta}}(f-g)(t)\|^{1+\delta}_{L^{1+\delta}(\Omega \times \mathbb{R}^3)} + \int_0^t \|w_{\widetilde{\vartheta}}(f-g)\|^{1+\delta}_{L^{1+\delta}(\gamma_+)} \mathrm{d}s \\
& \lesssim_t \|w_{\widetilde{\vartheta}}(f-g)(0)\|^{1+\delta}_{L^{1+\delta}(\Omega \times \mathbb{R}^3)}+ \|w_{\widetilde{\vartheta}}(h_1-h_2)\|^{1+\delta}_{L^{1+\delta}((0,T)\times \gamma_-)}.
\end{align*}
\end{proposition}

\begin{proof}
The weighted linearized equation of \eqref{perturb eqn} for the solution $f$ and $g$ are:
\begin{align*}
\{\partial_t + v \cdot \nabla_x - \nabla_x \phi_f \cdot \nabla_v + \widetilde{\nu}_f\}(w_{\widetilde{\vartheta}}f)=& w_{\widetilde{\vartheta}}(Kf+ \Gamma(f, f)- v \cdot \nabla_x \phi_f \sqrt{\mu}),\\
\{\partial_t + v \cdot \nabla_x - \nabla_x \phi_g \cdot \nabla_v + \widetilde{\nu}_g\}(w_{\widetilde{\vartheta}}g)=& w_{\widetilde{\vartheta}}(Kg+ \Gamma(g, g)- v \cdot \nabla_x \phi_g \sqrt{\mu}).
\end{align*}
Taking the difference between the both equations, we have
\begin{align}\label{f-g}
&\{\partial_t + v \cdot \nabla_x - \nabla_x \phi_f \cdot \nabla_v + \widetilde{\nu}_f\}(w_{\widetilde{\vartheta}}(f-g))\nonumber\\
=& w_{\widetilde{\vartheta}}(K[f-g]+\Gamma(f, f)-\Gamma(g, g)-v \cdot \nabla \phi_{f-g}\sqrt{\mu})+ \frac{v}{2} \cdot \nabla \phi_{g-f}w_{\widetilde{\vartheta}}g + w_{\widetilde{\vartheta}}\nabla \phi_{f-g} \cdot\nabla_v g.
\end{align}
With the above equation for $w_{\widetilde{\vartheta}}(f-g)$, we can deduce the corresponding inequality for $|w_{\widetilde{\vartheta}}(f-g)|^{1+\delta}$ by taking integration over $(0, t)\times \Omega \times \mathbb{R}^3$ and integrating by parts. It holds
\begin{align*}
&\|w_{\widetilde{\vartheta}}(f-g)(t)\|_{1+\delta}^{1+\delta}+ \int_0^t |w_{\widetilde{\vartheta}}(f-g)|_{1+\delta, +}^{1+\delta}+ \int_0^t \widetilde{\nu}_f |w_{\widetilde{\vartheta}}(f-g)|^{1+\delta}\mathrm{d}x\mathrm{d}v\mathrm{d}s\\
\leq & \|w_{\widetilde{\vartheta}}(f-g)(0)\|_{1+\delta}^{1+\delta}+ \int_0^t |w_{\widetilde{\vartheta}}(f-g)|_{1+\delta, -}^{1+\delta} +\int_0^t |w_{\widetilde{\vartheta}}(f-g)|^{\delta}\cdot |\mathrm{RHS} \,\,\mathrm{of}\,\,\eqref{f-g} |\mathrm{d}x\mathrm{d}v\mathrm{d}s.
\end{align*}
Note again that $\widetilde{\nu}_f \geq 0$, hence the above inequality can be immediately simplified, namely,
\begin{align*}
&\|w_{\widetilde{\vartheta}}(f-g)(t)\|_{1+\delta}^{1+\delta}+ \int_0^t |w_{\widetilde{\vartheta}}(f-g)|_{1+\delta, +}^{1+\delta}\\
\leq & \|w_{\widetilde{\vartheta}}(f-g)(0)\|_{1+\delta}^{1+\delta}+ \int_0^t |w_{\widetilde{\vartheta}}(f-g)|_{1+\delta, -}^{1+\delta} +\int_0^t |w_{\widetilde{\vartheta}}(f-g)|^{\delta}\cdot |\mathrm{RHS} \,\,\mathrm{of}\,\,\eqref{f-g} |\mathrm{d}x\mathrm{d}v\mathrm{d}s.
\end{align*}

We shall deal with the contribution of the right-hand side of \eqref{f-g} term by term.

(1). $w_{\widetilde{\vartheta}}(K[f-g]+\Gamma(f, f)-\Gamma(g, g)-v \cdot \nabla \phi_{f-g}\sqrt{\mu})$.
For the term $\Gamma_{\mathrm{loss}}(f-g, g)$, we have
\begin{align*}
w_{\widetilde{\vartheta}}(v)\Gamma_{\mathrm{loss}}(f-g, g)
=& w_{\widetilde{\vartheta}}(v)\int_{\mathbb{R}^3 \times \mathbb{S}^2}|u-v|^\gamma b_0(\theta)\sqrt{\mu(u)}(f-g)(u)g(v)\mathrm{d}u \mathrm{d}\omega\\
\lesssim & \Big\|w_{\frac{4 \vartheta}{s_0}}g\Big\|_\infty \int_{\mathbb{R}^3}|u-v|^\gamma \mu^{\frac{1}{2}}(u)w_{\vartheta}^{-1}(v)|w_{\widetilde{\vartheta}}(f-g)(u)|\mathrm{d}u.
\end{align*}
We notice that
\begin{align*}
\int_{\mathbb{R}^3}|u-v|^\gamma \mu^{\frac{1}{2}}(u)w_{\vartheta}^{-1}(v) \mathrm{d}v \lesssim & \langle u\rangle^\gamma \lesssim 1,\\
\int_{\mathbb{R}^3}|u-v|^\gamma \mu^{\frac{1}{2}}(u)w_{\vartheta}^{-1}(v) \mathrm{d}u \lesssim & \langle v\rangle^\gamma \lesssim 1.
\end{align*}
Similar to the deduction in Proposition \ref{prop 3}, we obtain that the contribution of $\Gamma_{\mathrm{loss}}(f-g, g)$ is
\begin{align*}
\sup_{0 \leq s \leq t}\Big\|w_{\frac{4 \vartheta}{s_0}}g\Big\|_\infty \int_0^t \|w_{\widetilde{\vartheta}}(f-g)\|_{1+\delta}^{1+\delta}.
\end{align*}
Meanwhile, the contribution of the other terms therein also shares the same deduction with Proposition \ref{prop 3}, thus the total contribution of $w_{\widetilde{\vartheta}}(K[f-g]+\Gamma(f, f)-\Gamma(g, g)-v \cdot \nabla \phi_{f-g}\sqrt{\mu})$ is
\begin{align*}
\left(1+\sup_{0 \leq s \leq t}\Big\|w_{\frac{4 \vartheta}{s_0}}f\Big\|_\infty+ \sup_{0 \leq s \leq t}\Big\|w_{\frac{4 \vartheta}{s_0}}g\Big\|_\infty\right) \int_0^t \|w_{\widetilde{\vartheta}}(f-g)\|_{1+\delta}^{1+\delta}.
\end{align*}

(2). $\frac{v}{2} \cdot \nabla \phi_{g-f}w_{\widetilde{\vartheta}}g$.
We have
\begin{align*}
\|\frac{v}{2} \cdot \nabla \phi_{g-f}w_{\widetilde{\vartheta}}g\|_{L^{1+\delta}_{x, v}}
\lesssim & \Big\|w_{\frac{4 \vartheta}{s_0}}g\Big\|_\infty \left(\int_x |\nabla \phi_{f-g}|^{1+\delta} \mathrm{d}x \int_v w_{\widetilde{\vartheta}}^{-(1+\delta)}(v)\mathrm{d}v\right)^{\frac{1}{1+\delta}}\\
\lesssim & \Big\|w_{\frac{4 \vartheta}{s_0}}g\Big\|_\infty \|\nabla \phi_{f-g}\|_{L^{1+\delta}_x}.
\end{align*}
Due to
\begin{align*}
\|\nabla \phi_{f-g}\|_{L^{1+\delta}_x} \leq \|\phi_{f-g}\|_{W^{2, 1+\delta}_x} \lesssim \left\|\int (f-g)\sqrt{\mu}\mathrm{d}v\right\|_{L^{1+\delta}_x} \lesssim \|w_{\widetilde{\vartheta}}(f-g)\|_{L^{1+\delta}_{x, v}},
\end{align*}
the contribution of $\frac{v}{2} \cdot \nabla \phi_{g-f}w_{\widetilde{\vartheta}}g$ is
\begin{align*}
\sup_{0 \leq s \leq t}\Big\|w_{\frac{4 \vartheta}{s_0}}g\Big\|_\infty \int_0^t \|w_{\widetilde{\vartheta}}(f-g)\|_{1+\delta}^{1+\delta}.
\end{align*}

(3). $w_{\widetilde{\vartheta}}\nabla \phi_{f-g} \cdot\nabla_v g$.
Inspired by \cite{Kim}, considering
\begin{align*}
1= \frac{1}{\frac{3(1+\delta)}{2-\delta}}+ \frac{1}{3}+\frac{1}{\frac{1+\delta}{\delta}}, \,\, 1= \frac{1}{1+\delta}+\frac{1}{\frac{1+\delta}{\delta}},
\end{align*}
and using H\"{o}lder's inequality, we deduce that
\begin{align*}
&\int_0^t \int_{\Omega \times \mathbb{R}^3}|w_{\widetilde{\vartheta}}(f-g)|^\delta |\nabla \phi_{f-g}||w_{\widetilde{\vartheta}}\nabla_v g|\\
\lesssim & \int_0^t \|\nabla \phi_{f- g}\|_{L_x^{\frac{3(1+\delta)}{2-\delta}}} \|w_{\widetilde{\vartheta}}\nabla_v g\|_{L_x^3 L_v^{1+\delta}} \|w_{\widetilde{\vartheta}}(f-g)\|_{L_{x, v}^{1+\delta}}^\delta.
\end{align*}
Thanks to the Gagliardo-Nirenberg-Sobolev inequality, it holds
\begin{align*}
\|\nabla \phi_{f- g}\|_{L_x^{\frac{3(1+\delta)}{2-\delta}}}\lesssim \|\nabla \phi_{f- g}\|_{W^{1, 1+\delta}_x} \lesssim \|w_{\widetilde{\vartheta}}(f-g)\|_{L^{1+\delta}_{x, v}}.
\end{align*}
Hence the contribution of $w_{\widetilde{\vartheta}}\nabla \phi_{f-g} \cdot\nabla_v g$ is
\begin{align*}
\sup_{0 \leq s \leq t}\|w_{\widetilde{\vartheta}}\nabla_v g\|_{L_x^3 L_v^{1+\delta}} \int_0^t \|w_{\widetilde{\vartheta}}(f-g)\|_{1+\delta}^{1+\delta}.
\end{align*}

In summary, we have
\begin{align*}
&\|w_{\widetilde{\vartheta}}(f-g)(t)\|_{1+\delta}^{1+\delta}+\int_0^t |w_{\widetilde{\vartheta}}(f-g)|_{1+\delta, +}^{1+\delta}\\
\lesssim & \|w_{\widetilde{\vartheta}}(f-g)(0)\|_{1+\delta}^{1+\delta}+ \int_0^t |w_{\widetilde{\vartheta}}(h_1-h_2)|_{1+\delta, -}^{1+\delta} \\
& + \left(1+\sup_{0 \leq s \leq t}\Big\|w_{\frac{4 \vartheta}{s_0}}f\Big\|_\infty+ \sup_{0 \leq s \leq t}\Big\|w_{\frac{4 \vartheta}{s_0}}g\Big\|_\infty + \sup_{0 \leq s \leq t}\|w_{\widetilde{\vartheta}}\nabla_v g\|_{L_x^3 L_v^{1+\delta}}\right) \int_0^t \|w_{\widetilde{\vartheta}}(f-g)\|_{1+\delta}^{1+\delta}.
\end{align*}
By using Gronwall's inequality, we obtain the desired result.
\end{proof}

\section{Local Existence}
\begin{theorem}
Let $0<\vartheta \ll 1$. Assume that for sufficiently small $M, \delta^* >0, \vartheta\theta > \frac{8}{s_0}M$, the initial datum $F_0=\mu + \sqrt{\mu}f_0 \geq 0$ and the boundary datum $F|_{\gamma_-}=\mu +\sqrt{\mu} g \geq 0$ satisfy $\|w_{\widetilde{\vartheta}}f_0\|_{\infty}+ \sup_{0\leq s\leq \infty}\|w_{\widetilde{\vartheta}}g(s)\|_{\infty}\leq M/2$, and
\begin{align*}
\|\nabla^2 \phi(0)\|_{\infty}+ \|w_\vartheta f(0)\|_p+ \big\|w_\vartheta \alpha^\beta_{f_0, \varepsilon}\nabla_{x, v}f(0)\big\|_p+ \|w_\vartheta g\|_{L^p_{s, x, v}} +\|w_\vartheta \nabla_{x, v}g\|_{L^p ((0, \infty); L_{x, v}^\infty )} \leq \delta^* M,
\end{align*}
where $(p, \beta)$ satisfies \eqref{p beta}, and $\|w_\vartheta \nabla_v f_0\|_{L^3_{x, v}}< \infty$.

Then there exists a $\widehat{T}>0$ and a unique solution $F(t,x,v)=\mu +\sqrt{\mu}f(t,x,v)\geq 0$ to the problem \eqref{perturb eqn}--\eqref{incoming} in $[0, \widehat{T}]\times \Omega \times \mathbb{R}^3$ such that
\begin{equation*}
\sup_{0\leq t \leq \widehat{T}}\|w_{\widetilde{\vartheta}}f(t)\|_{\infty} \leq M.
\end{equation*}
Moreover, for $0\leq t\leq \widehat{T}$, it holds that
\begin{align*}
\Big\|w_{\frac{s_0\widetilde{\vartheta}}{4}}\alpha^\beta_{f,\varepsilon}\nabla_{x,v}f(t)\Big\|_p< \infty,\quad \Big\|w_{\frac{s_0\widetilde{\vartheta}}{8}}\nabla_v f(t)\Big\|_{L^3_x L_v^{1+\delta}}<\infty,
\end{align*}
and
\begin{align*}
\|\nabla_{x,v}f(t)\|_\infty < \infty.
\end{align*}
Furthermore, $\|w_{\widetilde{\vartheta}}f(t)\|_{\infty}, \Big\|w_{\frac{s_0\widetilde{\vartheta}}{8}}\nabla_v f(t)\Big\|_{L^3_x L_v^{1+\delta}}$ and $\Big\|w_{\frac{s_0 \widetilde{\vartheta}}{4}}\alpha^\beta_{f,\varepsilon}\nabla_{x,v}f(t)\Big\|_p^p$ are continuous in $t$.
\end{theorem}

\begin{proof}
We consider a sequence for $\ell \in \mathbb{N}$,
\begin{align}
& \partial_t F^{\ell+1}+v \cdot \nabla_x F^{\ell+1}-\nabla \phi^\ell \cdot \nabla_v F^{\ell+1}                  = Q_{\mathrm{gain}}(F^\ell , F^\ell)-Q_{\mathrm{loss}}(F^\ell, F^{\ell+1}),\\
& -\Delta \phi^\ell = \int_{\mathbb{R}^3}F^\ell \mathrm{d}v-\rho_0,\quad \frac{\partial \phi^\ell}{\partial n}\bigg|_{\partial \Omega}=0,
\end{align}
where $F^{\ell+1}(0,x,v)=F_0(x,v)$ for $\ell \in \mathbb{N}$ and $F^\ell (t,x,v)\mid_{\gamma_-}=\mu+\sqrt{\mu}g(t,x,v)\geq 0$. We set $F^0(t,x,v)\equiv \mu, \phi^0 \equiv 0$.

We shall prove the main results through the following steps.

\emph{Step 1}. The proofs of the positivity, i.e. $F^\ell \geq 0, \ell \in \mathbb{N}$. It is similar to \cite{Kim}, we omit the details for brevity.

\emph{Step 2}. In this step, we claim that by choosing $M \ll 1$ and $\widehat{T}=\widehat{T}(M) \ll 1$, it holds
\begin{equation}\label{L infty bdd}
\sup_{0 \leq s \leq \widehat{T}}\max_\ell \big\|w_{\widetilde{\vartheta}}f^\ell (s)\big\|_{\infty} \leq M.
\end{equation}

We denote that
\begin{equation*}
h^\ell (t,x,v):= w_{\widetilde{\vartheta}}(v)f^\ell (t,x,v), \,\,\ell \in \mathbb{N}.
\end{equation*}
Since $h^0=w_{\widetilde{\vartheta}} f^0 \equiv 0$, it is clear that $\sup_{0\leq s\leq \widehat{T}}
\|h^0(s)\|_{\infty}\equiv 0 \leq M$.

For $\ell \geq 1$, we inductively assume that
\begin{equation*}
\sup_{0\leq s\leq \widehat{T}}\|h^\ell(s)\|_{\infty}\leq M,
\end{equation*}
and we shall prove it holds for $\ell+1$ below.

First of all, $h^{\ell+1}$ solves
\begin{align}\label{h l+1}
&\left[\partial_t + v \cdot \nabla_x -\nabla{\phi^\ell}\cdot \nabla_v +\nu + \frac{v}{2}\cdot \nabla \phi^\ell +  \frac{\nabla \phi^\ell \cdot \nabla_v w_{\widetilde{\vartheta}}}{w_{\widetilde{\vartheta}}} - \frac{\partial_t w_{\widetilde{\vartheta}}}{w_{\widetilde{\vartheta}}}                                                                \right]h^{\ell+1}\nonumber\\
&= K_{w_{\widetilde{\vartheta}}} h^\ell -v \cdot \nabla\phi^\ell w_{\widetilde{\vartheta}} \sqrt{\mu} +w_{\widetilde{\vartheta}} \Gamma_{\mathrm{gain}}\left(\frac{h^\ell}{w_{\widetilde{\vartheta}}},
\frac{h^\ell}{w_{\widetilde{\vartheta}}}\right)-w_{\widetilde{\vartheta}} \Gamma_{\mathrm{loss}}
\left(\frac{h^\ell}{w_{\widetilde{\vartheta}}},\frac{h^{\ell+1}}{w_{\widetilde{\vartheta}}}\right),
\end{align}
where $K_{w_{\widetilde{\vartheta}}}(\cdot)=w_{\widetilde{\vartheta}} K(\frac{1}{w_{\widetilde{\vartheta}}}\cdot)$.

We define $\nu (F):=\int_{\mathbb{R}^3\times \mathbb{S}^2}|v-u|^\gamma b_0(\theta)F(u)\mathrm{d}\omega \mathrm{d}u$. Thus,
\begin{equation*}
w_{\widetilde{\vartheta}} \Gamma_{\mathrm{loss}}
\left(\frac{h^\ell}{w_{\widetilde{\vartheta}}},\frac{h^{\ell+1}}{w_{\widetilde{\vartheta}}}\right)
=h^{\ell+1}(v)\nu\left(\sqrt{\mu}\frac{h^\ell}{w_{\widetilde{\vartheta}}}\right).
\end{equation*}
Let
\begin{align*}
\nu^\ell (t,x,v)&:=\nu + \frac{v}{2}\cdot \nabla \phi^\ell +  \frac{\nabla \phi^\ell \cdot \nabla_v w_{\widetilde{\vartheta}}}{w_{\widetilde{\vartheta}}} - \frac{\partial_t w_{\widetilde{\vartheta}}}{w_{\widetilde{\vartheta}}}+\nu\left(\sqrt{\mu}\frac{h^\ell}
{w_{\widetilde{\vartheta}}}\right),\\
g^\ell &:= -v\cdot \nabla \phi^\ell \sqrt{\mu}+\Gamma_{\mathrm{gain}}\left(\frac{h^\ell}{w_{\widetilde{\vartheta}}},
\frac{h^\ell}{w_{\widetilde{\vartheta}}}\right),
\end{align*}
and then \eqref{h l+1} becomes
\begin{equation*}
[\partial_t +v \cdot \nabla_x - \nabla \phi^\ell \cdot \nabla_v +\nu^\ell ]h^{\ell+1}=K_{w_{\widetilde{\vartheta}}}h^\ell+w_{\widetilde{\vartheta}} g^\ell.
\end{equation*}
Along the trajectory, we have
\begin{equation*}
\frac{\mathrm{d}}{\mathrm{d}s}\left(e^{\int^s_t \nu^\ell(\tau)\mathrm {d}{\tau}}h^{\ell+1}(s)\right)=e^{\int_t^s \nu^\ell(\tau)\mathrm{d}\tau}\times \{K_{w_{\widetilde{\vartheta}}}h^\ell (s)+w_{\widetilde{\vartheta}} g^\ell(s)\}.
\end{equation*}
Thus,
\begin{align*}
h^{\ell+1}(t,x,v)&=\mathbf{1}_{t-t_b\leq 0}\, e^{\int_t^0 \nu^\ell(\tau)\mathrm{d}\tau}             h^{\ell+1}(0, x(0), v(0))\\
&+ \mathbf{1}_{t-t_b\geq 0}\,e^{\int_t^{t-t_b} \nu^\ell(\tau)\mathrm{d}\tau}w_{\widetilde{\vartheta}}g(t-t_b, x_b, v_b) \\
&+ \int_{\max\{t-t_b, 0\}}^t e^{\int_t^s \nu^\ell(\tau)\mathrm{d}\tau} \mid [K_{w_{\widetilde{\vartheta}}}h^\ell (s)+w_{\widetilde{\vartheta}} g^\ell(s)](s, x(s), v(s))\mid \mathrm{d}s  .
\end{align*}

We notice that
\begin{equation*}
\exp\left\{\vartheta \left(1+\frac{1}{(1+t)^\theta}\right) |v|^2\right\}\leq
\exp\left\{\vartheta \left(1+\frac{1}{(1+t)^\theta}\right)\langle v \rangle^2\right\}\leq
e^{2\vartheta}\exp\left\{\vartheta \left(1+\frac{1}{(1+t)^\theta}\right) |v|^2\right\}.
\end{equation*}
Hence,
\begin{equation*}
w_{\widetilde{\vartheta}} (v)\sim \exp\left\{\vartheta \left(1+\frac{1}{(1+t)^\theta}\right)\langle v \rangle^2\right\}.
\end{equation*}
Therefore,
\begin{align*}
\widetilde{\nu}&:=\nu + \frac{v}{2}\cdot \nabla \phi^\ell +  \frac{\nabla \phi^\ell \cdot \nabla_v w_{\widetilde{\vartheta}}}{w_{\widetilde{\vartheta}}} - \frac{\partial_t w_{\widetilde{\vartheta}}}{w_{\widetilde{\vartheta}}}\\
&\geq \langle v\rangle^\gamma -\left(\frac{1}{2}+ 4 \vartheta \right)\langle v\rangle |\nabla \phi^\ell|+ \vartheta\theta \langle v\rangle^2 \frac{1}{(v)^{\theta+1}}.
\end{align*}
With the Young's inequality, we have
\begin{equation*}
\langle v\rangle \lesssim \frac{\langle v\rangle^2}{(1+t)^{\theta+1}}+\langle v\rangle^\gamma (1+t)^{(1+\theta)(1-\gamma)}.
\end{equation*}
Note that $\|\nabla \phi^\ell\|_{\infty}\lesssim \|h^\ell\|_{\infty}\leq M \ll 1$ and $\vartheta\theta > M$, then
\begin{align*}
\widetilde{\nu}&\gtrsim \langle v\rangle^\gamma (1-(1+t)^{(1+\theta)(1-\gamma)}|\nabla \phi^\ell|) +(\vartheta\theta -|\nabla \phi^\ell|)\langle v\rangle^2 \frac{1}{(1+t)^{\theta+1}}\\
& \gtrsim \langle v\rangle^\gamma (1-(1+t)^{(1+\theta)(1-\gamma)}\|\nabla \phi^\ell\|_{\infty}) +(\vartheta\theta -\|\nabla \phi^\ell\|_{\infty})\langle v\rangle^2 \frac{1}{(1+t)^{\theta+1}} \\
& \gtrsim \frac{1}{2}\langle v\rangle^\gamma \gtrsim \nu.
\end{align*}
On the other hand, it holds that
\begin{equation*}
\left| \nu \left(\sqrt{\mu}\frac{h^\ell}{w_{\widetilde{\vartheta}}}\right)\right| \lesssim \|h^\ell\|_{\infty}\int_{\mathbb{R}^3}|u-v|^\gamma \sqrt{\mu(u)}\mathrm{d}u \lesssim \langle v\rangle^\gamma\|h^\ell\|_{\infty} \leq \langle v\rangle^\gamma M.
\end{equation*}
Then we have
\begin{align*}
\nu^\ell &= \widetilde{\nu} + \nu \left(\sqrt{\mu}\frac{h^\ell}{w_{\widetilde{\vartheta}}}\right)\\
& \gtrsim \langle v\rangle^\gamma (1-(1+t)^{(1+\theta)(1-\gamma)}\|\nabla \phi^\ell\|_{\infty}-M) +(\vartheta\theta -\|\nabla \phi^\ell\|_{\infty})\langle v\rangle^2 \frac{1}{(1+t)^{\theta+1}} \\
& \gtrsim \frac{1}{2}\langle v\rangle^\gamma \gtrsim \nu.
\end{align*}
Thus we obtain that
\begin{equation*}
e^{\int_t^s \nu^\ell (\tau)\mathrm{d}\tau}\leq 1, \,\,\,0 \leq s \leq t.
\end{equation*}
Therefore,
\begin{align*}
|h^{\ell+1}(t)|&\leq \|h_0\|_{\infty}+ \sup_{0\leq s\leq t}\|w_{\widetilde{\vartheta}}g(s)\|_{\infty}\\
&+ \int_0^t \mid K_{w_{\widetilde{\vartheta}}}h^\ell (s, x(s), v(s))\mid \mathrm{d}s+ \int_0^t \mid w_{\widetilde{\vartheta}}g^\ell (s, x(s), v(s))\mid \mathrm{d}s.
\end{align*}
With the properties of $K$ and $\Gamma$ of soft potential (see for example \cite{shuangqian}, Lemmas 2.2 and 2.3), we immediately know that
\begin{align*}
\|K_{w_{\widetilde{\vartheta}}}h^\ell\|_{\infty}&\leq C \|h^\ell\|_{\infty},\\
\left\|w_{\widetilde{\vartheta}}\Gamma\left(\frac{h^\ell}{w_{\widetilde{\vartheta}}}, \frac{h^\ell}{w_{\widetilde{\vartheta}}}\right)\right\|_{\infty} &\lesssim \langle v\rangle^\gamma \|h^\ell\|_{\infty}^2 \lesssim M \|h^\ell\|_{\infty}\leq C \|h^\ell\|_{\infty},
\end{align*}
where $C>0$ is independent of the time $s$.

Combining these facts with $\|w_{\widetilde{\vartheta}}(-v\cdot \nabla \phi^\ell \sqrt{\mu})\|_{\infty}\lesssim \|\nabla \phi^\ell\|_{\infty}\leq C \|h^\ell\|_{\infty}$ and choosing $\widehat{T}= \frac{1}{2 C}$, we obtain that
\begin{align*}
\|h^{\ell+1}(t)\|_{\infty}&\leq \|h_0\|_{\infty}+\sup_{0\leq s\leq t}\|w_{\widetilde{\vartheta}}g(s)\|_{\infty}+Ct \sup_{0\leq s \leq t}\|h^\ell (s)\|_{\infty}\\
& \leq \frac{M}{2}+ C \widehat{T} M \leq M.
\end{align*}
Hence we have $\sup_{0\leq s\leq \widehat{T}}\|h^{\ell+1}(s)\|_{\infty}\leq M$ and obtain the desired result \eqref{L infty bdd}.
\begin{remark} \label{local exist}
We immediately know that if $\|h_0\|_{\infty}+\sup_{0\leq s\leq \infty}\|w_{\widetilde{\vartheta}}g(s)\|_{\infty} \leq (1-k)M$, $k \in (0,1)$, we can choose $\widehat{T} (k) =\frac{k}{C}$. As $k \rightarrow 0$, $\widehat{T}(k)\rightarrow 0$. Therefore, assuming that there exists some $T>0$ such that
\begin{align}\label{local time}
T = \sup_t \bigg\{t>0:\|h(t)\|_{\infty} + \sup_{0 \leq s \leq \infty}\|w_{\widetilde{\vartheta}}g(s)\|_{\infty} \leq M\bigg\},
\end{align}
we regard such $T$ as the ``initial'' time, then the corresponding $\widehat{T}$ will vanish. It means that such $T$ cannot be extended to $T+$, so $T$ satisfying \eqref{local time} is the largest existing time.
According to this way to expand the local solution, it is clear that the solution exists in $[0,T]$, even if supposing in the next section that $\|h_0\|_{\infty}+ \sup_{0\leq s\leq \infty}\|w_{\widetilde{\vartheta}}g(s)\|_{\infty}\leq \delta^* M$, $\delta^* \ll 1$.

\end{remark}


\emph{Step 3.} We prove that $\{f^\ell\}_{\ell \in \mathbb{N}}$ is the Cauchy sequence in $L^{1+}(\Omega \times \mathbb{R}^3) \cap L^{1+}(\gamma)$.

(i). We define, for $0< \varepsilon \ll 1$ and some $0 < \vartheta_1 < \vartheta$, $\widetilde{\vartheta_1}= \vartheta_1 (1+ (1+t)^{-\theta})$,
\begin{equation*}
E^{\ell+1}(t):= \big\|w_{\widetilde{\vartheta_1}}f^{\ell+1}(t)\big\|_p^p+\big\|w_{\widetilde{\vartheta_1}}
\alpha^\beta_{f^\ell, \varepsilon}\nabla_{x, v}f^{\ell+1}(t)\big\|_p^p.
\end{equation*}
Actually, we know that $\vartheta_1 \leq \frac{s_0}{4} \vartheta$ from Proposition \ref{prop 3}.

We inductively prove that there exist $T^{**}\ll 1$, $T^{**}< \widehat{T}$ and $C>0$ such that
\begin{align}\label{T **}
\max_{\ell} \sup_{0 \leq t\leq T^{**}}E^\ell (t)\leq C \{\|w_\vartheta f(0)\|_p+ \|w_\vartheta \alpha^\beta_{f_0, \varepsilon}\nabla_{x, v}f(0)\|_p+ \|w_\vartheta g\|_{L^p_{s, x, v}} +\|w_\vartheta \nabla_{x, v}g\|_{L^p ((0, \infty); L_{x, v}^\infty )}\}.
\end{align}

For the $\ell =0$ case, $E^0(t)\equiv 0$.
We inductively assume that it holds for the $\ell\geq 1$ case.

First, we claim that for some small constant $\Lambda_2 >0$, $\sup_{0\leq t \leq \widehat{T}}e^{\Lambda_2 t}\|\nabla^2 \phi^{\ell}(t)\|_{\infty}\lesssim M$, where $\widehat{T} \ll 1$ is chosen above.
In fact, considering $\|\nabla^2 \phi (0)\|_{\infty}\leq \delta^* M$, by the continuity, if there is $t'< \widehat{T}$, $\|\nabla^2 \phi^\ell(t')\|_{\infty}$ has already attain $M$, then we use Proposition \ref{prop 3} in $[0, t']$ to get $\|\nabla^2 \phi^\ell(t')\|_{\infty} \leq e^{Ct'} \delta^* M < M$. This conflict shows that the inequality holds in $[0, \widehat{T}]$. Hence, Proposition \ref{prop 3} can be successfully utilized in $[0, \widehat{T}]$.

Similar to the proof of Proposition \ref{prop 3} and considering the index $\ell$, we have
\begin{align*}
E^{\ell +1}(t)=& \|w_{\widetilde{\vartheta_1}}f^{\ell+1}(t)\|_p^p+\|w_{\widetilde{\vartheta_1}}
\alpha^\beta_{f^\ell, \varepsilon}\nabla_{x, v}f^{\ell+1}(t)\|_p^p \\
 \lesssim & \|w_\vartheta f(0)\|_p+ \|w_\vartheta \alpha^\beta_{f_0, \varepsilon}\nabla_{x, v}f(0)\|_p+ \|w_\vartheta g\|_{L^p_{s, x, v}} +\|w_\vartheta \nabla_{x, v}g\|_{L^p ((0, \infty); L_{x, v}^\infty )}\\
& + t \left( 1+ \sup_{0 \leq s \leq t}\|h^\ell (s)\|_{\infty}+ \sup_{0 \leq s \leq t}\|h^{\ell+1} (s)\|_{\infty}+ \sup_{0 \leq s \leq t}\|\nabla^2 \phi^\ell (s)\|_{\infty}\right) \left( \sup_{0 \leq s \leq t}E^\ell (s) +
\sup_{0 \leq s \leq t}E^{\ell+1} (s)\right).
\end{align*}
We choose $T^{**} < \widehat{T} \ll 1$ to get
\begin{align*}
\sup_{0\leq s\leq T^{**}}E^{\ell+1}(s)\lesssim & \|w_\vartheta f(0)\|_p+ \|w_\vartheta \alpha^\beta_{f_0, \varepsilon}\nabla_{x, v}f(0)\|_p+ \|w_\vartheta g\|_{L^p_{s, x, v}} +\|w_\vartheta \nabla_{x, v}g\|_{L^p ((0, \infty); L_{x, v}^\infty )}\\
& + T^{**}(1+M)\sup_{0\leq s\leq T^{**}}E^{\ell}(s).
\end{align*}
With the induction hypothesis, we obtain the conclusion \eqref{T **}.

(ii). We claim that for some $\vartheta_2 \leq \frac{\vartheta_1}{2}$, $\widetilde{\vartheta_2}=\vartheta_2 (1+ (1+t)^{-\theta})$, it holds
\begin{equation*}
\max_\ell \sup_{0\leq s \leq T^{**}} \|w_{\widetilde{\vartheta_2}}\nabla_v f^\ell (s)\|_{L^3_x L^{1+\delta}_v}\lesssim 1.
\end{equation*}

When $\ell =0$,  $\sup_{0\leq s \leq T^{**}} \|w_{\widetilde{\vartheta_2}}\nabla_v f^0 (s)\|_{L^3_x L^{1+\delta}_v}\equiv 0$.

Supposing that the result holds for the $\ell$ case, we follow the same proof of Proposition \ref{prop 4} and conclude that
\begin{align*}
\|w_{\widetilde{\vartheta_2}}\nabla_v f^{\ell+1} (t)\|_{L^3_x L^{1+\delta}_v}  \lesssim & \|w_{\vartheta}\nabla_v f_0\|_{L^3_{x,v}}+ t \left( 1+
\sup_{0 \leq s \leq t}\|h^\ell (s)\|_{\infty}+ \sup_{0 \leq s \leq t}\|h^{\ell+1} (s)\|_{\infty}\right)\\
& \times \sup_{0 \leq s \leq t} \left\{ \|h^\ell (s)\|_{\infty} + \|h^{\ell+1} (s)\|_{\infty}     + E^\ell (s)+ E^{\ell+1}(s)+ \|w_{\widetilde{\vartheta_2}}\nabla_v f^\ell (s)\|_{L^3_x L^{1+\delta}_v}\right\}.
\end{align*}
As $t \in [0, T^{**}]$, we have
\begin{equation*}
\sup_{0\leq s \leq T^{**}} \|w_{\widetilde{\vartheta_2}}\nabla_v f^{\ell+1} (s)\|_{L^3_x L^{1+\delta}_v}\lesssim 1.
\end{equation*}

(iii). We claim that $\{w_{\widetilde{\vartheta_2}} f^\ell\}$ is a Cauchy sequence in $L^{1+}(\Omega \times \mathbb{R}^3) \cap L^{1+}(\gamma)$, for $t \in [0, T^{**}]$.

Note that $(f^{\ell+1}- f^\ell)|_{t=0}\equiv f_0-f_0 =0, (f^{\ell+1}- f^\ell)|_{\gamma_-}
\equiv g-g =0$. Through the same procedure of Proposition \ref{prop 5}, we have
\begin{align*}
&\sup_{0 \leq s \leq t}\|w_{\widetilde{\vartheta_2}}(f^{\ell+1}- f^\ell)(s)\|^{1+\delta}_{L^{1+\delta}(\Omega \times \mathbb{R}^3)}
+ \int_0^t |w_{\widetilde{\vartheta_2}}(f^{\ell+1}- f^\ell)(s)|^{1+\delta}_{1+\delta, +} \\
& \lesssim t\left( 1+ \sup_{0 \leq s \leq t}\max_{m}\left\{\|h^m (s)\|_{\infty}+
\|w_{\widetilde{\vartheta_2}}\nabla_v f^m(s)\|_{L^3_x L^{1+\delta}_v }\right\} \right)
\sup_{0 \leq s \leq t}\big\|w_{\widetilde{\vartheta_2}}(f^{\ell}- f^{\ell-1})(s)\big\|^{1+\delta}_{L^{1+\delta}(\Omega \times \mathbb{R}^3)}.
\end{align*}
Denote $\Phi^\ell := \sup_{0 \leq s \leq t}\big\|w_{\widetilde{\vartheta_2}}(f^{\ell+1}- f^\ell)(s)\big\|^{1+\delta}_{L^{1+\delta}(\Omega \times \mathbb{R}^3)}
+ \int_0^t |w_{\widetilde{\vartheta_2}}(f^{\ell+1}- f^\ell)(s)|^{1+\delta}_{1+\delta, +} $. It holds that
\begin{equation*}
\Phi^\ell \lesssim t \Phi^{\ell-1} \lesssim \cdots \lesssim t^\ell \Phi^0,
\end{equation*}
where $\Phi^0 \equiv \sup_{0 \leq s \leq t}\big\|w_{\widetilde{\vartheta_2}}f^{1} (s)\big\|^{1+\delta}_{L^{1+\delta}(\Omega \times \mathbb{R}^3)}
+ \int_0^t \big|w_{\widetilde{\vartheta_2}}f^1(s)\big|^{1+\delta}_{1+\delta, +}$.

Hence we have, supposing $\ell > m$,
\begin{align*}
&\sup_{0 \leq s \leq t}\|w_{\widetilde{\vartheta_2}}(f^{\ell}- f^m)(s)\|^{1+\delta}_{L^{1+\delta}(\Omega \times \mathbb{R}^3)}
+ \int_0^t |w_{\widetilde{\vartheta_2}}(f^{\ell}- f^m)(s)|^{1+\delta}_{1+\delta, +} \\
& \leq \Phi^{\ell-1}+ \cdots + \Phi^m \lesssim (t^{\ell-1}+ \cdots + t^m)\Phi^0.
\end{align*}
Denoting $S_m := 1+ \cdots + t^{m-1}$ and noticing that $S_m \rightarrow \frac{1}{1-t}$ as $m\rightarrow \infty$, thus it holds that $\{S_m\}$ is a Cauchy sequence and $\{w_{\widetilde{\vartheta_2}} f^\ell\}$ is also a Cauchy sequence.

Therefore, $\{f^\ell\}$ is a Cauchy sequence in $L^{1+}(\Omega \times \mathbb{R}^3) \cap L^{1+}(\gamma)$. Hence there exists a unique $f$, such that $f^\ell \rightarrow f$ strongly in $L^{1+}(\Omega \times \mathbb{R}^3) \cap L^{1+}(\gamma)$ for $t \in [0, T^{**}]$.

\emph{Step 4}. We claim that the limit function $f$ is the weak solution to the VPB system \eqref{perturb eqn}--\eqref{incoming}.
We have already known that $f^\ell$ converges weakly$-*$ to $f$ in $L^\infty ([0, T]\times \overline{\Omega}\times \mathbb{R}^3)$ for some $T>0$. Then for each $h \in L^1 ([0, T]\times \overline{\Omega}\times \mathbb{R}^3)$, it holds $\langle f^\ell, h\rangle \rightarrow \langle f, h\rangle$. For any test function $\varphi \in C_c^\infty ([0, T]\times \overline{\Omega}\times \mathbb{R}^3)$, we have
\begin{align*}
&\int_0^T \int_{\Omega\times \mathbb{R}^3} \left[-\partial_t + v\cdot \nabla_x + \nu\right]\varphi f^{\ell+1} +\left[\nabla \phi^\ell \cdot \nabla_v \varphi+ \frac{v}{2}\cdot \nabla \phi^\ell \varphi\right]f^{\ell+1}\\
&= \int_0^T \int_{\Omega\times \mathbb{R}^3} Kf^\ell \varphi-v \cdot \nabla \phi^\ell \sqrt{\mu}\varphi + \Gamma_{\mathrm{gain}}(f^\ell, f^\ell)\varphi -\Gamma_{\mathrm{loss}}(f^\ell, f^{\ell+1})\varphi - \int_0^T \left( \int_{\gamma_+}- \int_{\gamma_-}\right)f^{\ell+1}\varphi \mathrm{d}\gamma \mathrm{d}t.
\end{align*}
The boundary term above follows from the $L^{1+\delta}(\gamma)$ stability of $f^\ell$. Because the product of weakly$-*$ converging sequences cannot preserve weak$-*$ convergence, we emphasize the following three terms:
\begin{align*}
&(1). \,\,\int_0^T \int_{\Omega\times \mathbb{R}^3}\Gamma_{\mathrm{loss}}(f^\ell, f^{\ell+1})\varphi,\\
&(2). \,\,\int_0^T \int_{\Omega\times \mathbb{R}^3}\Gamma_{\mathrm{gain}}(f^\ell, f^\ell)\varphi,\\
&(3). \,\,\int_0^T \int_{\Omega\times \mathbb{R}^3}\left[\nabla \phi^\ell \cdot \nabla_v \varphi+ \frac{v}{2}\cdot \nabla \phi^\ell \varphi\right]f^{\ell+1}.
\end{align*}

(1). To deal with the limit of $\int_0^T \int_{\Omega\times \mathbb{R}^3}\Gamma_{\mathrm{loss}}(f^\ell, f^{\ell+1})\varphi$, we take the difference:
\begin{align*}
&\left(\Gamma_{\mathrm{loss}}\left(f^\ell, f^{\ell+1}\right)-\Gamma_{\mathrm{loss}}\left(f, f\right)\right)\varphi (t, x, v)\\
&=\int_{\mathbb{R}^3 \times \mathbb{S}^2} |v-u|^\gamma b_0(\theta)\sqrt{\mu(u)}(f^\ell(u)(f^{\ell+1}(v)-f(v))+f(v)(f^\ell(u)-f(u)))\mathrm{d}u \mathrm{d}\omega \cdot \varphi (t, x, v)\\
:&= I_1+I_2.
\end{align*}
For the term $I_1$, we denote that
\begin{align*}
I_1=&(f^{\ell+1}(v)-f(v))\cdot\int_{\mathbb{R}^3 \times \mathbb{S}^2} |v-u|^\gamma b_0(\theta)\sqrt{\mu(u)}f^\ell(u)\mathrm{d}u \mathrm{d}\omega \cdot \varphi (t, x, v)\\
:=& (f^{\ell+1}(v)-f(v))\cdot h_1(t, x, v).
\end{align*}
We immediately have $|h_1|\lesssim \|f^\ell\|_{\infty}\langle v\rangle^\gamma |\varphi|\lesssim |\varphi (t, x, v)|\in L^1_{t, x, v}$ and thus
\begin{equation*}
\int_{t, x, v}(f^{\ell+1}(v)-f(v))\cdot h_1(t, x, v)\mathrm{d}t\mathrm{d}x\mathrm{d}v \rightarrow 0.
\end{equation*}
For the term $I_2$, by changing the variables $(u, v)\leftrightarrow (v, u)$, it holds
\begin{align*}
\int_{\mathbb{R}_v^3}I_2\mathrm{d}v & =\int_{\mathbb{R}^3_v \times \mathbb{R}^3_u \times \mathbb{S}^2}|v-u|^\gamma b_0(\theta)\sqrt{\mu(u)}f(v)(f^\ell(u)-f(u))\cdot \varphi(t, x, v)\mathrm{d}u\mathrm{d}v\mathrm{d}\omega \\
&= \int_{\mathbb{R}^3_v \times \mathbb{R}^3_u \times \mathbb{S}^2}|v-u|^\gamma b_0(\theta)\sqrt{\mu(v)}f(u)(f^\ell(v)-f(v))\cdot \varphi(t, x, u)\mathrm{d}u\mathrm{d}v\mathrm{d}\omega \\
& = (f^\ell(v)-f(v)) \cdot \int_{\mathbb{R}^3_v \times \mathbb{R}^3_u \times \mathbb{S}^2}|v-u|^\gamma b_0(\theta)\sqrt{\mu(v)}f(u)\cdot \varphi(t, x, u)\mathrm{d}u\mathrm{d}v\mathrm{d}\omega\\
:&= (f^\ell(v)-f(v)) \cdot \int_{\mathbb{R}^3_v} h_2(t, x, v)\mathrm{d}v.
\end{align*}
We have $|h_2|\leq \sqrt{\mu(v)}\int_{\mathbb{R}^3_u} |v-u|^\gamma w^{-1}_{\widetilde{\vartheta}}(u)\|w_{\widetilde{\vartheta}}f(u)\|_\infty \varphi(t, x, u)\mathrm{d}u \mathrm{d}\omega \lesssim \sqrt{\mu(v)}\chi_\varphi(t, x) \in L^1_{t, x, v}$, where $\chi_\varphi$ is the characteristic function of $ \mathrm{supp} \varphi|_{t, x}$.

Therefore, it holds for $\ell \rightarrow \infty$ that
\begin{align*}
\int_0^T \int_{\Omega\times \mathbb{R}^3}\Gamma_{\mathrm{loss}}(f^\ell, f^{\ell+1})\varphi \rightarrow \int_0^T \int_{\Omega\times \mathbb{R}^3}\Gamma_{\mathrm{loss}}(f, f)\varphi.
\end{align*}
(2). To deal with the limit of $\int_0^T \int_{\Omega\times \mathbb{R}^3}\Gamma_{\mathrm{gain}}(f^\ell, f^\ell)\varphi$, we take the difference:
\begin{align*}
&\left(\Gamma_{\mathrm{gain}}\left(f^\ell, f^{\ell}\right)-\Gamma_{\mathrm{gain}}\left(f, f\right)\right)\varphi (t, x, v)\\
&=\int_{\mathbb{R}^3 \times \mathbb{S}^2} |v-u|^\gamma b_0(\theta)\sqrt{\mu(u)}(f^\ell(u')(f^{\ell}(v')-f(v'))+f(v')(f^\ell(u')-f(u')))\mathrm{d}u \mathrm{d}\omega \cdot \varphi (t, x, v)\\
:&=I_3+I_4.
\end{align*}
For the term $I_3$, we have
\begin{align*}
\int_{\mathbb{R}_v^3}I_3\mathrm{d}v&=\int_{\mathbb{R}_v^3 \times\mathbb{R}^3_u \times \mathbb{S}^2}  |v-u|^\gamma b_0(\theta)\sqrt{\mu(u)}f^\ell(u')(f^{\ell}(v')-f(v'))\cdot \varphi (t, x, v)\mathrm{d}u \mathrm{d}v \mathrm{d}\omega\\
((u', v')\leftrightarrow(u, v))&= \int |v-u|^\gamma b_0(\theta)\sqrt{\mu(u')}f^\ell(u)(f^{\ell}(v)-f(v))\cdot \varphi (t, x, v')\mathrm{d}u \mathrm{d}v \mathrm{d}\omega\\
&= (f^\ell(v)-f(v))\cdot \int_v \left( \int |v-u|^\gamma b_0(\theta)\sqrt{\mu(u')}f^\ell(u)\cdot \varphi (t, x, v')\mathrm{d}u \mathrm{d}\omega\right)\mathrm{d}v\\
:& = (f^\ell(v)-f(v))\cdot \int_v h_3(t, x, v)\mathrm{d}v.
\end{align*}
Because $\varphi(t, x, v')$ has compact support, it holds $|v'|\leq C$, $C$ is a constant. Then we have
\begin{equation*}
\sqrt{\mu (u')}= e^{-\frac{1}{4} (|u|^2+|v|^2-|v'|^2)}\lesssim \sqrt{\mu(u)}\sqrt{\mu (v)}.
\end{equation*}
Hence $|h_3(t, x, v)|\lesssim \|f^\ell\|_\infty \int_u |v-u|^\gamma \sqrt{\mu(u)}\chi_\varphi \mathrm{d}u \cdot \sqrt{\mu(v)} \in L^1_{t, x, v}$. So $\int_{t,x,v}I_3 \rightarrow 0$ as $\ell \rightarrow \infty$.

For the term $I_4$, we have
\begin{align*}
\int_{\mathbb{R}_v^3}I_4\mathrm{d}v&=\int_{\mathbb{R}^3_v \times \mathbb{R}^3_v \times \mathbb{S}^2}  |v-u|^\gamma b_0(\theta)\sqrt{\mu(u)}f(v')(f^{\ell}(u')-f(u'))\cdot \varphi (t, x, v)\mathrm{d}u \mathrm{d}v \mathrm{d}\omega\\
((u', v')\leftrightarrow(v, u))&= \int |v-u|^\gamma b_0(\theta)\sqrt{\mu(v')}f(u)(f^{\ell}(v)-f(v))\cdot \varphi (t, x, u')\mathrm{d}u \mathrm{d}v \mathrm{d}\omega\\
&= (f^\ell(v)-f(v))\cdot \int_v \left( \int |v-u|^\gamma b_0(\theta)\sqrt{\mu(v')}f(u)\cdot \varphi (t, x, u')\mathrm{d}u \mathrm{d}\omega\right)\mathrm{d}v\\
:& = (f^\ell(v)-f(v))\cdot \int_v h_4(t, x, v)\mathrm{d}v.
\end{align*}
Because $\varphi(t, x, u')$ has compact support, it holds $|u'|\leq C$, $C$ is a constant. Then we have $\sqrt{\mu (v')}\lesssim \sqrt{\mu(u)} \sqrt{\mu(v)}$. Hence
$|h_4(t, x, v)|\lesssim \|f\|_\infty \int_u |v-u|^\gamma \sqrt{\mu(u)}\chi_\varphi \mathrm{d}u \cdot \sqrt{\mu(v)} \in L^1_{t, x, v}$. So $\int_{t,x,v}I_4 \rightarrow 0$ as $\ell \rightarrow \infty$.

Therefore, it holds for $\ell \rightarrow \infty$ that
\begin{align*}
\int_0^T \int_{\Omega\times \mathbb{R}^3}\Gamma_{\mathrm{gain}}(f^\ell, f^\ell)\varphi \rightarrow \int_0^T \int_{\Omega\times \mathbb{R}^3}\Gamma_{\mathrm{gain}}(f, f)\varphi.
\end{align*}

(3). To deal with the limit of $\int_0^T \int_{\Omega\times \mathbb{R}^3}\left[\nabla \phi^\ell \cdot \nabla_v \varphi+ \frac{v}{2}\cdot \nabla \phi^\ell \varphi\right]f^{\ell+1}$ as $\ell\rightarrow \infty$, we take the difference:
\begin{align*}
\nabla \phi^\ell \cdot \nabla_v \varphi f^{\ell+1}-\nabla \phi \cdot \nabla_v \varphi f&=\nabla_v \varphi \cdot (f^{\ell+1}(\nabla \phi^\ell- \nabla \phi)+\nabla \phi (f^{\ell+1}-f))\\
:&=I_5+I_6.
\end{align*}
Obviously, the another term, $\frac{v}{2}\cdot \nabla \phi^\ell \varphi f^{\ell+1}$, shares the same procedure with the latter.

For the term $I_6$:
\begin{equation*}
\int_{t,x,v}I_6=\int_0^t \int_{\Omega \times \mathbb{R}^3} (f^{\ell+1}-f)(\nabla_v \varphi \cdot \nabla \phi)\mathrm{d}x \mathrm{d}v \mathrm{d}s.
\end{equation*}
Note that $\int_0^t \int_{\Omega \times \mathbb{R}^3} |\nabla_v \varphi \cdot \nabla \phi|\mathrm{d}x \mathrm{d}v \mathrm{d}s \lesssim \|\nabla \phi\|_\infty \lesssim \|w_{\widetilde{\vartheta}}f\|_\infty \leq M$. Thus $\nabla_v \varphi \cdot \nabla \phi \in L^1_{t, x, v}$.

For the term $I_5$:
\begin{align*}
\int_{t,x,v}I_5&=\int_0^t \int_{\Omega \times \mathbb{R}^3} f^{\ell+1}\nabla_v \varphi \cdot (\nabla (\phi^\ell-\phi))\mathrm{d}x \mathrm{d}v \mathrm{d}s\\
&\leq \int_0^t \int_{\Omega \times \mathbb{R}^3} |f^{\ell+1}| |\nabla_v \varphi| \cdot |\nabla (\phi^\ell-\phi)|\mathrm{d}x \mathrm{d}v \mathrm{d}s\\
& \lesssim \int_0^t \int_{\Omega }  |\nabla (\phi^\ell-\phi)|\mathrm{d}x \mathrm{d}s.
\end{align*}
We notice that
\begin{align*}
\int_{\Omega}|\nabla (\phi^\ell-\phi)| \mathrm{d}x \lesssim \|\nabla (\phi^\ell-\phi)\|_{L^{1+\delta}_x}\lesssim \|\phi^\ell-\phi\|_{W^{2, 1+\delta}_x}
\lesssim \|\int_v (f^\ell-f)\sqrt{\mu}\mathrm{d}v\|_{L^{1+\delta}_x} \lesssim \|f^\ell-f\|_{L^{1+\delta}_{x, v}}.
\end{align*}
Therefore,
\begin{align*}
\int_0^t \int_{\Omega \times \mathbb{R}^3} f^{\ell+1}\nabla_v \varphi \cdot (\nabla (\phi^\ell-\phi))\mathrm{d}x \mathrm{d}v \mathrm{d}s \lesssim \int_0^t \|f^\ell-f\|_{L^{1+\delta}_{x, v}}\mathrm{d}s \rightarrow 0.
\end{align*}

In conclusion, $f$ is the weak solution to the VPB system \eqref{perturb eqn}--\eqref{incoming}.

\emph{Step 5}. The uniqueness and regularity of $f$.
We prove that $\nabla_{x,v}f^\ell $  converges weakly$-*$ to $\nabla_{x,v}f$ in $L^\infty(\Omega \times \mathbb{R}^3)$.
Actually, since $\|f^\ell\|_\infty\leq M$, it holds that $f^\ell $ converges weakly$-*$ to $f$ in $L^\infty(\Omega \times \mathbb{R}^3)$. For $\psi \in C_c^\infty (\Omega \times \mathbb{R}^3)$, we have
\begin{align*}
\lim_{\ell\rightarrow \infty}\int \nabla_{x,v}f^\ell \psi \mathrm{d}x\mathrm{d}v=-\lim_{\ell\rightarrow \infty}\int \nabla_{x,v}\psi f^\ell \mathrm{d}x\mathrm{d}v=-\int \nabla_{x,v}\psi f \mathrm{d}x\mathrm{d}v.
\end{align*}
Through the density argument, we have $\nabla_{x,v}f^\ell $ converges weakly$-*$ to $\mathcal{F}:=\nabla_{x,v}f$. By the weak lower semicontinuity of $L^p$, for $p=\infty$, it holds
\begin{align*}
\|\nabla_{x,v}f\|_\infty\leq \liminf_{\ell \rightarrow \infty}\|\nabla_{x,v}f^\ell\|_\infty < \infty.
\end{align*}
The fact that $\sup_{0 \leq t \leq T^*}\|f(t)\|_{\infty}$ and $\sup_{0 \leq t \leq T^*}\|\nabla_{x,v} f(t)\|_{\infty}$ are finite means that $f \in W^{1, \infty}_{\mathrm{loc}}([0, T^*]\times \Omega \times \mathbb{R}^3)$. Thus, $f$ is differentiable a.e. in $[0, T^*]\times \Omega \times \mathbb{R}^3$.
The strategies to claim that $\Big\|w_{\frac{s_0\widetilde{\vartheta}}{4}}\alpha^\beta_{f,\varepsilon}\nabla_{x,v}f(t)\Big\|_p< \infty$ is similar to that in \cite{Kim} and we omit the details. Therefore, $f$ is actually a strong solution to the VPB system \eqref{perturb eqn}--\eqref{incoming}.

\end{proof}

\section{Global Existence and Subexponential Decay}
For $0 < M, \delta^* \ll 1$ and $\vartheta \theta > \frac{4 M}{s_0}, \theta\gamma+2 >0$, we assume that there exists a $\lambda_0 >0$ such that
\begin{align*}
&\|w_{\widetilde{\vartheta}}f(0)\|_{\infty}+ \sup_{0\leq s\leq \infty}\|e^{\lambda_0 s^\rho} w_{\widetilde{\vartheta}}g(s)\|_{\infty} +(\|\nabla^2 \phi(0)\|_{\infty}
+ \|\nabla \phi(0)\|_2)\\
&+ \|w_\vartheta f(0)\|_p+ \|w_\vartheta \alpha^\beta_{f_0, \varepsilon}\nabla_{x, v}f(0)\|_p
+ \|w_\vartheta g\|_{L^p_{s, x, v}} +\|w_\vartheta \nabla_{x, v}g\|_{L^p ((0, \infty); L_{x, v}^\infty )} \leq \delta^* M,
\end{align*}
where $(p, \beta)$ satisfies \eqref{p beta}, and $\|w_\vartheta \nabla_v f_0\|_{L^3_{x, v}}< \infty$. With the local continuation and considering the Remark \ref{local exist} in the above section, we set $T>0$ is the largest time such that $\|e^{\lambda_m t^\rho}w_{\widetilde{\vartheta}}f(t)\|_\infty \leq M$ where $\lambda_m \leq \min\big\{\frac{\lambda}{2}, \lambda_2\big\}$ will be clear later.
If $T= \infty$, the case is trivial. Thus we assume that $T$ is finite.
Furthermore, there is an important observation that $T$ is independent of the specific value of $M$, but we will see that the selection of $M$ may depend on that fixed finite $T$. Choosing $M_T$ so small that $ 10 \sqrt{M_T} e^{CT^2}\leq \lambda_m$,
we shall prove that $\|e^{\lambda_m t^\rho}w_{\widetilde{\vartheta}}f(t)\|_\infty$ actually strictly less than $M$ on $[0, T]$, which conflicts to the definition of $T$. Through the bootstrap argument, we can extend the solution to $T+$ until $M_T$ no more satisfies $10 \sqrt{M_T}e^{C(T+ \sigma)^2}\leq \lambda_m$ for some $\sigma>0$.

Denote $h=w_{\widetilde{\vartheta}} f$. The weight linearized equation for $h$ reads
\begin{align*}
\left[\partial_t + v \cdot \nabla_x -\nabla{\phi}\cdot \nabla_v +\nu + \frac{v}{2}\cdot \nabla \phi +  \frac{\nabla \phi \cdot \nabla_v w_{\widetilde{\vartheta}}}{w_{\widetilde{\vartheta}}} - \frac{\partial_t w_{\widetilde{\vartheta}}}{w_{\widetilde{\vartheta}}}                                                                \right]h &=w_{\widetilde{\vartheta}}(Kf+\Gamma(f, f)-v\cdot \nabla \phi \sqrt{\mu}).
\end{align*}
We also denote that
\begin{align*}
w_{\widetilde{\vartheta}}(Kf+\Gamma(f, f)-v\cdot \nabla \phi \sqrt{\mu})&:= K_{w_{\widetilde{\vartheta}}}h + w_{\widetilde{\vartheta}} \widetilde{g},\\ \widetilde{\nu}&:= \nu + \frac{v}{2}\cdot \nabla \phi +  \frac{\nabla \phi \cdot \nabla_v w_{\widetilde{\vartheta}}}{w_{\widetilde{\vartheta}}} - \frac{\partial_t w_{\widetilde{\vartheta}}}{w_{\widetilde{\vartheta}}}.
\end{align*}
Along the trajectory, it holds
\begin{equation*}
\frac{\mathrm{d}}{\mathrm{d}s}\left( e^{\int_t^s \widetilde{\nu}(\tau)\mathrm{d}\tau} h(s)\right)= e^{\int_t^s \widetilde{\nu}(\tau)\mathrm{d}\tau } \left( K_{w_{\widetilde{\vartheta}}}h + w_{\widetilde{\vartheta}} \widetilde{g}\right).
\end{equation*}
Then we have
\begin{align}\label{h(t)}
h(t, x, v)=&\mathbf{1}_{t-t_b>0} e^{\int_t^{t-t_b} \widetilde{\nu}(\tau)\mathrm{d}\tau} h(t-t_b)\nonumber\\
& + \mathbf{1}_{t-t_b<0}e^{\int_t^0 \widetilde{\nu}(\tau)\mathrm{d}\tau} h_0
+ \int_{\max\{t-t_b, 0\}}^t e^{\int_t^s \widetilde{\nu}(\tau)\mathrm{d}\tau} w_{\widetilde{\vartheta}}\widetilde{g}(s)\mathrm{d}s\\
&+\int_{\max\{t-t_b, 0\}}^t e^{\int_t^s \widetilde{\nu}(\tau)\mathrm{d}\tau}
\left( K_{w_{\widetilde{\vartheta}}}^\chi h(s)+ K_{w_{\widetilde{\vartheta}}}^{1-\chi} h(s)\right)\mathrm{d}s. \nonumber
\end{align}
Similarly, we can obtain that
\begin{align}\label{h(s)}
h(s, x(s), v')=&\mathbf{1}_{s-t_b'>0} e^{\int_s^{s-t_b'} \widetilde{\nu}(\tau')\mathrm{d}\tau'} h(s-t_b')\nonumber\\
& + \mathbf{1}_{s-t_b'<0}e^{\int_s^0 \widetilde{\nu}(\tau')\mathrm{d}\tau'} h_0
+ \int_{\max\{s-t_b', 0\}}^s e^{\int_s^{s_1} \widetilde{\nu}(\tau')\mathrm{d}\tau'} w_{\widetilde{\vartheta}}\widetilde{g}(s_1)\mathrm{d}s_1\\
& +\int_{\max\{s-t_b', 0\}}^s e^{\int_s^{s_1} \widetilde{\nu}(\tau')\mathrm{d}\tau'}
\left( K_{w_{\widetilde{\vartheta}}}^\chi h(s_1)+ K_{w_{\widetilde{\vartheta}}}^{1-\chi} h(s_1)\right)\mathrm{d}s_1.\nonumber
\end{align}
Recalling the discussion in Section 4, it holds
\begin{align*}
\widetilde{\nu} \gtrsim \langle v\rangle^\gamma (1-(1+t)^{(1+\theta)(1-\gamma)}\|\nabla \phi\|_{\infty}) +(\vartheta\theta -\|\nabla \phi\|_{\infty})\langle v\rangle^2 \frac{1}{(1+t)^{\theta+1}}.
\end{align*}
On one hand, we deduce that $\widetilde{\nu}\gtrsim \frac{1}{2}\langle v\rangle^\gamma \gtrsim \nu$. On the other hand, by using the Young's inequality, we have
\begin{align*}
\widetilde{\nu}\gtrsim \frac{1}{2}\langle v\rangle^\gamma + C \frac{\langle v\rangle^2}{(1+t)^{\theta+1}}\gtrsim (1+t)^{\frac{(\theta+1)\gamma}{2-\gamma}}:= (1+t)^{\rho-1},
\end{align*}
where $\rho:= \frac{\theta \gamma+2}{2-\gamma}\in (0, 1)$, provided that $\theta\gamma +2 >0$. Hence,
\begin{align*}
e^{\int_t^s \widetilde{\nu}(\tau)\mathrm{d}\tau} \lesssim e^{\frac{C}{\rho}\int_t^s \mathrm{d}(1+\tau)^\rho}:= e^{\lambda ((1+s)^\rho-(1+t)^\rho)}\sim e^{\lambda s^\rho- \lambda t^\rho}, \,\,\, \lambda>0.
\end{align*}
With Lemmas \ref{kwh} and \ref{gamma} and the above crucial investigations, one obtains
\begin{align*}
\int_0^t e^{\int_t^s \widetilde{\nu}(\tau)\mathrm{d}\tau} K_{w_{\widetilde{\vartheta}}}^{1-\chi}h(s)\mathrm{d}s
\lesssim \int_0^t e^{\frac{1}{2}\int_t^s \widetilde{\nu}(\tau)\mathrm{d}\tau}\widetilde{\nu}(s, x(s), v(s))e^{\frac{\lambda}{2}(s^\rho-t^\rho)}\widetilde{\nu}^{-1}(s)
K_{w_{\widetilde{\vartheta}}}^{1-\chi}h(s)\mathrm{d}s.
\end{align*}
Note in addition that
\begin{align*}
&\widetilde{\nu}^{-1}(s)
K_{w_{\widetilde{\vartheta}}}^{1-\chi}h(s)\lesssim {\nu}^{-1}(s)
K_{w_{\widetilde{\vartheta}}}^{1-\chi}h(s)\lesssim \varepsilon^{\gamma+3}\|h(s)\|_\infty, \\
&\int_0^t e^{\frac{1}{2}\int_t^s \widetilde{\nu}(\tau)\mathrm{d}\tau}\widetilde{\nu}(s, x(s), v(s))\mathrm{d}s=2\int_0^t \mathrm{d}\left( e^{\frac{1}{2}\int_t^s \widetilde{\nu}(\tau)\mathrm{d}\tau}\right)\leq 2.
\end{align*}
Therefore, we have
\begin{align}\label{53}
\int_0^t e^{\int_t^s \widetilde{\nu}(\tau)\mathrm{d}\tau} K_{w_{\widetilde{\vartheta}}}^{1-\chi}h(s)\mathrm{d}s
\lesssim \varepsilon^{\gamma +3}e^{-\frac{\lambda}{2} t^\rho}\sup_{0 \leq s \leq t}\Big\|e^{\frac{\lambda}{2}s^\rho} h(s)\Big\|_\infty.
\end{align}
Considering $\|\nu^{-1}w_{\widetilde{\vartheta}}\Gamma (f, f)(s)\|_\infty \lesssim \|h(s)\|^2_\infty$  and then taking the same procedure as the above, we have
\begin{align}\label{54}
\int_0^t e^{\int_t^s \widetilde{\nu}(\tau)\mathrm{d}\tau} w_{\widetilde{\vartheta}}\Gamma (f, f)\mathrm{d}s
\lesssim \sup_{0 \leq s \leq t} \|h(s)\|_\infty e^{-\frac{\lambda}{2} t^\rho}\sup_{0 \leq s \leq t}\Big\|e^{\frac{\lambda}{2}s^\rho} h(s)\Big\|_\infty.
\end{align}
For the term of $\widetilde{g}$, i.e. $v \cdot \nabla \phi \sqrt{\mu}$, we have
\begin{align}\label{55}
&\int_0^t e^{\int_t^s \widetilde{\nu}(\tau)\mathrm{d}\tau} w_{\widetilde{\vartheta}}(v(s))v(s)\cdot \nabla \phi(s, x(s))\sqrt{\mu(v(s))}\mathrm{d}s \nonumber\\
 &\lesssim \int_0^t e^{\frac{1}{2}\int_t^s \widetilde{\nu}(\tau)\mathrm{d}\tau}\widetilde{\nu}(s)e^{\frac{\lambda}{2}(s^\rho-t^\rho)}
(\nu^{-1}(s)|v(s)|\mu^{1/4}(v(s)))|\nabla \phi(s, x(s))| \mathrm{d}s\nonumber\\
 &\lesssim e^{-\frac{\lambda}{2}t^\rho}\sup_{0\leq s \leq t}\|e^{\frac{\lambda}{2}s^\rho} \nabla \phi (s)\|_\infty.
\end{align}
In summary, combining \eqref{53}--\eqref{55} up, we obtain that
\begin{align*}
&\int_0^t e^{\int_t^s \widetilde{\nu}(\tau)\mathrm{d}\tau} (K_{w_{\widetilde{\vartheta}}}^{1-\chi}h(s)+ w_{\widetilde{\vartheta}} \widetilde{g}(s))\mathrm{d}s \\
\lesssim & e^{-\frac{\lambda}{2} t^\rho}\left\{(\varepsilon^{\gamma+3}+\sup_{0 \leq s \leq t} \|h(s)\|_\infty)\sup_{0 \leq s \leq t}\|e^{\frac{\lambda}{2}s^\rho} h(s)\|_\infty +\sup_{0\leq s \leq t}\|e^{\frac{\lambda}{2}s^\rho} \nabla \phi (s)\|_\infty\right\}.
\end{align*}
In addition, it holds
\begin{align*}
& e^{\int_t^{t-t_b} \widetilde{\nu}(\tau)\mathrm{d}\tau} h(t-t_b)+e^{\int_t^{0} \widetilde{\nu}(\tau)\mathrm{d}\tau} h(0)\\
 \lesssim & e^{-\frac{\lambda}{2}t^\rho}\left\{ \sup_{0 \leq s \leq t}\|e^{\frac{\lambda}{2}s^\rho}w_{\widetilde{\vartheta}}g(s)\|_\infty+
\|w_{\widetilde{\vartheta}}f(0)\|_\infty \right\}.
\end{align*}
By Lemma \ref{kwh}, we have $\widetilde{\nu}^{-1}(s)
K_{w_{\widetilde{\vartheta}}}^{\chi}h(s)\lesssim {\nu}^{-1}(s)
K_{w_{\widetilde{\vartheta}}}^{\chi}h(s)\lesssim \|h(s)\|_\infty$. Plugging \eqref{h(s)} into \eqref{h(t)} gives
\begin{align*}
&\int_0^t e^{\int_t^s \widetilde{\nu}(\tau)\mathrm{d}\tau}K_{w_{\widetilde{\vartheta}}}^\chi \left( \int_0^s e^{\int_s^{s_1} \widetilde{\nu}(\tau')\mathrm{d}\tau'}
\left(  K_{w_{\widetilde{\vartheta}}}^{1-\chi} h(s_1)+w_{\widetilde{\vartheta}}\widetilde{g}(s_1)\right)\mathrm{d}s_1 \right) \mathrm{d}s\\
&\lesssim e^{-\frac{\lambda}{2}t^\rho}\sup_{0 \leq s \leq t}\left\{ e^{\frac{\lambda}{2}s^\rho} \cdot e^{-\frac{\lambda}{2} s^\rho}\{(\varepsilon^{\gamma+3}+\sup_{0 \leq s_1 \leq s} \|h(s_1)\|_\infty)\sup_{0 \leq s_1 \leq s}\|e^{\frac{\lambda}{2}{s_1}^\rho} h(s_1)\|_\infty +\sup_{0\leq s_1 \leq s}\|e^{\frac{\lambda}{2}{s_1}^\rho} \nabla \phi (s_1)\|_\infty\}\right\}\\
& \lesssim e^{-\frac{\lambda}{2} t^\rho}\left\{(\varepsilon^{\gamma+3}+\sup_{0 \leq s \leq t} \|h(s)\|_\infty)\sup_{0 \leq s \leq t}\|e^{\frac{\lambda}{2}s^\rho} h(s)\|_\infty +\sup_{0\leq s \leq t}\|e^{\frac{\lambda}{2}s^\rho} \nabla \phi (s)\|_\infty\right\},
\end{align*}
and similarly,
\begin{align*}
&\int_0^t e^{\int_t^s \widetilde{\nu}(\tau)\mathrm{d}\tau}K_{w_{\widetilde{\vartheta}}}^\chi \left(e^{\int_s^{s-t_b'} \widetilde{\nu}(\tau')\mathrm{d}\tau'} h(s-t_b')+e^{\int_s^{0} \widetilde{\nu}(\tau')\mathrm{d}\tau'} h(0) \right)\mathrm{d}s \\
& \lesssim e^{-\frac{\lambda}{2}t^\rho}\left\{ \sup_{0 \leq s \leq t}\|e^{\frac{\lambda}{2}s^\rho}w_{\widetilde{\vartheta}}g(s)\|_\infty+
\|w_{\widetilde{\vartheta}}f(0)\|_\infty\right\}.
\end{align*}
In summary, after the twice iteration, we obtain
\begin{equation}\label{56}
\begin{aligned}
&|h(t, x, v)|\\
\lesssim & e^{-\frac{\lambda}{2} t^\rho}\left\{(\varepsilon^{\gamma+3}+\sup_{0 \leq s \leq t} \|h(s)\|_\infty)\sup_{0 \leq s \leq t}\|e^{\frac{\lambda}{2}s^\rho} h(s)\|_\infty +\sup_{0\leq s \leq t}\|e^{\frac{\lambda}{2}s^\rho} \nabla \phi (s)\|_\infty\right\}\\
& + e^{-\frac{\lambda}{2}t^\rho}\left\{ \sup_{0 \leq s \leq t}\|e^{\frac{\lambda}{2}s^\rho}w_{\widetilde{\vartheta}}g(s)\|_\infty+
\|w_{\widetilde{\vartheta}}f(0)\|_\infty \right\}\\
& +\int_{\max\{t-t_b, 0\}}^t \int_{\max\{s-t_b', 0\}}^s \int_{v'} \int_{v''}
e^{\int_t^{s} \widetilde{\nu}(\tau)\mathrm{d}\tau}
e^{\int_s^{s_1} \widetilde{\nu}(\tau')\mathrm{d}\tau'}
\mathbf{k}_w^\chi (v(s), v')\mathbf{k}_w^\chi (v'(s_1), v'')h(s_1, x'(s_1), v'')\mathrm{d}v''\mathrm{d}v'\mathrm{d}s_1\mathrm{d}s.
\end{aligned}
\end{equation}
The last term on the right-hand side of \eqref{56} is controlled by
\begin{align*}
\int_0^t \int_0^s \int_{v'} \int_{v''} &\left[ e^{\frac{1}{2}\int_t^{s} \widetilde{\nu}(\tau)\mathrm{d}\tau}\widetilde{\nu}(v(s))\right]
\left[ e^{\frac{1}{2}\int_s^{s_1} \widetilde{\nu}(\tau')\mathrm{d}\tau'}\widetilde{\nu}(v'(s_1))\right]
e^{\frac{\lambda}{2}({s_1}^\rho-t^\rho)}
(\widetilde{\nu}^{-1}(v(s))\mathbf{k}_w^\chi (v(s), v'))\\
& \cdot(\widetilde{\nu}^{-1}(v'(s_1))\mathbf{k}_w^\chi (v'(s_1), v''))
\cdot h(s_1, x'(s_1), v'')\mathrm{d}v''\mathrm{d}v'\mathrm{d}s_1\mathrm{d}s := \mathcal{E}.
\end{align*}
By Lemma \ref{kwh}, for $\varepsilon$ sufficiently small, we know that
\begin{equation*}
\int_{\mathbb{R}^3_u} \nu^{-1}(v)\mathbf{k}_w^\chi (v, u)e^{\varepsilon |v-u|^2}\mathrm{d}u \lesssim \langle v\rangle^{-2}\lesssim 1,
\end{equation*}
where
\begin{align*}
\nu^{-1}(v)\mathbf{k}_w^\chi (v, u)&\lesssim \langle v\rangle^{-\gamma} \frac{\exp\left( -\frac{s_2}{8}|u-v|^2- \frac{s_1}{8}\frac{(|v|^2-|u|^2)^2}{|v-u|^2}+\widetilde{\vartheta}
(|v|^2-|u|^2)\right)}{|v-u|(1+|v|+|u|)^{1-\gamma}}\\
& \lesssim \frac{\exp\left( -\frac{s_2}{8}|u-v|^2- \frac{s_1}{8}\frac{(|v|^2-|u|^2)^2}{|v-u|^2}+\widetilde{\vartheta}
(|v|^2-|u|^2)\right)}{|v-u|}:= \mathbf{k}_3 (v, u).
\end{align*}
We also have
\begin{align*}
\int_{\mathbb{R}^3_u} \mathbf{k}_3 (v, u) e^{\varepsilon |v-u|^2}\mathrm{d}u \lesssim \langle v\rangle^{-1} \lesssim 1.
\end{align*}
Hence,
\begin{align}\label{twice k3}
\mathcal{E}\lesssim e^{-\frac{\lambda}{2}t^\rho}\sup_{0\leq s_1 \leq t}\{\|e^{\frac{\lambda}{2}{s_1}^\rho}h(s_1)\|_\infty\}
\int \mathbf{k}_3 (v(s), v')\mathbf{k}_3 (v'(s_1), v'')\mathrm{d}v' \mathrm{d}v''.
\end{align}
Below we use some ideas developed in \cite{Guo2010} to simplify the integration (\ref{twice k3}).

\emph{Case 1}. $|v|>N$.
Note that $|v(t)-v(s)|\leq \int_s^t |\nabla \phi(\tau)|\mathrm{d}\tau \leq \int_s^t M e^{-\lambda_m \tau^\rho}\mathrm{d}\tau \leq \frac{M}{\lambda_m^{\frac{1}{\rho}}} \Gamma(\frac{1}{\rho}+1)\lesssim \frac{M}{\lambda_m^{\frac{1}{\rho}}}\ll 1$, where we have used the Gamma function:
\begin{align*}
\int_0^\infty e^{-s^\rho}\mathrm{d}s= \frac{1}{\rho}\int_0^\infty \tau^{\frac{1}{\rho}-1}e^{-\tau}\mathrm{d}\tau = \frac{1}{\rho}\Gamma\left(\frac{1}{\rho}\right)=\Gamma \left(\frac{1}{\rho}+1\right)< \infty.
\end{align*}
Thus we obtain $|v(s)|>N-M/(\lambda_m)^{\frac{1}{\rho}}$. Furthermore,
\begin{equation*}
\int \mathbf{k}_3 (v(s), v')\mathbf{k}_3 (v'(s_1), v'')\mathrm{d}v' \mathrm{d}v''\leq \frac{C_k}{1+|v(s)|}\leq \frac{C_k}{N}.
\end{equation*}

\emph{Case 2}. $|v|\leq N, |v'|>2N$ or $|v'|\leq 2N, |v''|>3N$.
In this case, we immediately know that $|v(s)|\leq N+M/(\lambda_m)^{\frac{1}{\rho}}, |v'|>2N$ or $|v'(s_1)|<2N+ M/(\lambda_m)^{\frac{1}{\rho}}, |v''|>3N$. It means that
$|v(s)-v'|> \frac{N}{2}$ or $|v'(s_1)-v''|>\frac{N}{2}$. Hence we have
\begin{align*}
|\mathbf{k}_3 (v(s), v')|&\leq e^{-\frac{\varepsilon}{4}N^2}|\mathbf{k}_3 (v(s), v') e^{\varepsilon|v(s)-v'|^2}|, \,\,\mathrm{or}\\
|\mathbf{k}_3 (v'(s_1), v'')|&\leq e^{-\frac{\varepsilon}{4}N^2}|\mathbf{k}_3 (v'(s_1), v'') e^{\varepsilon|v'(s_1)-v''|^2}|.
\end{align*}
Therefore,
\begin{align*}
\int \mathbf{k}_3 (v(s), v')\mathbf{k}_3 (v'(s_1), v'')\mathrm{d}v' \mathrm{d}v''
\leq e^{-\frac{\varepsilon}{4}N^2}.
\end{align*}

\emph{Case 3}. $s-s_1\leq \delta$, $\delta$ is sufficiently small.
By the time continuity of $\widetilde{\nu}(\tau')$, we have
\begin{align*}
&\int_{s-\delta}^s \left[ e^{\frac{1}{2}\int_s^{s_1} \widetilde{\nu}(\tau')\mathrm{d}\tau'}\widetilde{\nu}(v'(s_1))\right] \mathrm{d}s_1
= 2 \int_{s-\delta}^s \mathrm{d} \left(e^{\frac{1}{2}\int_s^{s_1} \widetilde{\nu}(\tau')\mathrm{d}\tau'}\right)=O(\delta).
\end{align*}

In summary, for the above three cases, it holds
\begin{align*}
\mathcal{E}\lesssim \{N^{-1} + e^{-\frac{\varepsilon}{4}N^2} +\delta\}e^{-\frac{\lambda}{2}t^\rho}\sup_{0\leq s \leq t}\{\|e^{\frac{\lambda}{2}{s}^\rho}h(s)\|_\infty\}.
\end{align*}

\emph{Case 4}. $s-s_1 >\delta, |v|\leq N, |v'|\leq 2N, |v''|\leq 3N$.
We choose some $\mathbf{k}_N (v, u)$ smooth with compact support, such that
\begin{align*}
\sup_{|v|\leq 3N}\int_{\mathbb{R}^3}|\mathbf{k}_N(v, u)-\mathbf{k}_3(v, u)|\mathrm{d}u \leq \frac{1}{N}.
\end{align*}
We have
\begin{align*}
&\mathbf{k}_3(v(s), v') \mathbf{k}_3(v'(s_1), v'')\\
=& ( \mathbf{k}_3(v(s), v')-\mathbf{k}_N(v(s), v') )\mathbf{k}_3(v'(s_1), v'')+\mathbf{k}_N(v(s), v') (\mathbf{k}_3(v'(s_1), v'')-\mathbf{k}_N(v'(s_1), v''))\\
&+\mathbf{k}_N(v(s), v') \mathbf{k}_N(v'(s_1), v'').
\end{align*}
On the right-hand-side above, the contribution of the first and second terms is
\begin{align*}
\frac{1}{N} e^{-\frac{\lambda}{2}t^\rho}\sup_{0\leq s \leq t}\{\|e^{\frac{\lambda}{2}{s}^\rho}h(s)\|_\infty\}.
\end{align*}
For the last term, it holds that $\mathbf{k}_N(v(s), v') \mathbf{k}_N(v'(s_1), v'') \lesssim 1$, so it remains
\begin{equation*}
\int_{v'}\int_{v''} h(s_1, x'(s_1), v'')\mathrm{d}v'\mathrm{d}v''
\end{equation*}
to be estimated.
Let $y\equiv x'(s_1)=x(s)-\int_{s_1}^s v'(\tau)\mathrm{d}\tau$ and note that $v'(\tau)=v'- \int_s^\tau \nabla \phi(\tau')\mathrm{d}\tau'$, we have
\begin{align*}
y=x(s)-v'(s-s_1)-\int^s_{s_1}\int^s_\tau \nabla_x \phi(\tau', x'(\tau'))\mathrm{d}\tau'\mathrm{d}\tau.
\end{align*}
Hence,
\begin{align*}
\frac{\partial y}{\partial v'}=-(s-s_1)\mathrm{Id}_{3\times 3}- \int_{s_1}^s \int_\tau^s \nabla^2 \phi(\tau')\nabla_{v'}x'(\tau')\mathrm{d}\tau' \mathrm{d}\tau.
\end{align*}
By taking the similar argument to the one in the preceding section, we know that $\sup_{0\leq s\leq T}\|\nabla^2 \phi (s)\|_\infty \leq M$, thus we have
\begin{align*}
\sup_{0\leq s\leq T}\|e^{\lambda_m s} \nabla^2 \phi (s)\|_\infty \leq  e^{\lambda_m T}M := C_T M
\end{align*}
for the finite $T$ selected above. By Lemma \ref{lemma10}, it holds
\begin{align*}
|\nabla_{v'}x'(\tau'; s, x'(s), v'(s))|\leq C_T e^{C_T M / \lambda_m^2} |s-\tau'|.
\end{align*}
So we have
\begin{align*}
&\int_{s_1}^s \int_\tau^s |\nabla^2 \phi(\tau')|\cdot  |\nabla_{v'}x'(\tau')|\mathrm{d}\tau' \mathrm{d}\tau \\
 \leq &\int_{s_1}^s \int_\tau^s e^{-\lambda_m \tau'}C_T M \cdot C_T e^{C_T M /\lambda_m^2}(s-\tau')\mathrm{d}\tau' \mathrm{d}\tau\\
 \leq & |s-s_1| \frac{C_T^2 M}{\lambda_m^2}e^{\frac{C_T^2 M}{\lambda_m^2}}.
\end{align*}
Let us take $ \frac{C_T^2 M}{\lambda_m^2}< \frac{1}{100}$, i.e. $M < \frac{\lambda_m^2}{100 C_T^2}$. Thus we obtain the crucial estimate:
\begin{align*}
\det \left( \frac{\partial y}{\partial v'}\right)\gtrsim |s-s_1|^3 \gtrsim \delta^3.
\end{align*}
Therefore,
\begin{align*}
\int_{v'}\int_{v''} h(s_1, x'(s_1), v'')\mathrm{d}v'\mathrm{d}v''
=& \int_{\Omega}\int_{v''}h(s_1, y, v'')\mathrm{d}v'' \left|\det\left( \frac{\partial v'}{\partial y}\right)\right|\mathrm{d}y\\
\lesssim & \delta^{-3}\int_{\Omega \times \{|v''|\leq 3N\}}w_{\widetilde{\vartheta}}(v'')f(s_1, y, v'')\mathrm{d}v''\mathrm{d}y\\
\lesssim & \delta^{-3}e^{C N^2}\|f(s_1)\|_{L^2_{x, v}}.
\end{align*}

As a result, we have
\begin{align*}
\mathcal{E}\lesssim \{N^{-1} + e^{-\frac{\varepsilon}{4}N^2} +\delta\}e^{-\frac{\lambda}{2}t^\rho}\sup_{0\leq s \leq t}\{\|e^{\frac{\lambda}{2}{s}^\rho}h(s)\|_\infty\} + \frac{e^{CN^2}}{\delta^3} e^{-\frac{\lambda}{2}t^\rho}\sup_{0\leq s\leq t}\{e^{\frac{\lambda}{2}s^\rho}\|f(s)\|_{L^2_{x, v}}\}.
\end{align*}

In summary, we eventually obtain that
\begin{align}\label{L infty}
e^{\frac{\lambda}{2}t^\rho}|h(t, x, v)| \lesssim & \|w_{\widetilde{\vartheta}}f(0)\|_\infty
+ \sup_{0 \leq s \leq t}\left\{e^{\frac{\lambda}{2}s^\rho}\|w_{\widetilde{\vartheta}}g(s)\|_\infty \right\}\nonumber\\
&+ \left[  \sup_{0 \leq s \leq t}\|h(s)\|_\infty + \varepsilon^{\gamma+3}+ \frac{1}{N}+ e^{-\frac{\varepsilon}{4}N^2}+\delta \right] \sup_{0 \leq s \leq t}\left\{e^{\frac{\lambda}{2}s^\rho}\|h(s)\|_\infty \right\}\nonumber\\
&+  \sup_{0 \leq s \leq t}\left\{e^{\frac{\lambda}{2}s^\rho}\|\nabla \phi(s)\|_\infty \right\}+ \frac{C_N}{\delta^3} \sup_{0 \leq s \leq t}\left\{e^{\frac{\lambda}{2}s^\rho}\|f(s)\|_{L^2_{x, v}}\right\}.
\end{align}

Below we deal with the $L^2$ estimate of $f$.
Multiplying both sides of (\ref{perturb eqn}) with $f$ and then with $e^{2\lambda_2 t^\rho}$, we have
\begin{align*}
&[\partial_t + v\cdot \nabla_x + E\cdot \nabla_v](e^{\lambda_2 t^\rho}f)^2-2\lambda_2 \rho t^{\rho-1}(e^{\lambda_2 t^\rho}f)^2-v\cdot E (e^{\lambda_2 t^\rho}f)^2 + 2 e^{\lambda_2 t^\rho}f L(e^{\lambda_2 t^\rho}f)\\
&= 2e^{2\lambda_2 t^\rho}f\Gamma (f, f)+ 2 e^{2\lambda_2 t^\rho}v\cdot E \sqrt{\mu}f.
\end{align*}
After taking integration over $(0, t)\times \Omega \times \mathbb{R}^3$, we arrive at
\begin{align*}
&\|e^{\lambda_2 t^\rho } f(t)\|_2^2- \|f(0)\|_2^2+ \int_0^t |e^{\lambda_2 \tau^\rho}f(\tau)|^2_{2,+}-\int_0^t |e^{\lambda_2 \tau^\rho}g(\tau)|^2_{2,-}-2 \lambda_2 \rho \int_0^t \tau^{\rho-1}\|e^{\lambda_2 \tau^\rho}f(\tau)\|_2^2\\
& + \int_0^t \int_{\Omega \times \mathbb{R}^3} v\cdot \nabla \phi e^{2\lambda_2 \tau^\rho}|f(\tau)|^2 + 2 \int_0^t e^{2\lambda_2 \tau^\rho}(f, Lf)\\
= & 2 \int_0^t \int_{\Omega \times \mathbb{R}^3} e^{2\lambda_2 \tau^\rho} f \Gamma (f, f)
-2 \int_0^t e^{2\lambda_2 \tau^\rho}\int_{\Omega}\nabla \phi \cdot \int_{\mathbb{R}^3}v\sqrt{\mu}f.
\end{align*}
On the other hand, we multiply both sides of (\ref{perturb eqn}) by $\sqrt{\mu}$ and then integrate over $\mathbb{R}^3_v$ to get
\begin{align*}
\partial_t \left(\int_v \sqrt{\mu}f \mathrm{d}v\right)+ \nabla_x \cdot \left(\int_v v \sqrt{\mu}f \mathrm{d}v\right)\equiv 0.
\end{align*}
If we denote
\begin{align*}
\widetilde{\rho}(t, x)=\int_v \sqrt{\mu}f \mathrm{d}v,\,\,\,\widetilde{j}(t, x)=\int_v v \sqrt{\mu}f \mathrm{d}v,
\end{align*}
then we have
\begin{align*}
\partial_t \widetilde{\rho}+ \nabla_x \cdot \widetilde{j}=0.
\end{align*}
By using the Neumann boundary condition $\frac{\partial \phi}{\partial n}\big|_{\partial\Omega}=0$, we obtain that
\begin{align*}
\frac{1}{2}\partial_t \int_{\Omega} |\nabla \phi|^2 \mathrm{d}x
=& \int_{\partial \Omega} \phi \partial_t \left( \frac{\partial \phi}{\partial n}\right) \mathrm{d}S_x + \int_{\Omega}\phi \partial_t (-\Delta \phi)\mathrm{d}x\\
= & \int_{\Omega}\phi \partial_t \widetilde{\rho}\mathrm{d}x
= \int_\Omega  -\phi \nabla_x \cdot \widetilde{j}\mathrm{d}x
= \int_{\partial \Omega}- \phi \widetilde{j}\cdot n \mathrm{d}S_x - \int_x E \cdot \widetilde{j}\mathrm{d}x\\
=& \int_{\partial\Omega}-\phi \mathrm{d}S_x \int_{\mathbb{R}^3_v}\sqrt{\mu}f \{n(x)\cdot v\} \mathrm{d}v- \int_{x, v}E \cdot v \sqrt{\mu} f \mathrm{d}x \mathrm{d}v
\end{align*}
To vanish the boundary term on the right-hand side of the last equality, we need
some condition about the in-flow, namely:
\begin{align*}
\int_{n\cdot v>0}F |n \cdot v|\mathrm{d}v= \int_{n\cdot v<0}G |n \cdot v|\mathrm{d}v,\,\,\,\, F|_{\gamma_-}=G.
\end{align*}
In fact, it follows that
\begin{align*}
0 = \int_{\mathbb{R}^3_v} F(t, x, v)\{n\cdot v\}\mathrm{d}v
= \int_{\mathbb{R}^3_v} (F(t, x, v)-\mu(v) )\{n\cdot v\}\mathrm{d}v
= \int_{\mathbb{R}^3_v} \sqrt{\mu} f(t, x, v)\{n\cdot v\}\mathrm{d}v.
\end{align*}
Thus, it holds
\begin{equation*}
\partial_t \int_{\Omega}|\nabla \phi|^2 \mathrm{d}x+ 2 \int_{\Omega \times \mathbb{R}^3}E\cdot v \sqrt{\mu}f \mathrm{d}x\mathrm{d}v =0.
\end{equation*}
After multiplying $e^{2 \lambda_2 t^\rho}$ to both sides of the above equation and integrating the result over $[0, t]$, we have
\begin{align*}
\|e^{\lambda_2 t^\rho}\nabla \phi(t)\|_{L^2(\Omega)}^2-\|\nabla \phi(0)\|_{L^2(\Omega)}^2
-2 \lambda_2 \rho \int_0^t \tau^{\rho-1}\|e^{\lambda_2 \tau^\rho}\nabla \phi(\tau)\|_{L^2(\Omega)}^2 \mathrm{d}\tau= \int_0^t 2 e^{2\lambda_2 \tau^\rho}\nabla \phi \cdot v \sqrt{\mu}f \mathrm{d}x\mathrm{d}v\mathrm{d}\tau.
\end{align*}
In a result, it holds
\begin{equation}\label{59}
\begin{aligned}
&\|e^{\lambda_2 t^\rho } f(t)\|_2^2 + \|e^{\lambda_2 t^\rho } \nabla \phi(t)\|_2^2 + \int_0^t |e^{\lambda_2 \tau^\rho}f(\tau)|^2_{2,+}
+ 2 \int_0^t e^{2\lambda_2 \tau^\rho}(f, Lf)\\
= & \|f(0)\|_2^2+ \|\nabla \phi(0)\|_2^2+\int_0^t |e^{\lambda_2 \tau^\rho}g(\tau)|^2_{2,-}+ 2 \lambda_2 \rho \int_0^t \tau^{\rho-1}\left(\|e^{\lambda_2 \tau^\rho}f(\tau)\|_2^2 +
\|e^{\lambda_2 \tau^\rho}\nabla \phi(\tau)\|_2^2 \right)\mathrm{d}\tau\\
& -\int_0^t \int_{\Omega \times \mathbb{R}^3} v\cdot \nabla \phi e^{2\lambda_2 \tau^\rho}|f(\tau)|^2 +2 \int_0^t \int_{\Omega \times \mathbb{R}^3} e^{2\lambda_2 \tau^\rho} f \Gamma (f, f).
\end{aligned}
\end{equation}
We estimate the terms on the right-hand side of \eqref{59} as follows:
\begin{align}
2 \lambda_2 \rho \int_0^t \tau^{\rho-1}\left(\|e^{\lambda_2 \tau^\rho}f(\tau)\|_2^2 +\|e^{\lambda_2 \tau^\rho}\nabla \phi(\tau)\|_2^2 \right)\mathrm{d}\tau
&\leq 2 \lambda_2 t^\rho \left( \sup_{0 \leq s \leq t}\|e^{\lambda_2 s^\rho} w_{\widetilde{\vartheta}}f(s)\|_\infty \right)^2,\label{1st}\\
\int_0^t \int_{\Omega \times \mathbb{R}^3} v\cdot \nabla \phi e^{2\lambda_2 \tau^\rho}|f(\tau)|^2
&\lesssim \int_0^t \int_{\Omega \times \mathbb{R}^3} |v|  w_{\widetilde{\vartheta}}^{-2}(v)|\nabla \phi (\tau)|e^{2\lambda_2 \tau^\rho}\|w_{\widetilde{\vartheta}}f(\tau)\|_\infty^2 \nonumber\\
&\lesssim t \sup_{0\leq s \leq t}\|h(s)\|_\infty \left( \sup_{0 \leq s \leq t}\|e^{\lambda_2 s^\rho} w_{\widetilde{\vartheta}}f(s)\|_\infty \right)^2,\label{2nd}\\
2 \int_0^t \int_{\Omega \times \mathbb{R}^3} e^{2\lambda_2 \tau^\rho} f \Gamma (f, f)
&\lesssim \varepsilon \int_0^t e^{2 \lambda_2 \tau^\rho}\|f\|_{\nu}^2 + C_\varepsilon \int_0^t e^{2 \lambda_2 \tau^\rho}\|\nu^{-\frac{1}{2}}\Gamma (f, f)\|_2^2.\label{3rd}
\end{align}

We shall deduce the $L^2$ coercivity for the case of the soft potential. Through the similar procedure as the ones in \cite{Kim, Esposito}, we have the following crucial estimate (provided $G(s)\leq \|f(s)\|_2^2$):
\begin{equation}
\|\mathbf{P}f(s)\|_\nu^2+ \|\nabla \phi(s)\|_2^2 \lesssim \frac{\mathrm{d}}{\mathrm{d}s}G(s)+ \|(\mathbf{I}-\mathbf{P})f(s)\|_\nu^2+|f(s)|_{2,+}^2 -|g(s)|_{2, -}^2+ \|\Gamma (f, f)\|_2^2.
\end{equation}
Adding $\|(\mathbf{I}-\mathbf{P})f(s)\|_\nu^2$ to the both sides of the above inequalities, we obtain that
\begin{equation}\label{(I-P)f}
\|f(s)\|_\nu^2+ \|\nabla \phi(s)\|_2^2 \lesssim \frac{\mathrm{d}}{\mathrm{d}s}G(s)+ 2\|(\mathbf{I}-\mathbf{P})f(s)\|_\nu^2+|f(s)|_{2,+}^2 -|g(s)|_{2, -}^2+ \|\Gamma (f, f)\|_2^2.
\end{equation}
It is well-known that there exists $\delta_0 >0$ such that
\begin{equation}\label{(Lf,f)}
(Lf, f)\geq \delta_0 \|(\mathbf{I}-\mathbf{P})f(s)\|_\nu^2.
\end{equation}
Multiplying $\delta_0 e^{2 \lambda_2 s^\rho}$ to (\ref{(I-P)f}) and integrating on $[0, t]$ gives
\begin{equation}\label{4th}
\begin{aligned}
&\delta_0 \int_0^t e^{2\lambda_2 s^\rho}\|f(s)\|_\nu^2
+\delta_0 \int_0^t e^{2\lambda_2 s^\rho}\|\nabla \phi(s)\|_2^2 \\
\lesssim & \delta_0 \int_0^t e^{2\lambda_2 s^\rho}G'(s)\mathrm{d}s
+ \int_0^t e^{2\lambda_2 s^\rho}2(Lf, f)+ \delta_0 \int_0^t e^{2\lambda_2 s^\rho}(|f(s)|_{2,+}^2 -|g(s)|_{2, -}^2)+\delta_0 \int_0^t e^{2\lambda_2 s^\rho}\|\nu^{-\frac{1}{2}}\Gamma(f, f)\|_2^2.
\end{aligned}
\end{equation}
On the right-hand side, it holds
\begin{align}\label{5th}
\delta_0 \int_0^t e^{2\lambda_2 s^\rho}G'(s)\mathrm{d}s
\lesssim & \delta_0 \left(e^{2\lambda_2 s^\rho} G(s)|_0^t -\int_0^t 2\lambda_2 \rho s^{\rho-1}
e^{2\lambda_2 s^\rho}G(s)\mathrm{d}s\right)\nonumber\\
\lesssim & \delta_0 \left( \|e^{2\lambda_2 t^\rho}f(t)\|_2^2+\|f(0)\|_2^2 +
\lambda_2 t^\rho \sup_{0 \leq s \leq t}\|e^{\lambda_2 s^\rho} h(s)\|_\infty^2\right),
\end{align}
and, by using Lemma \ref{gamma}, we have
\begin{align}\label{6th}
 \int_0^t e^{2\lambda_2 s^\rho}\|\nu^{-\frac{1}{2}}\Gamma (f, f)\|_2^2
= &  \int_0^t \int_{x,v}e^{2\lambda_2 s^\rho}w_{\widetilde{\vartheta}}^{-2}(v)|\nu^{-\frac{1}{2}}w_{\widetilde{\vartheta}}\Gamma (f,f)|^2 \mathrm{d}x\mathrm{d}v\mathrm{d}s\nonumber\\
\lesssim &  \int_0^t e^{2\lambda_2 s^\rho} \|h(s)\|_\infty^4 \mathrm{d}s\nonumber\\
\lesssim &  t \sup_{0 \leq s \leq t}\|h(s)\|_\infty^2
\sup_{0 \leq s \leq t}\|e^{\lambda_2 s^\rho} h(s)\|_\infty^2.
\end{align}

In summary, combining \eqref{59}--\eqref{3rd} and \eqref{4th}--\eqref{6th} up implies that
\begin{align*}
&(1-\delta_0)\|e^{\lambda_2 t^\rho } f(t)\|_2^2 + \|e^{\lambda_2 t^\rho } \nabla \phi(t)\|_2^2 \\
&+ (\delta_0- \varepsilon)\int_0^t e^{2\lambda_2 s^\rho }\|f(s)\|_\nu^2 +
\delta_0\int_0^t e^{2\lambda_2 s^\rho }\|\nabla \phi(s)\|_2^2
+ (1-\delta_0)\int_0^t |e^{\lambda_2 s^\rho}f(s)|^2_{2,+}
\\
\lesssim & (1+ \delta_0)\|f(0)\|_2^2+ \|\nabla \phi(0)\|_2^2+(1-\delta_0)\int_0^t |e^{\lambda_2 s^\rho}g(s)|^2_{2,-}\\
& + (\lambda_2 t^\rho+ t \sup_{0 \leq s \leq t}\| h(s)\|_\infty + C_\varepsilon t \sup_{0 \leq s \leq t}\| h(s)\|_\infty^2 )\cdot \sup_{0 \leq s \leq t}\|e^{\lambda_2 s^\rho} h(s)\|_\infty^2.
\end{align*}
Note that for $\lambda_2 < \lambda_0$,
\begin{align*}
\int_0^t |e^{\lambda_2 s^\rho}g(s)|^2_{2,-}
= & \int_0^t e^{2\lambda_2 s^\rho}\left\|\frac{w_{\widetilde{\vartheta}} g(s)}{w_{\widetilde{\vartheta}}}\right\|_{\gamma_-}^2\\
\lesssim & \int_0^t e^{2(\lambda_2- \lambda_0) s^\rho} e^{2\lambda_0 s^\rho}\|w_{\widetilde{\vartheta}} g(s)\|_\infty^2\\
\leq & \int_0^t e^{2(\lambda_2- \lambda_0) s^\rho}\mathrm{d}s \cdot
\sup_{0 \leq s \leq t}\|e^{\lambda_0 s^\rho} w_{\widetilde{\vartheta}} g(s)\|_\infty^2\\
\leq & (2(\lambda_0 -\lambda_2 ))^{-\frac{1}{\rho}}\Gamma \left(\frac{1}{\rho}+1\right)
\sup_{0 \leq s \leq t}\|e^{\lambda_0 s^\rho} w_{\widetilde{\vartheta}} g(s)\|_\infty^2\\
\lesssim & \sup_{0 \leq s \leq t}\|e^{\lambda_0 s^\rho} w_{\widetilde{\vartheta}} g(s)\|_\infty^2.
\end{align*}
We choose $\delta_0$ and $\varepsilon$ such that $M < \varepsilon < \delta_0 < \frac{1}{2}$, then it holds for $0 < \lambda_2 \ll 1$ that
\begin{align*}
\|e^{\lambda_2 t^\rho}f(t)\|_2 \lesssim  \|f(0)\|_2+ \|\nabla \phi (0)\|_2
+ \sup_{0 \leq s \leq t}\|e^{\lambda_0 s^\rho} w_{\widetilde{\vartheta}} g(s)\|_\infty
+ \left(\lambda_2^{\frac{1}{2}}t^{\frac{\rho}{2}}
+ t^{\frac{1}{2}}M^{\frac{1}{2}}\right) \sup_{0 \leq s \leq t}\|e^{\lambda_2 s^\rho} h(s)\|_\infty.
\end{align*}
Hence, for $0 < \lambda_m \leq \min\left\{\frac{\lambda}{2}, \lambda_2\right\}$, we have
\begin{align*}
& e^{-\frac{\lambda}{2}(t^\rho-s^\rho)}\|f(s)\|_2 \\
\leq & e^{-\lambda_m (t^\rho-s^\rho)}\|f(s)\|_2 \\
\lesssim  & e^{-\lambda_m t^\rho}\left(\|f(0)\|_2+ \|\nabla \phi (0)\|_2
+ \sup_{0 \leq s \leq t}\|e^{\lambda_0 s^\rho} w_{\widetilde{\vartheta}} g(s)\|_\infty
+ \left(\lambda_m^{\frac{1}{2}}t^{\frac{\rho}{2}}
+ t^{\frac{1}{2}}M^{\frac{1}{2}}\right) \sup_{0 \leq s \leq t}\|e^{\lambda_m s^\rho} h(s)\|_\infty\right).
\end{align*}
Therefore, for $M, \lambda_m \ll_T 1$, it holds
\begin{align*}
&\sup_{0 \leq s \leq t}\left\{e^{-\frac{\lambda}{2}(t^\rho-s^\rho)}\|f(s)\|_2\right\}\\
\lesssim & e^{-\lambda_m t^\rho}\left(\|f(0)\|_2+ \|\nabla \phi (0)\|_2
+ \sup_{0 \leq s \leq t}\|e^{\lambda_0 s^\rho} w_{\widetilde{\vartheta}} g(s)\|_\infty
+ o(1) \sup_{0 \leq s \leq t}\|e^{\lambda_m s^\rho} h(s)\|_\infty\right).
\end{align*}
Plugging it into (\ref{L infty}), we arrive at
\begin{align*}
e^{\lambda_m t^\rho}|h(t, x, v)| \lesssim & \|w_{\widetilde{\vartheta}}f(0)\|_\infty
+ \|f(0)\|_2+ \|\nabla \phi (0)\|_2 + \sup_{0 \leq s \leq t}\left\{e^{\lambda_0 s^\rho}\|w_{\widetilde{\vartheta}}g(s)\|_\infty \right\}\\
&+ o(1)  \sup_{0 \leq s \leq t}\left\{e^{\lambda_m s^\rho}\|h(s)\|_\infty \right\} +  \sup_{0 \leq s \leq t}\left\{e^{\lambda_m s^\rho}\|\nabla \phi(s)\|_\infty \right\}.
\end{align*}
In light of \cite{Kim}, we can control the last term on the right-hand side of the above inequality, with the following interpolation inequality. Actually, for $3< r< p$, it holds
\begin{align*}
\|\nabla \phi\|_\infty \lesssim \|\nabla \phi\|_{C^{0, 1-\frac{3}{r}}(\Omega)} \lesssim \|\nabla \phi\|_{W^{1, r}(\Omega)}.
\end{align*}
Note that
\begin{align*}
&\|\nabla \phi\|_{W^{1, 2}(\Omega)}\lesssim \left\|\int f \sqrt{\mu}\mathrm{d}v\right\|_{L^2(\Omega)}\lesssim e^{-\lambda_m t^\rho}\sup_{0 \leq s \leq t}\left\{e^{\lambda_m s^\rho}\|f(s)\|_{L^2_{x, v}}\right\},\\
&\|\nabla \phi\|_{W^{1, p}(\Omega)}\lesssim \left\|\int f \sqrt{\mu}\mathrm{d}v\right\|_{L^p(\Omega)}\lesssim e^{-\lambda_m t^\rho}\sup_{0 \leq s \leq t}\left\{e^{\lambda_m s^\rho}\|h(s)\|_{L^\infty_{x, v}}\right\}.
\end{align*}
For $0<\theta <1$, $\frac{1}{r}= \frac{\theta}{2}+ \frac{1-\theta}{p}$, we use the standard interpolation:
\begin{align*}
\|\nabla \phi\|_{W^{1, r}(\Omega)}&\lesssim \|\nabla \phi\|_{W^{1, 2}(\Omega)}^\theta \|\nabla \phi\|_{W^{1, p}(\Omega)}^{1-\theta}\\
& \lesssim e^{-\lambda_m t^\rho}\left(\sup_{0 \leq s \leq t}\left\{e^{\lambda_m s^\rho}\|f(s)\|_{L^2_{x, v}}\right\}\right)^\theta \left( \sup_{0 \leq s \leq t}\left\{e^{\lambda_m s^\rho}\|h(s)\|_{L^\infty_{x, v}}\right\}\right)^{1-\theta}.
\end{align*}
By using Young's inequality, we get
\begin{align*}
\sup_{0 \leq s \leq t}\left\{e^{\lambda_m s^\rho}\|\nabla \phi(s)\|_\infty \right\}\lesssim \sup_{0 \leq s \leq t}\left\{e^{\lambda_m s^\rho}\|f(s)\|_{L^2_{x, v}}\right\}+ o(1)\sup_{0 \leq s \leq t}\left\{e^{\lambda_m s^\rho}\|h(s)\|_{L^\infty_{x, v}}\right\}.
\end{align*}
In a result, we obtain
\begin{align*}
e^{\lambda_m t^\rho}|h(t, x, v)| \lesssim & \|w_{\widetilde{\vartheta}}f(0)\|_\infty
+ \|f(0)\|_2+ \|\nabla \phi (0)\|_2 + \sup_{0 \leq s \leq t}\left\{e^{\lambda_0 s^\rho}\|w_{\widetilde{\vartheta}}g(s)\|_\infty \right\}\\
&+ o(1)  \sup_{0 \leq s \leq t}\left\{e^{\lambda_m s^\rho}\|h(s)\|_\infty \right\},
\end{align*}
which implies that
\begin{align*}
\sup_{0 \leq s \leq t}\left\{e^{\lambda_m s^\rho}\|h(s)\|_\infty \right\} \lesssim & \|w_{\widetilde{\vartheta}}f(0)\|_\infty
+ \|f(0)\|_2+ \|\nabla \phi (0)\|_2 + \sup_{0 \leq s \leq t}\left\{e^{\lambda_0 s^\rho}\|w_{\widetilde{\vartheta}}g(s)\|_\infty \right\}\\
&+ o(1)  \sup_{0 \leq s \leq t}\left\{e^{\lambda_m s^\rho}\|h(s)\|_\infty \right\}.
\end{align*}

In conclusion, it is shown that
\begin{align*}
\sup_{0 \leq t \leq T}\left\{e^{\lambda_m t^\rho}\|h(t)\|_\infty \right\} \lesssim & \|w_{\widetilde{\vartheta}}f(0)\|_\infty + \|\nabla \phi (0)\|_2 + \sup_{s\geq 0}\left\{e^{\lambda_0 s^\rho}\|w_{\widetilde{\vartheta}}g(s)\|_\infty \right\}\leq  \mathcal{C}\delta^* M.
\end{align*}
Choosing $\delta^*$ such that $\mathcal{C}\delta^* \ll 1$, we conclude that $\sup_{0 \leq t \leq T}\left\{e^{\lambda_m t^\rho}\|h(t)\|_\infty \right\}$ actually cannot attain $M$. With the bootstrap argument, the solution $f$ can be extended until there is some $T_1>T$ too large to satisfy $\sqrt{M}e^{CT_1^2}\ll 1$. It means that, for any given (large enough) $T>0$, we choose a sufficiently small $M>0$ such that it holds $\sqrt{M}e^{CT^2}\ll 1$, then the local solution can be extended to $T$.
Hence, the proof of Theorem \ref{Thm1} is completed.
\hfill $\square$

\smallskip\medskip

{\bf Acknowledgements}: The authors are very grateful to the nice reviewer for his/her  constructive suggestions.
This work  is supported by NSFC (Grant No. 12071212) and A Project Funded by the Priority Academic Program Development of Jiangsu Higher Education Institutions.


\begin{thebibliography}{99}

\bibitem{Kim} Y. Cao, C. Kim and D. Lee, Global strong solutions of the Vlasov-Poisson-Boltzmann system in bounded domains, \emph{Arch. Ration. Mech. Anal.}, \textbf{233}(3), 1027--1130, 2019

\bibitem{jsp}H.-X. Chen, C. Kim and Q. Li, Local well-posedness of Vlasov-Poisson-Boltzmann equation with generalized diffuse boundary condition, \emph{J. Stat. Phys.}, \textbf{179}(2), 535--631, 2020

\bibitem{Lions}R. DiPerna and P.-L. Lions, On the Cauchy problem for the Boltzmann equation: global existence and weak stability, \emph{Ann. Math.}, \textbf{130}, 312--366, 1989

\bibitem{Ouyang}H.-J. Dong, Y. Guo and Z.-M. Ouyang, The Vlasov-Poisson-Landau system with the specular-reflection boundary condition, arXiv:2010.05314v2

\bibitem{duan2017}R.-J. Duan, F.-M. Huang, Y. Wang and T. Yang, Global well-posedness of the Boltzmann equation with large amplitude initial data, \emph{Arch. Ration. Mech. Anal.}, \textbf{225}(1), 375--424, 2017

\bibitem{zhao2012}R.-J. Duan, T. Yang and H.-J. Zhao, The Vlasov-Poisson-Boltzmann system in the whole space: the hard potential case, \emph{J. Differential Equations}, \textbf{252}(12), 6356--6386, 2012

\bibitem{zhao2013}R.-J. Duan, T. Yang and H.-J. Zhao, The Vlasov-Poisson-Boltzmann system for soft potentials, \emph{Math. Models Methods Appl. Sci.}, \textbf{23}(6), 979--1028, 2013

\bibitem{Esposito}R. Esposito, Y. Guo, C. Kim and R. Marra, Non-isothermal boundary in the Boltzmann theory and Fourier law, \emph{Commun. Math. Phys.}, \textbf{323}(1), 177--239, 2013

\bibitem{Glassey}R. T. Glassey,  \emph{The Cauchy problem in kinetic theory}, Society for Industrial and Applied Mathematics
(SIAM), Philadelphia, PA, 1996.

\bibitem{Guo2001}Y. Guo, The Vlasov-Poisson-Boltzmann system near vacuum, \emph{Commun. Math. Phys.}, \textbf{218}(2), 293--313, 2001

\bibitem{Guo2002}Y. Guo, The Vlasov-Poisson-Boltzmann system near Maxwellians, \emph{Comm. Pure Appl. Math.}, \textbf{55}(9), 1104--1135, 2002

\bibitem{Guo2010} Y. Guo, Decay and continuity of the Boltzmann equation in bounded domains, \emph{Arch. Ration. Mech. Anal.}, \textbf{197}(3), 713--809, 2010

\bibitem{Guo2020}Y. Guo, H. J. Hwang, J. W. Jang and Z.-M. Ouyang, The Landau equation with the specular reflection boundary condition, \emph{Arch. Ration. Mech. Anal.}, \textbf{236}(3), 1389--1454, 2020

\bibitem{Juhi}Y. Guo and J. Jang, Global Hilbert expansion for the Vlasov-Poisson-Boltzmann system, \emph{Commun. Math. Phys.}, \textbf{299}(2), 469--501, 2010

\bibitem{Lee}C. Kim and D. Lee, The Boltzmann equation with specular boundary condition in convex domains, \emph{Comm. Pure Appl. Math.}, \textbf{71}(3), 411--504, 2018


\bibitem{Hwang} J. Kim, Y. Guo and H. J. Hwang, An $L^2$ to $L^\infty$ framework for the Landau equation, \emph{Peking Mathematical Journal}, \textbf{3}, 131--202, 2020

\bibitem{Lieb} E.-H. Lieb, M. Loss, \emph{Analysis}, \emph{volume 14 of Graduate Studies in Mathematics}, 2nd ed. American Mathematical Society, Providence 2001

\bibitem{shuangqian} S.-Q. Liu and X.-F. Yang, The initial boundary value problem for the Boltzmann equation with soft potential, \emph{Arch. Ration. Mech. Anal.}, \textbf{233}(1), 463--541, 2017

\bibitem{Mischler}S. Mischler, On the initial boundary value problem for the Vlasov-Poisson-Boltzmann system, \emph{Commun. Math. Phys.}, \textbf{210}(2), 447--466, 2000

\bibitem{SrainVMB}R. Strain,  The Vlasov-Maxwell-Boltzmann system in the whole space, \emph{Commun. Math. Phys.}, \textbf{268}(2), 543--567, 2006

\bibitem{Strain}R. Strain and Y. Guo, Exponential decay for soft potentials near Maxwellians, \emph{Arch. Ration. Mech. Anal.}, \textbf{187}(2), 287--339, 2008

\bibitem{huijiang2013}Q.-H. Xiao, L.-J. Xiong and H.-J. Zhao, The Vlasov-Poisson-Boltzmann system with angular cutoff for soft potentials, \emph{J. Differential Equations}, \textbf{255}(6), 1196--1232, 2013

\bibitem{huijiang2014}Q.-H. Xiao, L.-J. Xiong and H.-J. Zhao, The Vlasov-Poisson-Boltzmann system for non-cutoff hard potentials, \emph{Sci. China Math.}, \textbf{57}(3), 515--540, 2014

\bibitem{huijiang2017}Q.-H. Xiao, L.-J. Xiong and H.-J. Zhao, The Vlasov-Poisson-Boltzmann system for the whole range of cutoff soft potentials, \emph{J. Funct. Anal.}, \textbf{272}(1), 166--226, 2017

\bibitem{Hongjun}T. Yang, H.-J. Yu and H.-J. Zhao, Cauchy problem for the Vlasov-Poisson-Boltzmann system. \emph{Arch. Ration. Mech. Anal.},  \textbf{182}(3), 415--470, 2006

\bibitem{zhao2006}T. Yang and H.-J. Zhao, Global existence of classical solutions to the Vlasov-Poisson-Boltzmann system, \emph{Commun. Math. Phys.}, \textbf{268}(3), 569--605, 2006
\end{thebibliography}
\end{document}